\documentclass{amsart}
\usepackage[utf8]{inputenc}
\usepackage[mathscr]{euscript}
\usepackage{amssymb}
\usepackage{amsmath}
\usepackage{amscd}
\usepackage{url}
\usepackage{bbm}
\usepackage{ytableau}
\usepackage[colorlinks]{hyperref}
\hypersetup{%
  colorlinks = true,
  linkcolor  = black
}


\newtheorem{prop}{Proposition}[section]

\newtheorem{ques}[prop]{Question}
\newtheorem{lem}[prop]{Lemma}
\newtheorem{thm}[prop]{Theorem}
\newtheorem{cor}[prop]{Corollary}
\theoremstyle{definition}
\newtheorem{defi}[prop]{Definition}
\newtheorem{exo}[prop]{Example}
\newtheorem{remar}[prop]{Remark}

\newcommand{\Aut}{{\mathrm {Aut}}}

\def\id{\mathop{\mathrm{ id}}\nolimits}

\newcommand{\ord}{{\mathrm {ord}}}

\newcommand{\Hom}{{\mathrm {Hom}}}

\newcommand{\Nm}{{\mathrm {Nm}}}

\newcommand{\tr}{{\mathrm {tr}}}

\newcommand{\Sym}{{\mathrm {Sym}}}

\newcommand{\Frob}{{\mathrm {Frob}}}

\newcommand{\Gal}{\mathrm {Gal}}

\newcommand{\diag}{\mathrm{diag}}

\newcommand{\A}{{\mathbb A}}
\newcommand{\CC}{{\mathbb C}}
\newcommand{\C}{{\mathbb C}}
\newcommand{\RR}{{\mathbb R}}
\newcommand{\R}{{\mathbb R}}

\newcommand{\QQ}{{\mathbb Q}}
\newcommand{\Q}{{\mathbb Q}}
\newcommand{\ZZ}{{\mathbb Z}}
\newcommand{\Z}{{\mathbb Z}}

\newcommand{\SSS}{{\mathbb S}}

\newcommand{\OOO}{{\mathcal O}}

\newcommand{\HH}{{\mathfrak H}}

\newcommand{\UU}{{\mathfrak U}}

\newcommand{\pp}{{\mathfrak p}}

\newcommand{\p}{{\mathfrak p}}

\newcommand{\BB}{{\mathcal B}}
\newcommand{\FF}{{\mathbb F}}

\newcommand{\GL}{\mathrm {GL}}

\newcommand{\Sp}{\mathrm {Sp}}
\newcommand{\SO}{\mathrm {SO}}

\newcommand{\GUtwo}[1]{\mathrm {GU}(2,#1)}
\newcommand{\GUone}[1]{\mathrm {GU}(1,1,#1)}
\newcommand{\OO}{\mathrm {O}}
\newcommand{\GSp}{\mathrm {GSp}}
\newcommand{\GSpin}{\mathrm {GSpin}}

\newcommand{\Qbar}{\overline{\mathbb Q}}
\newcommand{\Zbar}{\overline{\mathbb Z}}
\newcommand{\Fbar}{\overline{\mathbb F}}
\newcommand{\rhobar}{\overline{\rho}}

\DeclareMathOperator{\adj}{Adj}
\DeclareMathOperator{\Stab}{Stab}
\DeclareMathOperator{\Rad}{Rad}
\DeclareMathOperator{\HW}{HW}
\DeclareMathOperator{\norm}{{\mathscr N}}
\DeclareMathOperator{\trace}{{\mathrm{Tr}}}

\newcommand{\DimSix}{U}
\newcommand{\DimFive}{V}
\newcommand{\Lattice}{L}

\newcommand{\mxlram}{R_p}
\newcommand{\mxlramtwo}{R_2}
\def\id#1{{\mathfrak{#1}}}      % an ideal
\newcommand{\Level}{\UU^+}

\newcommand{\Cliff}{{\rm Cliff}}

\newcommand{\lmfdbform}[3]{\href{http://www.lmfdb.org/ModularForm/GL2/Q/holomorphic/#1/#2/#3}{{\text{\rm#1.#2.#3}}}}
\newcommand{\lmfdbfform}[4]{\href{http://www.lmfdb.org/ModularForm/GL2/Q/holomorphic/#1/#2/#3/#4}{{\text{\rm#1.#2.#3.#4}}}}
\def\smat#1{\left(\begin{smallmatrix}#1\end{smallmatrix}\right)}

\begin{document}
\title{Quinary forms and paramodular forms}
\author{Neil Dummigan}

\date{December 7th, 2021.}
\address{University of Sheffield\\ School of Mathematics and Statistics\\
Hicks Building\\ Hounsfield Road\\ Sheffield, S3 7RH\\
U.K.}
\email{n.p.dummigan@shef.ac.uk}
\author{Ariel Pacetti}
\address{Center for Research and Development in Mathematics and Applications (CIDMA),
Department of Mathematics, University of Aveiro, 3810-193 Aveiro, Portugal}
\email{apacetti@ua.pt}
\thanks{AP was partially supported by FonCyT BID-PICT 2018-02073 and by
the Portuguese Foundation for Science and Technology (FCT) within
project UIDB/04106/2020 (CIDMA)}

\author{Gustavo Rama}
\address{Facultad de Ingeniería, Universidad de la República, Montevideo, Uruguay}
\email{grama@fing.edu.uy}

\author{Gonzalo Tornar\'ia}
\address{Centro de Matemática, Universidad de la República, Montevideo, Uruguay}
\email{tornaria@cmat.edu.uy}
\thanks{GR and GT were partially supported by CSIC I+D 2020/651}


\begin{abstract}
 We work out the exact relationship between algebraic modular forms for a two-by-two general unitary group over a definite quaternion algebra, and those arising from genera of positive-definite quinary lattices, relating stabilisers of local lattices with specific open compact subgroups, paramodular at split places, and with Atkin-Lehner operators. Combining this with the recent work of R\"osner and Weissauer, proving conjectures of Ibukiyama on Jacquet-Langlands type correspondences (mildly generalised here), provides an effective tool for computing Hecke eigenvalues for Siegel modular forms of degree two and paramodular level. It also enables us to prove examples of congruences of Hecke eigenvalues connecting Siegel modular forms of degrees two and one. These include some of a type conjectured by Harder at level one, supported by computations of Fretwell at higher levels, and a subtly different congruence discovered experimentally by Buzzard and Golyshev.
\end{abstract}
\maketitle

\tableofcontents

\section{Introduction}

Modular forms play a central role in modern mathematics, leading to
the development of very fruitful areas of mathematics (such as
automorphic forms, Galois representations and many applications to
diophantine problems).  An instance of the interaction between modular
forms and geometry is the modularity of rational elliptic curves as
conjectured by Shimura and Taniyama, and proved by Wiles et al (in
\cite{MR1333035} and \cite{MR1839918}). A natural generalisation in
this direction is understanding the relation between analytic objects
and higher dimensional abelian varieties (a particular case of the
Langlands program). In \cite{MR586427} (\S 8, Example 2) Yoshida 
suggested that an abelian surface whose endomorphism ring over $\Q$ equals $\Z$
should be related to a Siegel modular form of degree $2$.

Let $\HH_2$ be the Siegel upper half-plane of degree $2$ consisting of
$2 \times 2$ complex symmetric matrices whose imaginary part is
positive-definite (a natural generalisation of Poincar\'e's upper half
plane). Siegel modular forms are holomorphic functions on $\HH_2$ that
satisfy a transformation property similar to classical modular
forms. More concretely, let $V$ be a finite dimensional $\C$-vector
space and let $\rho:\GL_2(\C) \to \Aut(V)$ be a representation. A Siegel modular form of weight $\rho$ is an holomorphic map $f: \HH_2 \to V$ such that
\[
F((AZ+B)(CZ+D)^{-1})=\rho(CZ+D)(F(Z)),
  \]
  for all $\left(\begin{smallmatrix}A&B\\C&D\end{smallmatrix}\right)$
  in a subgroup of the symplectic group $\Sp_2(\Q)$ (see \cite{MR2409679} for a nice exposition).

  In \cite{BK} (see also \cite{MR3976591}) Brumer and Kramer made the
  following precise conjecture (known as the ``paramodular
  conjecture''): abelian surfaces (with the same endomorphism
  restriction as in Yoshida's remark) should be related to weight $2$
  (i.e. $V = \C$ and $\rho(CZ+D) w = \det(CZ+D)^{2}w$) Siegel modular
  forms, transforming as above under the paramodular group of level $N$ (the
  conductor of the surface) given by
  \[
    P(N):=\left[\begin{array}{cccc} \Z & N\Z & \Z &\Z \\ \Z & \Z & \Z
        & \frac{1}{N}\Z \\ \Z & N\Z & \Z &\Z \\ N\Z & N\Z & N\Z &
   \Z \end{array}\right] \cap \mathrm{Sp}_2(\Q).
 \]
 Some genuine cases of the paramodular conjecture were proven in
 \cite{MR3981316} (see also \cite{1812.09269}). This conjecture
 motivated the study of Siegel paramodular forms, in particular, their
 $L$-series, the theory of newforms (as developed in \cite{RS}) and
 their Galois representations (see \cite{M}).

 A related problem is that of constructing tables of paramodular
 forms. There are nowadays several different algorithms for computing
 classical modular forms. The most well-known are the modular symbol
 approach (as in \cite{MR1628193}), quaternion algebras and Brandt
 matrices (as in \cite{MR579066}), or the use of ternary quadratic
 forms (as in \cite{MR1151865}, \cite{MR2717378}, \cite{rama_msc},
 \cite{MR3553638} and \cite{htv}).
 
 There are some tables of paramodular forms, based on Fourier series expansions, due mostly to Poor, Yuen
 and some coauthors (see \cite{MR3315514}, \cite{MR3713095},
 \cite{MR3739221} and \cite{MR3981316}). A different approach using quinary forms, analogous to Birch's use of ternary quadratic forms, can be used to compute Hecke eigenvalues more easily. This builds on the lattice-neighbour method for algebraic modular forms on orthogonal groups, using an algorithm of Plesken and Souvignier \cite{MR1484483} to test for lattice isometry, as introduced by Greenberg and Voight \cite{GV}. Following earlier computations by Hein \cite{MR3553638} and Ladd \cite{Lad}, this approach was developed in \cite{RT}. One of the main achievements of the present article is to extend their method to more
 general values of $N$ (not just square-free ones) and weights. Conjecture 15 of \cite{RT} is a special case of results proved here. Note that what we call ``$N$'' here will generally be ``$D$'' later in the paper.
 
 Our result is in the spirit of Eichler's basis problem (as in
 \cite{MR0485698}). Eichler's statement of the basis problem is the
 following: ``the basis problem, is to give bases of linearly
 independent forms of these spaces which are arithmetically
 distinguished and whose Fourier series are known or easy to
 obtain''. Eichler's solution, given a positive integer $N$
 (under the assumption that $N$ is square-free, which was later relaxed by
 Hijikata in \cite{MR337783}), takes a prime $p$ dividing it. Then the
 space of quaternionic modular forms for the quaternion algebra
 ramified at $\{p,\infty\}$, of level given by an Eichler order (of
 level $N$) provides a solution to the basis problem. Furthermore,
 such a space can be computed easily (as do Fourier expansions,
 corresponding to theta functions of positive-definite quadratic forms
 in four variables).

 The main idea of Birch was to relate the arithmetic of quaternion
 algebras to ternary quadratic forms (instead of quaternary ones),
 making computations more efficient. In the present article, we
 present a partial solution to the basis problem for paramodular
 forms. The word partial refers to two main obstacles of our
 method. The first one is related to the possible weights we can
 compute. Unfortunately, paramodular forms of weight $2$ (related to
 abelian surfaces) are not cohomological (as happens for classical
 weight $1$ modular forms), hence they cannot be computed with our
 approach, whereas all weights with scalar part $3$ or more can.
 The second issue has to do with a big difference between classical
 and Siegel modular forms (of degree greater than one). For
 classical modular forms, a newform satisfies that its Fourier
 expansion is trivially determined by the eigenvalues of Hecke operators. This is no longer the case for Siegel modular forms. Our
 approach only allows to compute a basis for the space of algebraic
 modular forms (for the orthogonal group of a positive-definite
 quinary quadratic form) and to compute Hecke operators acting on them
 (as in \cite{RT}).

 More concretely, let $N$ be a positive integer,
 and assume that there exists a prime $p$ such that $p \mid N$ but
 $p^2 \nmid N$. In Section~\ref{section:special} we prove that there
 is a (unique up to semi-equivalence) quinary positive-definite
 integral quadratic form $Q$, of determinant $2N$, with the following
 properties:
 \begin{itemize}
 \item The Hasse-Witt invariant of $Q$ is $-1$ at $p$ and $\infty$,
     and $+1$ at all other primes.
 \item The quadratic form $Q$ is special, with Eichler invariant
     $e(Q_q) = +1$ for all primes different from $p$ (see \S
     \ref{section:special}).
 \end{itemize}
 Then, neglecting Yoshida lifts (cf. Proposition \ref{prop:yoshida}) and any Saito-Kurokawa lifts (cf. Proposition \ref{prop:saito-kurokawa}), the space of algebraic modular forms for the orthogonal group of $Q$, with values in a certain representation $W_{j+k-3, k-3}$, is isomorphic (as a
 Hecke module) to the space of $p$-new paramodular forms of weight $\det^k\otimes\Sym^j$, with $k\geq 3$. This is Theorem~\ref{thm:Rainer1}, with $D^-=p$, $D^+=N/p$.
 
Let $B$ denote a definite quaternion algebra over $\Q$.  The proof of
our result exploits the relation between the algebraic group $\GSp_2$
and its compact twist $\GUtwo{B}$. In a series of articles, Ibukiyama
and some coauthors (see \cite{MR3638279} and also \cite{I,I2}) stated
conjectures relating automorphic forms on $\GUtwo{B}$ and $\GSp_2$, in
the case of square-free levels (see the articles
\cite{MR3194136,MR2555703} on computations of automorphic forms on
$\GUtwo{B}$).  The conjectures were proven by R\"osner and Weissauer
in a recent article \cite{RW}, using the trace formula. A somewhat
less general result was obtained independently by van Hoften \cite{vH}
using very different, algebro-geometric tools. Although in \cite{RW}
the result is proven for groups whose level involves only primes
ramified in the quaternion algebra $B$, we extend their result to our
more general setting (see Theorem~\ref{thm:Rainer1}).

 A main contribution of the present article is to relate algebraic
 modular forms for $\GUtwo{B}$ with those for $\SO(Q)$ for a suitable
 integral quadratic form $Q$ (see Theorem~\ref{GU2SO5iso}). A partial
 result in this direction was obtained by Ladd (\cite{Lad}) in his
 doctoral thesis in the case $N=p$, though our approach was influenced more by a paper of Ibukiyama \cite{I}. Our strategy is to construct a six-dimensional space $\DimSix$ in $M_2(B)$ invariant under conjugation
 by $\GUtwo{B}$, and a quadratic form on it that is invariant under the
 action of $\GUtwo{B}$, and also under translation by scalar matrices. This induces a quadratic form on the five-dimensional quotient
 $\DimFive := \DimSix/\Q I$. We construct a rank $6$ lattice inside
 $\DimSix$ (defined in \S \ref{section:lattice}) and consider its
 quotient by $\ZZ I$. The dual of this produces the rank $5$ integral lattice $\Lattice$.
 One is left to relate the compact
 level in $\GUtwo{B}$ studied by Ibukiyama with the stabiliser of
 $\Lattice$ (up to the centre of $\GUtwo{B}$). These two groups are not
 exactly the same, as the Atkin-Lehner operator on $\GUtwo{B}$
 stabilises the lattice $\Lattice$. To get the right subgroup, we
 define a character on the stabiliser of
 $\Lattice_p:=\Lattice \otimes \Z_p$ for each prime $p$ (see
 Definition~\ref{defi:character}), whose kernel does match the open
 compact subgroup of $\GUtwo{B_p}$ corresponding to a paramodular
 form. This allows us to transfer automorphic forms from one algebraic
 group to the other one.

 It is important to mention that we can prove not only that the
 correspondence preserves Hecke operators, but also a
 precise relation between the action of the Atkin-Lehner operators
 (see Theorem~\ref{ALsignchange}). In particular, for genuine forms
 (those that cannot be constructed from forms for $\GL_2$), the
 Atkin-Lehner sign changes sign for the ramified primes, while it
 stays the same for the non-ramified ones (as happens with the
 classical Jacquet-Langlands correspondence between $\GL_2$ and $B$).
 
 The above is directly applicable to the efficient computation of Hecke eigenvalues for Siegel modular forms of degree two and paramodular level, at least if the vector part of the weight is small, and the scalar part at least $3$. But actually looking at the eigenvectors within spaces of algebraic modular forms also allows us to prove various instances of congruences of Hecke eigenvalues. This is the subject of \S \ref{section:congruences}. We warm up by re-proving a congruence originally obtained by Poor and Yuen \cite[\S 8, Example 1]{MR3315514}. This is of the form
 $$\lambda_F(p)\equiv a_p(g)+p+p^2\pmod{\lambda}.$$
On the left, $F$ is a cuspidal Hecke eigenform of degree $2$, weight $3$ and paramodular level $61$, and $\lambda_F(p)$ its eigenvalue for $T(p)$, with $p$ any prime number different from $61$. On the right, $g$ is a newform of degree $1$, of weight $4$ for $\Gamma_0(61)$, with Hecke eigenvalues $a_p(g)$ in a field of degree $6$, in which the modulus $\lambda$ is a divisor of the rational prime $43$.
The right-hand-side can be interpreted as the eigenvalue of $T(p)$ on the Saito-Kurokawa lift $\mathrm{SK}(g)$ of $g$. Both $F$ and $\mathrm{SK}(g)$ have corresponding eigenforms inside a space of algebraic modular forms arising from a certain genus of quinary lattices of determinant $2\cdot 61$. We can prove the congruence of Hecke eigenvalues by observing that these eigenvectors are the same modulo $\lambda$. The modulus comes from the algebraic part of the critical $L$-value $L(3,g)$.

Such congruences can be extended to $F$ of weight $(k,j)$ with $k\geq 3$ and even $j>0$, with $g$ of weight $j+2k-2$ and $\lambda\mid \ell$ coming from $L(j+k, g)$. For $F$ and $g$ of level $1$ this is a conjecture of Harder \cite{MR2409680}. Computational evidence for some examples of levels $2,3,5,7$ was obtained by Fretwell \cite{MR3803970}. Congruences involving Saito-Kurokawa lifts are a degenerate case $j=0$, but for $j>0$ the problem is that the right hand side of the congruence, $a_p(g)+p^{k-2}+p^{j+k-1}$, is not the Hecke eigenvalue of $T(p)$ on any Siegel modular form. We address this by observing that $p^{k-2}+p^{j+k-1}=p^{k-2}(1+p^{j+1})$, that $(1+p^{j+1})$ is a Hecke eigenvalue for an Eisenstein series of level $1$ and weight $j+2$, and that modulo $\lambda$ this can be replaced by a cuspidal eigenform of level $q$ and weight $j+2$, where $q$ is an auxiliary prime such that $q^{j+2}\equiv 1\pmod{\ell}$. Thus, modulo $\lambda$, $a_p(g)+p^{k-2}+p^{j+k-1}$ becomes the eigenvalue of $T(p)$ on some Yoshida lift, which does not exist as a holomorphic Siegel modular form of paramodular level, but does exist in one of our spaces of algebraic modular forms, allowing us to proceed almost as before to prove several congruences of this type. Actually the target $F$ is represented by an eigenvector in a different space of algebraic modular forms, coming from a different genus of quinary lattices with the same determinant. But it is linked to the Yoshida lift by their mutual congruence with a third form, of paramodular level $qN$, represented by eigenvectors in both spaces.

The same idea using Yoshida lifts allows us to prove a congruence discovered experimentally by Buzzard and Golyshev (see Theorem~\ref{thm:buzzard}). This involves the same $F$ as in the example of Poor and Yuen, but the right hand side is now $1+p^3+pa_p(g)$, where $g$ is now weight $2$ and level $61$, with Hecke eigenvalues in a cubic field, and $\lambda$ is a divisor of $19$ in this field. We use an auxiliary weight $4$ form of level $37$. This congruence, and its proof, is more subtle in two ways. First, the modulus is not observable in a critical value of $L(s,g)$. Second, in the proof we see two eigenvectors that are not the same modulo $\lambda$, but they are forced nonetheless to lie in the same mod-$\lambda$ Hecke eigenspace, thanks to the intervention of a newform of level $61\cdot 37$, with Hecke eigenvalues congruent to both.

\subsection*{Acknowledgements}
This project has its roots in a visit of Jeffery Hein and Watson Ladd
to Uruguay in 2014 for a small research workshop on the subject of
quinary orthogonal modular forms. The workshop topic was suggested by
John Voight to whom we are grateful for many conversations regarding
orthogonal modular forms. A main motivation for the present article
was to prove the conjectures stated in \cite{RT}.

The project benefitted from communications with V. Golyshev (on
congruences) and R.  Weissauer (on the relation between
Siegel modular forms and automorphic forms for $\GUtwo{B}$); indeed
another main motivation for this paper was to prove the mod $19$
congruence brought to our attention by Golyshev.  The first and
fourth-named authors met during the workshop ``Picard-Fuchs Equations
and Hypergeometric Motives'' at the Hausdorff Research Institute for
Mathematics, Bonn, in March 2018, and are also grateful for the
hospitality of the Max Planck Institute for Mathematics, Bonn, during
a short visit in April 2019.
 
\section{The general spin group}
\label{section:spingroup}

Let $k$ be a field of characteristic different from $2$ and let
$(V,Q)$ be a quadratic space over $k$. Let $\Cliff(V)$ be the Clifford
algebra attached to $(V,Q)$. Recall that the Clifford algebra
$\Cliff(V)$ has a natural $\Z/2$-graduation, so let $\Cliff_0(V)$ denote
its even part. The subspace $\Cliff_0(V)$ is a subalgebra of $\Cliff(V)$ which
is central if $(V,Q)$ is regular (i.e. non-degenerate) of odd dimension, so from now on we
assume this is the case.

The Clifford algebra $\Cliff(V)$ has two natural anti-involutions
$*: \Cliff(V) \to \Cliff(V)$ (see \cite{2008.12847}) which agree on
the even part $\Cliff_0(V)$ so we do not need to make any particular
choice.
\begin{defi}
  The General Spin group $\GSpin(V)$ is the subgroup of
  $\Cliff_0(V)^\times$ given by
  \[
    \GSpin(V) = \{g \in \Cliff_0(V) \; : \; g^*  g \in
    k^\times \text{ and }g^{-1} V g = V\}.
  \]
  The spinor norm $\nu:\GSpin(V)\to k^\times$ is given by $\nu(g)=g^*g$.
\end{defi}

There is a natural homomorphism
\begin{equation}
  \label{eq:mapphi}
\phi : \GSpin(V) \to \OO(V)  
\end{equation}
given by $\phi(g)(v) = g v g^{-1}$
(see \cite[Chapter 10, Lemma 3.1]{MR522835}).
\begin{thm}
  The image of $\phi$ equals $\SO(V)$ and its kernel equals
  $k^\times$.
\end{thm}
\begin{proof}
See \cite[Chapter 10, Theorem 3.1]{MR522835}.
\end{proof}

When $V$ has odd dimension, the natural copy of $V$ in $\Cliff(V)$
lies in the odd part. However, in such a case, the centre of
$\Cliff(V)$ has a one dimensional odd part. For example, if
$\{e_1,\ldots,e_n\}$ is an orthogonal basis,
then the vector $c = e_1\cdots e_n$ generates such subspace and
$c^2\in k$. In
particular, the subspace $c \cdot V$ does lie in $\Cliff_0(V)$ and
furthermore, the elements of $\Cliff_0(V)$ normalising $V$ are the
same as the ones normalising $c \cdot V$.
Note also that the involution acts on $c\cdot V$ by
$(-1)^{\lfloor n/2\rfloor}$.

\subsection{The $5$-dimensional case}
From now on we restrict to quadratic forms over a number field $k$.
Let $(V,Q)$ be a regular quinary quadratic space over $k$. 

%Recall that the subspace $c \cdot V$ lies in the even Clifford algebra $\Cliff_0(V)$.
\begin{lem}
  The subspace $\DimSix = \{v \in \Cliff_0(V) \; : \; v^* = v\}$
  equals the subspace $\Q \oplus c\cdot V$.
\label{lemma:clifordinvariant}
\end{lem}
\begin{proof}
  Clearly $k$ is invariant under the involution. Let
  $\{e_1,\ldots,e_5\}$ be an orthogonal basis for $V$, so
  $\Cliff_0(V) = k \oplus c\cdot V \oplus \langle e_ie_j \;:\;
  1\le i<j \le 5\rangle$.
  The involution sends $e_ie_j$ to $e_je_i=-e_ie_j$ and fixes $c\cdot
  V$, hence $U=k\oplus c\cdot V$.
\end{proof}

%{\color{red}
%I don't understand this remark. There's no quadratic form on $c\cdot
%V$.
%\begin{remar}
%  The elements
%  $\{e_2e_3e_4e_5, e_1e_3e_4e_5, e_1e_2e_4e_5, e_1e_2e_3e_5,
%  e_1e_2e_3e_4\}$ form a basis for the vector space $c \cdot V$. Note
%  in particular that the quadratic space $(c \cdot V,Q)$ is dual to
%  $(V,Q)$ hence $\SO(V) \simeq \SO(c\cdot V)$.
%  \label{rem:duallattice}
%\end{remar}
%}
%
In the five dimensional case, the
condition $g^{-1} V g = V$ on the definition of the General Spin group
is superfluous provided $g^* g \in k^\times$.

\begin{lem}
  If $\dim(V)=5$, then
  \[
    \GSpin(V) = \{g \in \Cliff_0(V) \; : \; g^* g \in k^\times\}.
    \]
\end{lem}

\begin{proof}
    %See \cite[\S5.5]{Eichler} or \cite[Lemma 4]{Lad}.
    We recall the proof from \cite[\S5.5]{Eichler}.
    Let $g\in\Cliff_0(V)$ such that $g^* g\in k^\times$ and let
    $v\in V$.
    Then $g^* (cv) g$ is fixed by the involution so by
    Lemma~\ref{lemma:clifordinvariant} we have
    $g^* (cv) g = \alpha + cw$ with $\alpha\in k$ and $w\in V$.
    Squaring this equality gives
    $\nu(g) c^2 Q(v) = \alpha^2 + c^2 Q(w) + 2\alpha cw \in k$.
    If $v\neq 0$ then clearly $w\neq 0$ so $\alpha=0$ and
    $g^* (cv) g = cw$. Since $c$ is central, we conclude
    $g^{-1} v g = \frac1{\nu(g)}w\in V$.
\end{proof}

For each place $v$ of $k$ the Hasse-Witt invariant
$\mathrm{HW}_v(Q)\in\{\pm 1\}$ is an invariant of the
quadratic space $V_v$ given by the class of
$\Cliff_0(V_v)$ in the Brauer group (see \cite[(3.12) in p.117]{MR2104929}).
If $\{e_1,\dotsc,e_5\}$ is an orthogonal basis with $Q(e_i)=a_i$
then
\begin{equation}
  \label{eq:HWdefi}
  \mathrm{HW}_v(Q)=(-1,-1)_v\prod_{i<j}(a_i,a_j)_v,
\end{equation}
where the
quadratic Hilbert symbol $(a,b)_v$ is $+1$ or $-1$, according as
$ax^2+by^2=z^2$ has, or has not (respectively), a solution
$(x,y,z)\neq (0,0,0)$ in $k_v^3$ (see \cite{MR2104929} Proposition
3.20).

\begin{remar}
  The definition of the Hasse-Witt invariant for quinary forms coincides
  with the classical Hasse invariant for odd primes,
  but it differs by $(-1,-1)_v$ for even primes and real
  places.
  %(see \cite[Proposition V.3.20(3)]{MR2104929}).
\end{remar}
Let
\begin{equation}
  \label{eq:ramifiedset}
  S = \{ v \; : \; \HW_v(Q) = -1\},
\end{equation}
the set of places where the quadratic form $Q$ has Hasse-Witt invariant
$-1$. By Hilbert's reciprocity $S$ has even cardinality.
\begin{remar}
  For computational purposes, we assume $k$ is totally real and
  the quinary quadratic form $Q$ is totally positive definite,
  so all the archimedean places are in $S$.
\end{remar}
\begin{remar}\label{M2BCliff}
Let $B$ be the quaternion algebra over $k$ ramified precisely at the places of
$S$ and denote $b\mapsto\overline{b}$ its standard involution.
By definition, the central simple algebras $\Cliff_0(V)$ and $B$
correspond to the same class in the Brauer group. It follows that
$\Cliff_0(V)\simeq M_2(B)$ since $\dim\Cliff_0(V)=4\dim B$.
\end{remar}
% Note that $M_2(B)$ has an involution given by
% $m^* = \overline{m}^{t}$, where conjugation denotes the standard
% involution on $B$ acting on $M_2(B)$ via its action on each
% entry.

\begin{lem}
    The isomorphism $\Cliff_0(V)\simeq M_2(B)$
    can be chosen so that the involution of
$\Cliff_0(V)$ corresponds to the involution of $M_2(B)$
given by $m^*=\overline{m}^t$.
\end{lem}

\begin{proof}
    Let $D=\det V$ and choose $\alpha$ and $\beta$ such that
    $(-\alpha D,-\beta D)_v=\HW_v(Q)$.
    Without loss of generality we can assume
    $V$ has an orthogonal basis $\{e_1,e_2,e_3,e_4,e_5\}$
    with $Q(e_1)=\alpha$, $Q(e_2)=\beta$, $Q(e_3)=\alpha\beta D$
    and $Q(e_4)=Q(e_5)=D$.
    \par
    Consider the representation of $\Cliff_0(V)$ as a tensor
    product of quaternion algebras given in \cite[(5.18)]{Eichler}:
    \[
        \Cliff_0(V) \simeq
        [1, e_1e_2, e_2e_3, e_3e_1]
        \otimes
        [1, e_1e_2e_3e_4, e_1e_2e_3e_5, e_4e_5] \,.
    \]
    By our choice of $\alpha$ and $\beta$ we have
    $[1, e_1e_2, e_2e_3, e_3e_1]\simeq B$
    with the involution on the left side corresponding to the
    standard involution of $B$.
    \par
    On the other hand we have
    $[1, e_1e_2e_3e_4, e_1e_2e_3e_5, e_4e_5]\simeq M_2(k)$
    as follows:
    \[
    e_1e_2e_3e_4\mapsto\smat{\alpha\beta D&0\\0&-\alpha\beta D},
    \qquad
    e_1e_2e_3e_5\mapsto\smat{0 & \alpha\beta D\\\alpha\beta D&0},
    \qquad
    e_4e_5\mapsto\smat{0 & D\\-D&0},
    \]
    and the involution on the left side corresponds to the transpose
    in $M_2(k)$.
    Thus $\Cliff_0(V)\simeq B\otimes M_2(k)\simeq M_2(B)$
    and the involution of $\Cliff_0(V)$ corresponds to the
    involution on $M_2(B)$ given by $m^*=\overline{m}^t$.
\end{proof}

By the lemma we can (and will) identify $\Cliff_0(V)$ with $M_2(B)$ with the
involution given by $m^*=\overline{m}^t$.
The group $\GSpin(V)$ is then isomorphic to the group
\begin{equation}
  \label{eq:G}
  \GUtwo{B}:=\{g\in M_2(B):\,g^* g=\nu(g) I,\; \nu(g)\in k^{\times}\}.
\end{equation}
%  G:=\{ g \in M_2(B) \; : \; g \cdot g^* = \nu(g) \in k^\times\},
The group $\GUtwo{B}$ consists of the invertible elements in $M_2(B)$
preserving the hermitian form
$\langle(x,y),(r,s)\rangle = \overline{x}{r} + \overline{y}{s}$ on
$B^2$ (via left multiplication) up to scale.

\section{Quaternionic unitary groups and some local subgroups}
Keep the notation of the previous section.  Let $B$ be a definite
quaternion algebra over $\Q$ ramified at a finite set of primes $S$
(containing the infinity place), with main involution
$\alpha\mapsto\overline{\alpha}$ and denote
$\GUtwo{B}:=\{g\in M_2(B):\,g^* g=\nu(g) I, \;\nu(g)\in \Q^{\times}\}$.
%Given $\gamma\in M_2(B)$, define $\gamma^*:={}^t\overline{\gamma}$ (with the transpose as a $2$-by-$2$ matrix), and
%$$\GUtwo{B}:=\{\gamma\in M_2(B):\,\gamma^*\gamma=\nu I, \nu\in \Q^{\times}\}.$$

\subsection{Local subgroups of $\GUtwo{B}$} For $v$ a rational place,
let $B_v := B \otimes \Q_v$ denote the completion of $B$ at $v$. We
define for each place $v$ an open compact subgroup
$U_v \subset \GUtwo{B_v}$ as follows.

\subsubsection{Archimedean place}
Our assumption that $B$ is definite implies that $\GUtwo{B_\infty}$ is
the compact group $\Sp(2)$ of rank $2$ (the compact form of
$\Sp(2,\RR)$). Hence we take $K_\infty:=\GUtwo{B_\infty}$ as our compact open subgroup.

\subsubsection{Non-archimedean places in S} \label{section:levelatS}
Let $\mxlram$ denote the unique maximal order of $B_p$ \cite[Chapitre II, Lemme 1.5]{MR580949}, and $\id{p}$
its maximal ideal (given by the elements of norm divisible by $p$). Let $R_p^0:=\{r\in R_p:\, r+\overline{r}=0\}$.
Following \cite{MR3638279}, let $\xi \in M_2(B_p)$ be such that
$\xi^*\, \xi = \left(\begin{smallmatrix} 0 & 1\\ 1 &
    0\end{smallmatrix} \right)$ and consider the $\ZZ_p$-lattice
% \footnote{Sea $\norm\varepsilon=-1$, $d=1/\trace\varepsilon$,
%     entonces $\xi=\smat{1&\varepsilon\\d\varepsilon&d}$,
%     $\xi^*=\smat{1&\overline{\varepsilon}d\\\overline{\varepsilon}&d}$,
%     $\xi^*\,\xi=\smat{0&1\\1&0}$.
% }
\begin{equation}
  \label{eq:ramifiedleveldefi}
L_p:=\xi\cdot\begin{pmatrix}\id{p}\\ R_p\end{pmatrix} \subset B_p^2
\,.
\end{equation}
Let $K_p^-$ be the subgroup of $\GUtwo{B_p}$ given by the stabiliser
of $L_p$ (column vectors) under the natural left action of
$\GUtwo{B_p}$ on $B_p^2$.

Let $H:=\left(\begin{smallmatrix} 0 & 1\\1 & 0\end{smallmatrix}\right)$,
and denote by $\GUone{B_p}$ the group
\[
\GUone{B_p}:=\{g\in M_2(B_p):\, g^*Hg=\nu(g) H,\,\,\nu(g)\in\Q_p^{\times}\}.
\]
Conjugation by $\xi$ gives an isomorphism between $\GUtwo{B_p}$ and
$\GUone{B_p}$. Under this isomorphism, the group $K_p^-$ maps to the
stabiliser $K^-(p)$ of the lattice $\id{p} \oplus \mxlram$ (see \cite[\S 2]{I}
and \cite[\S 4.3]{vH}), the intersection of $\GUone{B_p}$ with the order
$\begin{pmatrix}\mxlram & \id{p} \\
    \id{p}^{-1} & \mxlram\\
  \end{pmatrix}\,\text{of $M_2(B_p)$}$. 
      
  Note that the subgroup $K^-(p)$ of $\GUone{B_p}$ is normalised by
  an Atkin-Lehner element
  $\omega'_p:=\begin{pmatrix}0&p\\1&0\end{pmatrix}\in\GUone{B_p}$,
  which satisfies $(\omega'_p)^2=pI$, and
  $(\omega'_p)^*H\omega'_p=pH$. So we define
  $\omega_p:=\xi\omega_p'\xi^{-1}\in \GUtwo{B_p}$, which normalises
  $K_p^-$, with $\omega_p^2=pI$ and $\omega_p^*\omega_p=pI$, so
  $\nu(\omega_p)=p$.
  
\subsubsection{Non-archimedean places not in S} \label{section:levelnotatS}
At any prime number $p \not \in S$ we fix an isomorphism
$B_p\simeq M_2(\Q_p)$, then let $R_p$ denote the matrix order $M_2(\ZZ_p)$. For $n$ a non-negative
integer, consider the $\ZZ_p$-lattice
%Following \cite{MR3638279}, for each finite prime $p$,
%define the following lattices in $B_p \times B_p$:
%  \begin{enumerate}
%  \item If $p \not \in \tilde{S}$, let $L_p := \Z_{B_p} \times \Z_{B_p}$.
%    
%  \item If $p \in S$, let $\xi \in M_2(B)$ be such that
%    $\xi \cdot \xi^* = \left(\begin{smallmatrix} 0 & 1\\ 1 &
%        0\end{smallmatrix} \right)$ and let $L_p:=(\id{p},\Z_{B_p})\cdot \xi$, where $\id{p}$ is the maximal ideal of $\Z_{B_p}$.
%  \end{enumerate}
%
%  For each prime $p$, let $U_p$ be the subset of $M_2(B)$ preserving
%  the lattice $L_p$ (via right multiplication) intersected with $G_p$,
%  and let $\widehat{U}_f$ be the restricted product of the groups
%  $U_p$. 
\begin{equation}
  \label{eq:extralevel}
L_{p^n}:= \begin{pmatrix}M_2(\Z_p)\\ \pi^n \cdot M_2(\ZZ_p)\end{pmatrix}
\subset B_p^2
\,,
\end{equation}
where
$\pi = \left(\begin{smallmatrix} 1 & 0\\ 0 &
    p\end{smallmatrix}\right)$. Similar to the definition of $K_p^-$ above, we
define $K_{p^n}^+$ as the subgroup of $G(\Q_p)$ of elements preserving
the lattice $L_{p^n}$ under left multiplication.

The main involution is given by
$\begin{pmatrix}a&b\\c&d\end{pmatrix}\mapsto\begin{pmatrix}d&-b\\-c&a\end{pmatrix}$.
Consider the isomorphism
$\Psi:M_2(B_p) \stackrel\sim\longrightarrow M_4(\Q_p)$ given by
\begin{equation}
  \label{eq:isom}
  \Psi
  \smat{\smat{a_{11}&a_{12}\\a_{21}&a_{22}} &
    \smat{b_{11}&b_{12}\\b_{21}&b_{22}} \\
    \smat{c_{11}&c_{12}\\c_{21}&c_{22}} &
    \smat{d_{11}&d_{12}\\d_{21}&d_{22}}}
  =
  \left(\begin{smallmatrix}
      a_{11} & b_{11} & a_{12} &b_{12} \\
      c_{11} & d_{11} & c_{12} & d_{12} \\
      a_{21} & b_{21} & a_{22} & b_{22}\\
      c_{21} & d_{21} & c_{22} & d_{22}
    \end{smallmatrix} \right).
%  \left(\begin{smallmatrix} a_{11} & a_{12} & b_{11} &b_{12} \\
%      a_{21} & a_{22} & b_{21} & b_{22}\\
%      c_{11} & c_{12} & d_{11} & d_{12} \\
%      c_{21} & c_{22} & d_{21} & d_{22}
%    \end{smallmatrix} \right).
    \end{equation}
    Note that $\Psi$ swaps second and third rows, then also second and third columns.
    Let
    $\GSp_2(\Q_p)=\{g\in M_4(\Q_p):\,g^tJg=\nu
    J,\,\nu\in\Q^{\times}\}$, where
    \[
        J=\smat{0&0&1&0\\0&0&0&1\\-1&0&0&0\\0&-1&0&0}
%     J=\left(\begin{smallmatrix}0&-1&0&0\\1&0&0&0\\0&0&0&-1\\0&0&1&0\end{smallmatrix}\right).
      \]
\begin{lem}\label{symplectic}
  The isomorphism $\Psi$ induces an isomorphism between the groups
  $\GUtwo{B_p}$ and $\GSp_2(\Q_p)$.
\end{lem}
\begin{proof}
 % Since
 % $\left(\begin{smallmatrix}
 %     0&1\\
 %     -1&0\end{smallmatrix}\right)
 % \left(\begin{smallmatrix}
 %     a&b\\c&d
 %   \end{smallmatrix}\right)^t
 % \left(
 %   \begin{smallmatrix}
 %     0&-1\\1&0\end{smallmatrix}\right)
 % =\left(\begin{smallmatrix}
 %     d&-b\\-c&a
 %   \end{smallmatrix}\right),$
 If we let $J'=\begin{pmatrix} 0&1&0&0\\-1&0&0&0\\0&0&0&1\\0&0&-1&0\end{pmatrix}$, then $\Psi(J')=J$ and ${J'}^{-1}=-J'$.
 Since $\begin{pmatrix}0&-1\\1&0\end{pmatrix}\,{}^t\begin{pmatrix}a&b\\c&d\end{pmatrix}\begin{pmatrix}0&1\\-1&0\end{pmatrix}=\begin{pmatrix}d&-b\\-c&a\end{pmatrix},$ we see that $B^*={J'}^{-1}\,{}^tBJ'$, where the transpose is as a $4$-by-$4$ matrix. Hence for $g\in M_2(B_p)$, 
$$g^*g=\nu I\iff {J'}^{-1}\,{}^tg J'g=\nu I\iff {}^tg J'g=\nu J'\iff {}^t\Psi(g)J\Psi(g)=\nu J.$$
\end{proof}
%
%\begin{lem}\label{symplectic}
%For any prime $p$ at which $B$ splits, 
%$$\GU(2,B\otimes\Q_p)\simeq \GSp_2(\Q_p)=\{g\in M_4(\Q_p):\,{}^tgJg=\nu J,\,\nu\in\Q^{\times}\},$$
%with $J:=\left(\begin{smallmatrix}0&-1&0&0\\1&0&0&0\\0&0&0&-1\\0&0&1&0\end{smallmatrix}\right)$.
%\end{lem}
%\begin{proof}
%
%  Since $\begin{pmatrix}0&1\\-1&0\end{pmatrix}{}^t\begin{pmatrix}a&b\\c&d\end{pmatrix}\begin{pmatrix}0&-1\\1&0\end{pmatrix}=\begin{pmatrix}d&-b\\-c&a\end{pmatrix},$ we see that $B^*=J^{-1}\,{}^tBJ$, for $\gamma\in M_2(B\otimes\Q_p)\simeq M_4(\Q_p)$, where the transpose is as a $4$-by-$4$ matrix. Hence
%$$\gamma^*\gamma=\nu I\iff J^{-1}\,{}^t\gamma J\gamma=\nu I\iff {}^t\gamma J\gamma=\nu J.$$
%\end{proof}
%
The paramodular group of level $p^n$ is given by
$K(p^n):=\{k\in \GSp_2(\Q_p):\,h^{-n}kh^n\in \GL_4(\Z_p)\}$, where
$h:=\diag(1,1,1,p)$.

\begin{lem}
  The group $K_{p^n}^+$ maps onto the paramodular
  group $K(p^n)$ under the isomorphism (\ref{eq:isom}).
\end{lem}
\begin{proof}
%  Recall that the paramodular group $K(p^n)$ can be represented
  %as the matrices in $M_4(\Z_p)$ preserving the symplectic form
  %$\left(\begin{smallmatrix} 0 & 0 & 1 & 0\\ 0 & 0 & 0 & p^n\\ -1 & 0 &
  %    0 & 0\\ 0 & -p^n & 0 & 0\end{smallmatrix}\right)$ or 
%  as the
%  matrices in $\GSp_4(\Q_p)$ belonging to the lattice
  The isomorphism (\ref{eq:isom}) sends the lattice
  $L=\left(\begin{smallmatrix} \Z_p & \Z_p & \Z_p & \Z_p/p^n\\ \Z_p &
      \Z_p & \Z_p & \Z_p/p^n\\ \Z_p & \Z_p & \Z_p & \Z_p/p^n\\ p^n\Z_p
      & p^n\Z_p & p^n\Z_p & \Z_p\end{smallmatrix} \right)$ to
  itself, and $K(p^n)$ is the intersection of $L$ with $\GSp_2(\Q_p)$, so it suffices to show that $K^+_{p^n}$ is the intersection of $L$ with $\GUtwo{B_p}$. If we describe $L$ with the notation
  $\left(\begin{smallmatrix} M_2(\ZZ_p) & M_2(\ZZ_p) \cdot \pi^{-n}\\
      \pi^n\cdot M_2(\ZZ_p) & \pi^n \cdot M_2(\ZZ_p)\cdot
      \pi^{-n}\end{smallmatrix}\right)$, it is easy to verify this, recalling that $L_{p^n}:= \begin{pmatrix}M_2(\Z_p)\\ \pi^n \cdot M_2(\ZZ_p)\end{pmatrix}$.
\end{proof}

It is well known that the integral condition on
elements of $K(p^n)$ plus the fact that it preserves the symplectic
form imply that $K(p^n)$ is the intersection of $\GSp_2(\Q_p)$ with the order
\begin{equation}
  \label{eq:eltsparamodular}
R:=  \left( \begin{smallmatrix}
    \Z_p&p^n \Z_p&\Z_p&\Z_p\\
    \Z_p& \Z_p&\Z_p&p^{-n}\Z_p\\
    \Z_p&p^n\Z_p&\Z_p&\Z_p\\
    p^n\Z_p&p^n\Z_p&p^n\Z_p&\Z_p
\end{smallmatrix}\right) \subset M_4(\Q_p).
%R:=  \left( \begin{smallmatrix}
%    \Z_p&\Z_p&p^n \Z_p&\Z_p\\
%    \Z_p& \Z_p&p^n\Z_p&\Z_p\\
%    \Z_p&\Z_p&\Z_p&p^{-n}\Z_p\\
%    p^n\Z_p&p^n\Z_p&p^n\Z_p&\Z_p
\end{equation}

   Note that the subgroup $K(p^n)$ is normalised by the Atkin-Lehner
   element of $\GSp_2(\Q_p)$
   \begin{equation}
     \label{eq:A-L}
     W_{p^n}:=\left(\begin{smallmatrix}
       0 & p^n &0 & 0\\
       1 & 0 & 0 & 0\\
       0 & 0 & 0 & 1\\
       0 & 0 & p^n & 0
      \end{smallmatrix}\right).
   \end{equation}
%$$W_{p^n}:=\begin{pmatrix}0&0&p^n&0\\0&0&0&1\\1&0&0&0\\0&p^n&0&0\end{pmatrix}\in\GSp_2(\Q_p),$$
  The Atkin-Lehner involution satisfies $W_{p^n}^2=p^nI$, and
  $W_{p^n}^*W_{p^n}=p^nI$ hence $\nu(W_{p^n})=p^n$.  The preimage of
  (\ref{eq:eltsparamodular}) under (\ref{eq:isom}) can be expressed as
  the block matrices
\begin{equation}
  \label{eq:paramodularlevel}
    \left(\begin{smallmatrix} M_2(\ZZ_p) & M_2(\ZZ_p)\cdot \bar{\pi}^n\\ \pi^n \cdot M_2(\ZZ_p) & \pi^n\cdot M_2(\ZZ_p) \cdot \pi^{-n}\end{smallmatrix} \right) 
  = 
  \smat{1&0\\0&\bar\pi^{-n}}
  \smat{M_2(\ZZ_p) & M_2(\ZZ_p)\\ p^n
  M_2(\ZZ_p)&M_2(\ZZ_p)}
  %\smat{M_2(\ZZ_p) & M_2(\ZZ_p)\\ p^n M_2(\ZZ_p)&M_2(\ZZ_p)}
  \smat{1&0\\0&\bar\pi^n}.
\end{equation}
 The preimage under (\ref{eq:isom}) of the Atkin-Lehner involution equals $W_{p^n}^+:=\smat{0&0&p^n&0\\0&0&0&1\\1&0&0&0\\0&p^n&0&0}$.

The following is \cite[Chapitre II, Theoreme 2.3(1)]{MR580949}.

\begin{lem}
  Let $R_1, R_2$ be two maximal orders of $M_n(\Q_p)$, then there exists $\alpha \in \GL_n(\Q_p)$ such that $\alpha R_1 \alpha^{-1} = R_2$.
\end{lem}

\begin{lem}\label{maximal} The only maximal orders in $M_4(\Q_p)$ containing $R$ are $\alpha_m^{-1}\,M_4(\Z_p)\alpha_m$, for $0\leq m\leq n$, where $\alpha_m:=\diag(1,p^m,1,p^{m-n})$.
\end{lem}
\begin{proof} Recall the definition of $R$ given in
  (\ref{eq:eltsparamodular}). By the previous lemma, any maximal order
  is of the form $\alpha^{-1}\,M_4(\Z_p)\alpha$, for some
  $\alpha\in\GL_4(\Q_p)$. Suppose that
  $R\subset \alpha^{-1}\,M_4(\Z_p)\alpha$. By the Iwasawa
  decomposition, we may left multiply $\alpha$ by an element of
  $\GL_4(\Z_p)$ to put it in upper triangular form
  $\alpha=\begin{pmatrix}a&b&c&d\\0&e&f&g\\0&0&h&i\\0&0&0&j\end{pmatrix}$. Applying
  $\alpha\,R\alpha^{-1}\subset M_4(\Z_p)$ to the elements
  $\diag(0,1,0,0)$, $\diag(0,0,1,0)$ and $\diag(0,0,0,1)$ of $R$, and
  inspecting the above-diagonal elements of the second, third and
  fourth columns respectively, we find that
  $\frac{b}{e}, \frac{c}{h}, \frac{f}{h}, \frac{d}{j}, \frac{g}{j},
  \frac{i}{j}\in\Z_p$.  The factorisation
  \[
    \alpha=
    \begin{pmatrix}1&\frac{b}{e}&\frac{c}{h}&\frac{d}{j}\\
      0&1&\frac{f}{h}&\frac{g}{j}\\
      0&0&1&\frac{i}{j}\\
      0&0&0&1\end{pmatrix}
    \begin{pmatrix}
      a&0&0&0\\
      0&e&0&0\\
      0&0&h&0\\0&0&0&j
    \end{pmatrix},
    \]
    and the fact that the first factor is in $\GL_4(\Z_p)$ implies
    that we can assume $\alpha$ equals the second factor. Then
    \[
      \alpha\,R\alpha^{-1}=
      \begin{pmatrix}
        \Z_p&p^n\frac{a}{e}\Z_p&\frac{a}{h}\Z_p&\frac{a}{j}\Z_p\\
        \frac{e}{a}\Z_p&\Z_p&\frac{e}{h}\Z_p&p^{-n}\frac{e}{j}\Z_p\\
        \frac{h}{a}\Z_p&p^n\frac{h}{e}\Z_p&\Z_p&\frac{h}{j}\Z_p\\
        p^n\frac{j}{a}\Z_p&p^n\frac{j}{e}\Z_p&p^n\frac{j}{h}\Z_p&\Z_p
      \end{pmatrix}.
      \]
      For this to be contained in $M_4(\Z_p)$, we see that
      $-n\leq\ord_p(\frac{a}{e})\leq 0$, $-n\leq\ord_p(\frac{h}{e})\leq 0$,
      $\ord_p(\frac{a}{h})=0$,
      $0\leq\ord_p(\frac{a}{j})\leq n$, $\ord_p(\frac{e}{j})=n$ and
      $0\leq\ord_p(\frac{h}{j})\leq n$. Hence (removing a scalar
      power of $p$ and a diagonal unit matrix) we may assume that
      $\alpha=\alpha_m:=\diag(1,p^m,1,p^{m-n})$ for $0\leq m\leq n$.
    \end{proof}

\begin{lem}
  \label{intersection}
The unique way of writing $R$ as an intersection of maximal orders in $M_4(\Q_p)$ is
\[
  R=h^n\,M_4(\Z_p)h^{-n}\cap W_{p^n}^{-1}h^n\,M_4(\Z_p)h^{-n}W_{p^n}.
\]
\end{lem}
\begin{proof} 
Recall that
\[
  R=\left(\begin{smallmatrix}
    \Z_p&p^n\Z_p&\Z_p&\Z_p\\\
    Z_p&\Z_p&\Z_p&p^{-n}\Z_p\\
    \Z_p&p^n\Z_p&\Z_p&\Z_p\\
    p^n\Z_p&p^n\Z_p&p^n\Z_p&\Z_p
  \end{smallmatrix}\right),
\]
while
\[
  \alpha_m^{-1}\,M_4(\Z_p)\alpha_m=
  \left(\begin{smallmatrix}
    \Z_p&p^m\Z_p&\Z_p&p^{m-n}\Z_p\\
    p^{-m}\Z_p&\Z_p&p^{-m}\Z_p&p^{-n}\Z_p\\
    \Z_p&p^m\Z_p&\Z_p&p^{m-n}\Z_p\\
    p^{n-m}\Z_p&p^n\Z_p&\p^{n-m}Z_p&\Z_p
  \end{smallmatrix}\right).
  \]
If
$R=\alpha_{m_1}^{-1}\,M_4(\Z_p)\alpha_{m_1}\cap\alpha_{m_2}^{-1}\,M_4(\Z_p)\alpha_{m_2}$,
with $0\leq m_1,m_2\leq
n$ then to avoid non-integral elements in the left entry of the
second row we must have some
$m_i=0$, and to avoid non-integral elements in the top entry of
the fourth column we must have some $m_i=n$, say $m_1=0$, $m_2=n$.

Clearly $$\alpha_0^{-1}\,M_4(\Z_p)\alpha_0=h^n\,M_4(\Z_p)h^{-n},$$
since $\alpha_0=\diag(1,1,1,p^{-n})=h^{-n}$, and $$\alpha_n^{-1}\,M_4(\Z_p)\alpha_n=W_{p^n}^{-1}h^n\,M_4(\Z_p)h^{-n}W_{p^n},$$
since
\[
  h^{-n}W_{p^n}=
  \begin{pmatrix}
    0&p^n&0&0\\
    1&0&0&0\\
    0&0&0&1\\
    0&0&1&0\end{pmatrix}=
  \begin{pmatrix}
    0&1&0&0\\
    1&0&0&0\\
    0&0&0&1\\
    0&0&1&0\end{pmatrix}\,\alpha_n.
  \]
\end{proof}
\begin{lem}\label{normaliser} The normaliser of $R$, i.e. the set $\{g\in\GL_4(\Q_p):\,\,g^{-1}Rg=R\}$, is the union of $p^{\Z}W_{p^n}^{\mu}\,R^{\times}$ for $\mu=0,1$.
\end{lem}
\begin{proof} By Lemma \ref{intersection}, conjugation by $g$ either fixes or swaps the maximal orders $h^n\,M_4(\Z_p)h^{-n}$ and $W_{p^n}^{-1}h^n\,M_4(\Z_p)h^{-n}W_{p^n}$. Clearly it suffices to show that the normaliser of $M_4(\Z_p)$ is $p^{\Z}M_4(\Z_p)^{\times}$, but this can be done by reducing to diagonal elements, as in the proof of Lemma \ref{maximal}, but easier.
\end{proof}

\section{Local lattices in a six-dimensional $\GUtwo{B}$-space}

Consider an action of $\GUtwo{B}$ on the $\Q$-vector space (of
Lemma~\ref{lemma:clifordinvariant})
\begin{equation}
  \label{eq:starinvariant}
\DimSix:=\{A\in M_2(B):\,A^*=A\}.  
\end{equation}
This space is very related to the one considered by Ibukiyama in
\cite{I}, but we remove the trace zero hypothesis.

\begin{lem}
  The vector space $\DimSix$ is of dimension $6$. Moreover, it is given
  by
  \[
\DimSix = \left \{ \left(\begin{matrix} s & r \\ \overline{r} & t\end{matrix}\right) \; : \; s, t \in \QQ, r \in B\right\}.
\]
\label{lemma:W}
\end{lem}
\begin{proof}
  It is clear that if $\overline{M^t} =M$ then the entries $(1,1)$ and
  $(2,2)$ of the matrix are fixed by the involution, hence
  rational. The second hypothesis is also clear.
\end{proof}

Note that the space $\DimSix$ contains the centre of $M_2(B)$
(i.e. the rational scalar matrices). In 
Section~\ref{section:quadraticform} we will define a quadratic form on
$\DimSix$ invariant under translation by the centre, hence the
quotient space becomes a quadratic space.

\begin{prop}
  The group $\GUtwo{B}$ acts on $\DimSix$ via conjugation.
\label{prop:Winvariance}
\end{prop}

\begin{proof}
  If $g \in \GUtwo{B}$ it satisfies that $g^{-1} = \frac{g^*}{\nu(g)}$. Then
  $(g v g^{-1})^* = (g v \frac{g^*}{\nu(g)})^* = \frac{g}{\nu(g)}
  v^* g^*  = g v g^{-1}$ because
  $\nu(g)$ is in the centre of $M_2(B)$.
\end{proof}

\subsection{Local lattices in $U$ at split primes.}
\label{section:lattice}
Let $p$ be an unramified prime (i.e. $p \not \in S$). 
Given $n\geq 0$, define a $\Z_p$-lattice
$\UU^+_{p^n}\subseteq \DimSix_p:=\DimSix\otimes\Q_p$ by
\begin{equation}
  \label{eq:splitlattice}
  \UU^+_{p^n}:=\left\{\begin{pmatrix}
      s & \begin{pmatrix}p^na&b\\p^nc&d \end{pmatrix}\\
      \begin{pmatrix}d&-b\\-p^nc&p^na\end{pmatrix} & t\end{pmatrix}
    :\, a,b,c,d, s,t\in \Z_p\right\}.  
\end{equation}
Using the previous notation (i.e. $\pi=\left(\begin{smallmatrix}1 & 0\\ 0 & p\end{smallmatrix}\right)$), the lattice can be written in the compact form $\UU^+_{p^n}
= \left \{ \left(\begin{matrix} s & r\\ \bar{r} &
        t\end{matrix} \right) \; : \; s,t \in \Z_p, \text{ } r
    \in \Z_{B_p}\cdot \bar{\pi}^n\right\}$.

  \begin{lem}
    \label{order}
    Suppose that $p$ is prime. The subring $R'$ of $M_4(\Q_p)$ generated by
    $\UU^+_{p^n}$ equals
    \[
\qquad       \begin{pmatrix}
        \Z_pI_2&0_2\\
        0_2&\Z_pI_2\end{pmatrix}
      \oplus
        \begin{pmatrix}
          p^n\Z_p&p^n\Z_p&p^n\Z_p&\Z_p\\
          p^n\Z_p&p^n\Z_p&p^n\Z_p&\Z_p\\
          \Z_p&\Z_p&p^n\Z_p&\Z_p\\
          p^n\Z_p&p^n\Z_p&p^{2n}\Z_p&p^n\Z_p\end{pmatrix}.
          \]
\end{lem}
\begin{proof}
  Since $\begin{pmatrix}I_2&0_2\\0_2&0_2\end{pmatrix} \in R'$ and
  $\begin{pmatrix}0_2&0_2\\0_2&I_2\end{pmatrix}\in R'$, if we multiply
  these elements by elements in $\UU^+_{p^n}$ of the form
  $\begin{pmatrix} 0_2 & \begin{pmatrix}
      a&b\\c&d\end{pmatrix}\\\begin{pmatrix}d&-b\\-c&a\end{pmatrix} &
    0_2\end{pmatrix}$, with $b,d\in\Z_p$ and $a,c\in p^n\Z_p$, we find
  that
  $\begin{pmatrix} 0_2 & \begin{pmatrix}
      p^n\Z_p&\Z_p\\p^n\Z_p&\Z_p\end{pmatrix}\\0_2 &
    0_2\end{pmatrix}\subset R'$ and
  $\begin{pmatrix} 0_2 &
    0_2\\\begin{pmatrix}\Z_p&\Z_p\\p^n\Z_p&p^n\Z_p\end{pmatrix} &
    0_2\end{pmatrix}\subset R'$. Multiplying elements of these subsets
  together, one way round or the other, we find that
$$\begin{pmatrix}\Z_pI_2&0_2\\0_2&\Z_pI_2\end{pmatrix}\oplus\begin{pmatrix}
  p^n\Z_p&p^n\Z_p&p^n\Z_p&\Z_p\\p^n\Z_p&p^n\Z_p&p^n\Z_p&\Z_p\\\Z_p&\Z_p&p^n\Z_p&\Z_p\\p^n\Z_p&p^n\Z_p&p^{2n}\Z_p&p^n\Z_p\end{pmatrix}\subseteq
R'.$$ It is easy to see that the left hand side is closed under
multiplication, hence the inclusion must be equality.
\end{proof}
\begin{remar} 
If one takes the second summand, divides the top left and bottom right $2 \times 2$ blocks by $p^n$ and applies the map $\Psi: M_2(B_p) \to M_4(\Q_p)$~(cf. (\ref{eq:isom})) then one obtains $R$ from the previous section (cf. (\ref{eq:eltsparamodular})).
\end{remar}

\begin{remar}\label{p2trace0}
  Anticipating Remark~\ref{remark:tracezero} below, when
  $p=2$, had we replaced $\UU^+_{p^n}$ by its trace zero
  sublattice, we would only have been able to show that
  $\left(\begin{smallmatrix} 2I_2 & 0_2 \\ 0_2 &
      0_2\end{smallmatrix}\right)$ lies in $R'$.
\end{remar}

\subsection{Local lattices in $U$ at non-split primes.}
Let $p$ now be a ramified prime. Recall that $\xi$ was chosen so that
$\xi^*\xi=H:=\left(\begin{smallmatrix} 0 & 1\\1 & 0\end{smallmatrix}\right)$, giving
$\xi^{-1}\GUtwo{B_p}\xi=\GUone{B_p}$. Following \cite[\S
4]{I}, define $\widetilde{\DimSix}_p=\xi^{-1}\DimSix_p\xi$.

\begin{lem}
  The vector space $\widetilde{\DimSix}_p$ equals the space
  $\left\{\left(\begin{smallmatrix} r & s\\ t &
        \bar{r}\end{smallmatrix} \right) \, : \, s,t \in \Q_p, r \in
      B_p\right\}$.
\end{lem}
\begin{proof}
  The proof is given in \cite{I} (page 212), although the author
  considers only the subspace of $\DimSix_p$ of trace zero
  elements. Recall that over $\Q_p$ the two spaces differ by the
  identity matrix, which maps to itself under conjugation.
%An
%  easy computation proves that if $v \in \DimSix_p$ then
%  $v_1:=\xi \cdot v \cdot \xi^{-1}$ satisfies that
%  $v_1 \cdot \left(\begin{smallmatrix} 0 &1 \\ 1 & 0\end{smallmatrix}
%  \right) = \left(\begin{smallmatrix} 0 & 1\\ 1 & 0\end{smallmatrix}
%  \right) \cdot v_1^*$ and the result follows.
\end{proof}
For any $p$, define $\UU^-_p\subseteq \DimSix_p$ by $\UU^-_p:=\xi\tilde{\UU}^-_p\xi^{-1}$, where
$\tilde{\UU}^-_p\subseteq\widetilde{\DimSix}_p$ is defined by
\begin{equation}
  \label{eq:nonsplitlattice}
  \tilde{\UU}^-_p:=\left\{
    \begin{pmatrix}
      r&ps\\
      t&\overline{r}
    \end{pmatrix}:\,r \in \OOO_p, s,t\in \Z_p\right\}.
\end{equation}

  \begin{lem}
    \label{lemma:orderramifiedprime}
    Suppose that $p$ is prime. The subring $R'$ of $M_2(B_p)$ generated by
    $\tilde{\UU}^-_{p}$ equals 
    \[
\qquad      \left\{ \begin{pmatrix}
        a & b\\
        c & d\end{pmatrix} \in M_2(B_p) \; : \; a,c,d \in \mxlram, b \in p\mxlram \text{ and } a \equiv \bar{d} \pmod{\pi}\right\}.
          \]
\end{lem}
\begin{proof}
  The proof is similar to \cite[Lemma 4.2]{I}. Multiplying the elements
\[
  \begin{pmatrix} 0 & 0 \\ 1 & 0\end{pmatrix}
  \begin{pmatrix}r & 0\\ 0 & \bar{r}\end{pmatrix} = \begin{pmatrix} 0
    & 0 \\ r & 0\end{pmatrix},
\]
proves that the element in the place $(2,1)$ of the matrix can be
arbitrary. Considering the element
$\left(\begin{smallmatrix} 0 & p\\0 & 0 \end{smallmatrix}\right)$ we
get the same result for the entry $(1,2)$ but with multiples of $p$. To get the diagonal elements, note that
\[
  \begin{pmatrix} r & 0\\ 0 & \bar{r} \end{pmatrix} \begin{pmatrix} s
    & 0 \\ 0 & \bar{s} \end{pmatrix} = \begin{pmatrix} rs & 0 \\ 0 &
    \overline{rs}\end{pmatrix} + \begin{pmatrix} 0 & 0 \\ 0 & \bar{r}\bar{s} - \bar{s}
    \bar{r}\end{pmatrix}.
\]
In particular, we can add to an element of $\tilde{\UU}^-_p$ a matrix
with any element of the form $rs - sr$ to the place
$(2,2)$. Note that any such difference lies in the unique maximal ideal
(as the quotient of $\mxlram$ by $\pi$ is the finite field of $p^2$
elements, in particular is abelian) and in fact, they generate the maximal ideal, hence the statement.
\end{proof}

\section{Special lattices}
\label{section:special}
Let $R$ be a Dedekind domain with field of fractions $k$, of
characteristic different from $2$.
Let $(V,Q)$ be a regular quadratic space over $k$ with
associated symmetric bilinear form
$\langle v,w\rangle=Q(v+w)-Q(v)-Q(w)$, so that $\langle v,v\rangle = 2Q(v)$.

An $R$-lattice on $V$ is a finitely generated $R$-submodule
$L\subset V$ such that $k\,L=V$. If $L$ is an $R$-lattice, its \emph{dual lattice} $L^\vee$ is defined by
\[
  L^\vee:=\{v\in V \;:\; \langle v,w\rangle \in R \; \,\,\forall w \in L\}.
\]
A lattice $L$ is \emph{integral} if $Q(L)\subset R$; this implies
$L\subseteq L^\vee$.
If $L$ is integral and $L=L^\vee$ it is called \emph{even unimodular}.
More generally $L$ is called \emph{modular}
if $L=IL^\vee$ for some ideal $I\lhd R$.
The \emph{(signed) determinant} of a free lattice $L$ is defined
as $\det L:=(-1)^{\lfloor n/2\rfloor}\det(\langle v_i,v_j\rangle)$
where $\{v_1,\dotsc,v_n\}$ is a basis of $L$.
Note that $\det L$ is, up to squares of units, independent of the
choice of basis.
\par
It is clear that if $L$ is integral then $\det L\in R$.
If, in addition,
the rank is odd we have $\det L\in 2R$: considering the expression for
$\det L$ as an alternating sum of products, transposition gives us pairs
of products, equal because the matrix is symmetric,
and a product can be paired with itself only when either the rank is
even or the product includes an even factor from the diagonal.
\par
A lattice is called \emph{maximal} if it is maximal among all integral
lattices in its ambient quadratic space. Note
that an integral lattice with unit determinant is necessarily
maximal since any proper super-lattice would have non-integral
determinant; for a similar reason an integral lattice of odd rank
whose determinant is twice a unit is maximal.

% As a fractional $R$-ideal $(\det L)R^\times$ equals the
% group index $(L^\vee:L)_R$ \cite[Lemma 1.64]{2008.12847}.

%Choosing a basis $\{v_1, \ldots,v_n\}$ for $L$, the Gram matrix
%$(\langle v_i,v_j\rangle)$ of the bilinear form is the Hessian (matrix
%of second partial derivatives) of the quadratic form $Q$, denoted
%$H(L)$.
%We will use the term \emph{determinant of $Q$} to denote the
% determinant of the Hessian matrix, which does not depend on the choice
% of basis, up to squares in $R$.

\begin{defi}
    A \emph{special lattice} is
    an integral $R$-lattice $L$ of odd rank such that
    $L^\vee/L$ is cyclic as an $R$-module.
\end{defi}

% Note that since $L$ is of odd rank we must have
% $L^\vee/L\simeq R/2I$ for some ideal $I\lhd R$. WHY?

% \begin{defi}
%   Let $L$ be a special $R$-lattice of odd rank. The \emph{level} of
%   $L$ is the ideal $I$ such that $L^\vee/L \simeq R/2I$.
% \end{defi}

% \begin{prop}
%   Let $L$ be an $R$-lattice of odd rank.
%   \begin{enumerate}
%   \item $L$ is special if and only if $L_p$ is special for all
%     non-archimedean places $p$ of $R$.
%   \item The level of $L$ factors as $I = \prod_{p} p^{e_p}$, where the
%     product runs over all primes and $(pR_p)^{e_p}$ is the level of
%     $L_p$.
%   \end{enumerate}
%   \label{prop:speciallocalglobal}
% \end{prop}

\subsection{Local classification of special lattices}
In this section $R$ is the ring of integers of a local field with
valuation $v$ and maximal ideal $p$. Since $R$ is a principal ideal
domain all $R$-lattices are free.
%and determinants are well defined
%up to squares in $R^\times$.

\begin{thm} \label{thm:decomposition}
    If $L$ is a special lattice then $L=A\perp Rw$, where $A$ is even
    unimodular and $w\in L$.
    %In particular $\det L=2N$ where $N=Q(w)\det A\in R$.
    %In particular $\det L=2D$ where $D=Q(w)\det A$.
    %When $D\in R^\times$ we can choose $A$ such that $e(A)=1$
    %Moreover $L^\vee/L\simeq R/2aR$ where $a=Q(w)$.
    %Moreover $\det L=2Q(w)$.
    %the level of $L$ is $aR$ where $a=Q(w)$.
    %and $\det L/2a=\det A$ is a unit congruent to $1$ modulo $4$.
  \end{thm}
\begin{proof}
    The lattice $L$ can be written as the orthogonal sum of modular
    lattices of rank 1 or 2 \cite[§91C]{OMeara}.
    %\cite[Corollary 5.11]{2008.12847}. 
    Say
    \[
        L = A_1 \perp \dotsb \perp A_s
    \]
    where $A_i = I_i\,A_i^\vee$ for some $I_i\lhd R$. Then
    \[
        L^\vee/L = \bigoplus_i A_i^\vee/I_i\,A_i^\vee
        \simeq \bigoplus_i (R/I_i)^{\dim A_i}.
    \]
    Since $L$ is special, at most one $I_{i_0}\neq R$ and necessarily
    $\dim A_{i_0}=1$. If all $I_i=R$ choose any $i_0$ with
    $\dim A_{i_0}=1$, which exists since $\dim L$ is odd.
    In any case $L = A \perp Rw$ where $A=\bigoplus_{i\neq i_0} A_i$
    is even unimodular and $A_{i_0}=Rw$.
    %Finally $L^\vee=A^\vee\perp(Rw)^\vee=A\perp\frac1{2a}Rw$,
    %from which the last statement follows.
\end{proof}

We now recall some useful facts about even unimodular lattices.
In what follows we will say a unit $u$ is \emph{unramified} if
it is a square modulo $4$. 
\begin{lem}\label{lem:unimodular_HW}
    Let $A$ be an even unimodular lattice.
    Then $\det A$ is an unramified unit
    and $\HW_v(A)=1$.
\end{lem}
\begin{proof}
    When $p\nmid 2$
    the lattice $A$ has an orthogonal basis of vectors
    with unit norm so $\det A$ is a unit.
    The Hilbert symbol is trivial on units, hence
    $\HW_v(A)=1$.
    \par
    When $p\mid 2$ the lattice $A$ is the orthogonal sum of
    even unimodular lattices of rank 2 (since there are no even
    unimodular lattices of rank 1).
    An even unimodular lattice of rank 2 has matrix
    $\smat{2a&b\\b&2c}$ with $b$ a unit,
    hence its (signed) determinant $b^2-4ac$ is the square of a unit
    modulo 4.
    Since the determinant for lattices of even rank is multiplicative,
    it follows that $\det A$ is the square of a unit modulo $4$.
    \par
    For the Hasse-Witt invariant use the fact that
    for lattices $A_1$ and $A_2$ of even rank we have
    $\HW_v(A_1\perp A_2)=\HW_v(A_1)\cdot\HW_v(A_2)\cdot(\det A_1,\det A_2)_v$
    (see \cite[(3.13) in p.117]{MR2104929}); the Hilbert symbol is
    trivial on unramified units so it suffices to
    prove $\HW_v(A)=1$ when $A$ has rank 2.
    Since $A$ is even unimodular there is some $w_1\in A$
    with $Q(w_1)=\alpha\in R^\times$.
    Let $w_2\in A$ such that $\langle w_1,w_2\rangle=0$
    and let $Q(w_2)=\beta$. The vectors $w_1$ and $w_2$ span a lattice
    $B\subseteq A$ with $\det(B) = -\alpha\beta = \det A\cdot s^2$
    for some $s\in R$ and
    $\HW_v(A)=\HW_v(B)
    %=\HW_v(Rw_1\perp Rw_2)
    =(\alpha,\beta)_v
    = (\alpha,-\alpha\beta)_v =(\alpha,\det A)_v=1$
    where in the last equality we have used that $\alpha$ is a unit
    and that $\det A$ is an unramified unit.
    % Since $A$ is integral of
    % unit determinant,
    % its Clifford algebra $\Cliff(A)$ is a quaternion order of unit
    % discriminant.
    % Hence $\Cliff(A)$ must be the maximal order in a split
    % quaternion algebra and its class in the Brauer group is trivial;
    % therefore $\HW_v(A)=1$.
\end{proof}

In view of this lemma we define an invariant for even unimodular
lattices as follows.
\begin{defi}
    If $A$ is an even unimodular lattice, we let
    \[
        d(A) = (\tilde p,\det A)_v \,.
    \]
    where $\tilde p$ is a local uniformizer.
\end{defi}

Note that $d(A)$ is independent of the choice of uniformizer.
Indeed, since $\det A$ is an unramified unit
the Hilbert symbol $(u,\det A)_v$ equals $1$ for any unit $u$.

\begin{lem}\label{lem:unimodular_invariant}
    Let $A$ be an even unimodular lattice. Then
    the invariant $d(A)$, together with the rank,
    determines the isometry class of $A$.
\end{lem}
\begin{proof}
    Suppose $A$ and $A'$ are unimodular lattices
    of the same rank such that $d(A)=d(A')$.
    The latter means that
    $\det A/\det A'$ is a square modulo $4p$ and by the local square
    theorem (see \cite[63:1]{OMeara}) this implies that $\det A/\det A'$ is
    a square in $R$. We conclude $A$ and $A'$ have the same rank,
    determinant and Hasse-Witt invariant so they lie in isometric
    quadratic spaces (\cite[63:20]{OMeara}). Finally, $A$ and $A'$
    are maximal lattices in isometric quadratic spaces so they are
    themselves isometric (\cite[91:2]{OMeara}).
\end{proof}

\begin{lem}
    \label{lem:unimodular_existence}
    Let $n\geq 1$ and $d\in\{\pm1\}$. There is an even
    unimodular lattice $A$ of rank $2n$ and $d(A)=d$.
\end{lem}
\begin{proof}
    Pick $u\in R^\times$ such that $u\equiv1\pmod{4}$ and $(p,u)_v=d$,
    and write $u=1-4\alpha$ with $\alpha\in R$.
    Let $J$ be the binary lattice with
    matrix $\smat{2&1\\1&2\alpha}$ which is even unimodular of determinant $u$,
    and let $H$ be the hyperbolic plane which is even unimodular of
    determinant $1$.  Then $A=J\perp H^{n-1}$ has rank $2n$, determinant $u$
    and invariant $d(A)=d$.
\end{proof}

\begin{prop}\label{prop:special_unitdet}
    Let $L$ be a special lattice of determinant $2N$
    with $p\nmid N$.
    Then $\HW_v(L)=1$, $L$ is maximal, and
    the class of $N$ modulo squares, together with the
    rank, determines the isometry class of $L$.
    In particular $L=A\perp Rw$ where $A$ is even unimodular of
    determinant $1$ and $Q(w)=N$.
\end{prop}
\begin{proof}
    Write $L=A\perp Rw$ as in Theorem~\ref{thm:decomposition}.
    From \cite[(3.13) in p.117]{MR2104929} we have
    $\HW_v(L)=\HW_v(A)\cdot\HW_v(Rw)\cdot(-Q(w),\det A)_v$.
    By Lemma~\ref{lem:unimodular_HW}, we have
    $\HW_v(A)=1$ and $\det A$ is an unramified unit;
    the hypothesis implies $Q(w)$ is a unit and so
    the Hilbert symbol $(-Q(w),\det A)_v$
    equals $1$. Finally $\HW_v(Rw)=1$ because it has rank 1, and it
    follows that $\HW_v(L)=1$.
    \par
    Since $L$ is of odd rank with $\det L\in 2R^\times$ we have that
    $L$ is maximal. From this it follows, as in the proof of
    Lemma~\ref{lem:unimodular_invariant},
    that $N$ and the rank determine the isometry class of $L$.
    % Let  $L'$ be another special lattice of the same rank and of
    % determinant $2D$ up to squares. Then $L$ and $L'$ have the same
    % rank, determinant and Hasse-Witt invariant so they lie in
    % isometric quadratic spaces, and they are maximal lattices so
    % they are themselves isometric.
    \par
    For the last claim let $A$ be an orthogonal sum of hyperbolic
    planes so it is even unimodular with $\det A=1$ and consider a unary
    lattice $Rw$ with $Q(w)=N$. Then $A\perp Rw$ has the same
    determinant and rank as $L$.
\end{proof}

We aim to classify the special lattices of a given rank and
determinant.
For this purpose we introduce an invariant
which, in the case of rank 3, is related to the Eichler
invariant of quaternion orders.

\begin{defi}
    The \emph{Eichler invariant} of a special lattice $L$ of
    determinant $2N$ is given by
    \[
        e(L) = \begin{cases}
            1 & \text{if $4N\,Q(v)=1$ for some $v\in L^\vee$;} \\
            -1 & \text{otherwise.}
        \end{cases}
    \]
\end{defi}

\begin{prop}
    \label{prop:special_nonunitdet}
    Let $L$ be a special lattice of determinant $2N$
    and write $L=A\perp Rw$ as in Theorem~\ref{thm:decomposition}.
    If $p\nmid N$ then $e(L)=1$
    and if $p\mid N$ then $e(L)=d(A)$.
\end{prop}
\begin{proof}
We have $L^\vee=A\perp\frac1{2a}Rw$ where $a=Q(w)$ so
any vector $v\in L^\vee$ can be written as
$v=\frac{x}{2a}w+u$ with $x\in R$ and $u\in A$.
We compute
$4N\,Q(v)=tx^2 + 4N\,Q(u)$ where $t=\frac Na=\det A$.
When $p\nmid N$ we can assume, by Proposition~\ref{prop:special_unitdet},
that $t=\det A=1$ and so $4N\,Q(v)$ represents 1.
When $p\mid N$ we have $4N\,Q(v)\equiv tx^2\pmod{4p}$ represents 1 if
and only if $t$ is a square modulo $4p$, i.e. if and only if $d(A)=1$.
\end{proof}

In any case we can always assume $e(L)=d(A)$; when $p\nmid N$ using
the last part of
Proposition~\ref{prop:special_unitdet}.

\begin{cor} \label{cor:hasse-eichler}
    For a special lattice of determinant $2N$ the Hasse-Witt invariant
    satisfies $\HW_v(L)=e(L)^{v(N)}$.
\end{cor}
\begin{proof}
  Write $L=A\perp Rw$ as in Theorem~\ref{thm:decomposition}.
  We know $\HW_v(A)=1$ by Lemma~\ref{lem:unimodular_HW};
  also $\HW_v(Rw)=1$ since $Rw$ has rank 1.
  Applying \cite[(3.13) in p.117]{MR2104929} as before
  we conclude $\HW_v(A\perp Rw)=(-Q(w),\det A)_v$.
  Since $\det A$ is an unramified unit and $N/Q(w)$ is a
  unit we have $(-Q(w),\det A)_v=(N,\det A)_v
  %=(\tilde p,\det A)_v^{v(N)} 
  =d(A)^{v(N)}=e(L)^{v(N)}$.
\end{proof}

\begin{cor}
\label{cor:special_unique}
    Two special lattices of the same rank, determinant and Eichler
    invariant are isometric.
\end{cor}
\begin{proof}
Let $L$ and $L'$ be the two lattices of determinant $2N$.
When $p\nmid N$ the claim follows from
Proposition~\ref{prop:special_unitdet}.
When $p\mid N$ write $L=A\perp Rw$ and $L'=A'\perp Rw'$.
The hypothesis $e(L)=e(L')$ implies, by
Proposition~\ref{prop:special_nonunitdet}, that
$d(A)=d(A')$ and hence, by Lemma~\ref{lem:unimodular_invariant} that
$A$ and $A'$ are isometric.
Moreover $Q(w)/Q(w')=\det A'/\det A$ is the square of a unit so
that $Rw$ and $Rw'$ are isometric.
We conclude that $L$ and $L'$ are isometric.
\end{proof}

\begin{thm}
    \label{thm:special_local}
    For a given $n\geq1$, $N\in R$ and $e\in\{\pm1\}$ there exists a unique
    isometry class of special lattices of rank $2n+1$, determinant $2N$
    and Eichler invariant $e$, provided $e=1$ when $N\in R^\times$.
\end{thm}
\begin{proof}
    Let $A$ be an even unimodular lattice of rank $2n$ with invariant
    $d(A)=e$ which exists by Lemma~\ref{lem:unimodular_existence}.
    Then $L=A\perp Rw$ where $Q(w)=N/\det A$ has rank $2n+1$,
    determinant $2N$ and $e(L)=d(A)=e$.
    Uniqueness follows from Corollary~\ref{cor:special_unique}.
\end{proof}

\subsection{Global classification}
In this section we let $R$ be the ring of integers of a number field
$k$ with $r_1$ real places.
If $p$ is a prime ideal of $R$, by $R_p$ we denote the completion of
$R$ at $p$. If $L$ is an $R$-lattice, then $L_p = L \otimes R_p$.
\begin{prop}
  \label{prop:speciallocalglobal}
  Let $L$ be an $R$-lattice of odd rank.
  Then $L$ is special if and only if $L_p$ is special for all
    non-archimedean places $p$ of $R$.
\end{prop}
\begin{proof}
    If $L$ is special then $L^\vee/L\simeq R/I$ for some ideal
    $I\lhd R$. Then $L_p^\vee/L_p\simeq R_p/I_p$ is cyclic.
    \par
    For the converse let $S$ be a finite set of primes such that
    $L_p$ is even unimodular for $p\not\in S$.
    For $p\in S$ we have $L_p^\vee/L_p\simeq R_p/I_p= R/(I_p\cap R)$
    for some ideal $I_p\lhd R_p$.
    Then,
    using the Chinese remainder theorem,
    we have
    $L^\vee/L\simeq R/I$ where $I=\prod_{p\in S} (R\cap I_p)$.
\end{proof}

\begin{thm}
For a given $n\geq3$ odd, $N\in R$ and $e_p\in\{\pm1\}$ for each
$p\mid N$, there is a unique genus of totally positive definite
special lattices of rank $n$, determinant $2N$ and local Eichler
invariants $e_p$, provided
\[
    \prod_{p\mid N} e_p^{v(N)} = (-1)^{[k:\QQ]}.
\]
\end{thm}
\begin{proof}
    This follows immediately from the local classification
    Theorem~\ref{thm:special_local}, together with
the fact \cite[72:1]{OMeara}
that the only global obstruction is $\prod \HW_v(L)=1$.
By Corollary~\ref{cor:hasse-eichler}
the left hand side is the product of $\HW_v(L)$ for the
non-archimedean places.
Note that $\HW_v(L)=1$ for complex places and $\HW_v(L)=-1$ for real
places; hence the product of $\HW_v(L)$ for the archimedean places is
$(-1)^{r_1}$. Since $[k:\QQ]\equiv r_1\pmod{2}$ this equals the right
hand side.
\end{proof}

\section{A special global quinary lattice for $\GUtwo{B}$}
\label{section:quadraticform}

Let $D^-=\prod_{p \in S}p$ and let $D^+$ be a positive integer, not
necessarily square-free, but coprime to $D^-$, and let $D:=D^-D^+$.
Let $Q$ be the quadratic form on $\DimSix$ given by
\begin{equation}
  \label{eq:defquadratic}
Q\left( \left(\begin{smallmatrix}s & r \\ \overline{r} &
      t\end{smallmatrix} \right) \right) = \frac{1}{4D}((s-t)^2 + 4
\norm(r)),   
\end{equation}
where $\norm(r)=r\overline{r}$ is the norm of $r$ as an element of the quaternion
algebra.
\begin{remar}
  It follows from its definition that the quadratic form $Q$ is
  invariant under translation by centre elements, i.e.
  $Q(v+\lambda I) = Q(v)$ if $\lambda \in \QQ$.
\end{remar}

If $A  = \left(\begin{smallmatrix} a & b\\ c & d\end{smallmatrix}\right) \in M_2(B)$, let $\adj(A) = \left( \begin{smallmatrix} d & -b\\ -c & a\end{smallmatrix}\right)$ denote its usual adjoint matrix.
\begin{lem}
  If $v \in \DimSix$ the following relation holds
\begin{equation}
  \label{eq:adjoing}
\frac{1}{4D}(v - \adj(v))^2 = \left(\begin{smallmatrix} Q(v) & 0 \\ 0 & Q(v)\end{smallmatrix} \right).
\end{equation}
\label{lemma:quadraticalternatedefinition}
\end{lem}
\begin{proof}
  Follows from an elementary computation.
\end{proof}
  In particular, we can alternatively define the quadratic form via
  \begin{equation}
    \label{eq:normaltdef}
    Q(v) = \frac{1}{8D} \trace((v - \adj(v))^2).    
  \end{equation}
\begin{defi}
  Let $(\DimFive,Q)$ be the quadratic space
  $\DimFive = \DimSix/\Q I$, with the quadratic form $Q$ in the quotient
  space.
\end{defi}
\begin{remar}
  Over $\QQ$ we can take an orthogonal complement for the scalar
  matrices subspace (given by the elements in $\DimSix$ whose trace is
  zero), and the quadratic space
  $\left(\left(\begin{smallmatrix} 1 & 0 \\ 0 & 1\end{smallmatrix}
    \right)^\perp, Q\right)$ (isometric to the space $(\DimFive,Q)$)
  is isometric to the space considered by Ibukiyama in \cite{I}.  In
  this case the quadratic forms becomes
  $Q\left(\begin{smallmatrix}t & \bar{r}\\ r &
      -t\end{smallmatrix}\right) = \frac{1}{D}(t^2+\norm(r))$
  (i.e. $Q(A) = \frac{1}{2D}\trace(A^2)$).  The advantage of working
  with the $6$-dimensional space will prove crucial while working over
  $\Z_2$, cf. Remark \ref{p2trace0}.
\label{remark:tracezero}
\end{remar}
Since $B\otimes\R$ splits over $\C$, $B$ can be embedded in $M_2(\C)$, hence $M_2(B)$ in $M_4(\C)$, so the following is immediate. 
\begin{lem}
\label{lemma:traceinvariance} Let $g \in \GUtwo{B}$ and $v \in U$, then $\trace(g v  g^{-1}) = \trace(v)$.
\end{lem}
Note that $v$, considered as an element of $M_2(B)$, has rational scalar entries on the leading diagonal, so its trace as a $2\times 2$ matrix and its trace as an element of $M_4(\C)$ differ only by a factor of $2$.

\begin{prop}
  The action of $\GUtwo{B}$ on $B$ preserves the quadratic form $Q$, up to
  a factor $\nu(g)$.
\label{proposition:quadforminvariance}
\end{prop}

\begin{proof}
  Note that $v - \adj(v) = 2v - \trace(v)$, hence by (\ref{eq:normaltdef}),
  $8D \, Q(g v g^{-1}) = \trace(2g v g^{-1} -
  \trace(g v g^{-1}))^2$. But
  $(2g v g^{-1} -\trace(g v g^{-1}))^2 =
  g (2v-\trace(v))^2 g^{-1} = g(v-\adj(v))^2g^{-1} = 
  4D g\left(\begin{smallmatrix} Q(v) & 0 \\ 0 &
      Q(v) \end{smallmatrix}\right) g^{-1} = 4D \, Q(v) I$
  and the result follows.
\end{proof}

\subsubsection{The quadratic form at split primes}

If $p$ is a split prime, i.e.  $B_p\simeq M_2(\Q_p)$, with
$r\mapsto \begin{pmatrix}a&b\\c&d\end{pmatrix}$, the quadratic form
(on trace zero elements) $\frac{1}{D}\left(t^2+r\overline{r}\right)$
becomes $\frac{1}{D}\left(t^2+ad-bc\right)$, giving an isomorphism
$\GUtwo{B_p}/\Q_p^{\times}\simeq\SO_5(\Q_p)$, the split special
orthogonal group.

\subsubsection{The quadratic form at non-split primes}
As mentioned earlier, we can identify the
quotient $\DimSix_p/\Q_p I$ with the subspace of trace zero
matrices. If $A\in \DimSix_p$ and
$\tilde{A}=\xi^{-1}A\xi\in\widetilde{\DimSix}_p$,
$\tr(A)=0\iff\tr(\tilde{A})=0$, hence the trace zero elements can be
described by
\[
  \widetilde{\DimFive}_p=\left\{\begin{pmatrix} y & s\\t & \overline{y}\end{pmatrix}:\,t,s\in\Q_p,\,y\in B_p, y+\overline{y}=0\right\}.
  \]
  and the quadratic form becomes
  \[
    Q(A)=\frac{1}{2D}\tr(A^2)=\frac{1}{2D}\tr(\tilde{A}^2)=\frac{1}{D}\left(st+y^2\right).
    \]
This is the quadratic form associated with the non-split special orthogonal group $\SO_5^*(\Q_p)$ (cf. \cite[\S 3]{GR}), and we get isomorphisms
$$\GUtwo{B_p}/\Q_p^{\times}\simeq\GUone{B_p}/\Q_p^{\times}\simeq\SO_5^*(\Q_p).$$

\subsection{A global lattice in $\DimFive$}
\label{section:globallattice}
Let $D = D^- D^+$ as before, and consider the $5$-dimensional
$\Q$-vector space $\DimFive = \DimSix/\Q I$ with the quadratic form $Q$ defined in
(\ref{eq:defquadratic}).

%Recall that the bilinear form attached to a
%quadratic form $q$ equals $\langle v,w\rangle:=q(v+w)-q(v)-q(w)$, so
%$q(v)=\frac{1}{2}\langle v,v\rangle$.
%
%Choosing a basis $\{v_1, \ldots,v_5\}$ for $\DimFive$, the Gram matrix
%$(\langle v_i,v_j\rangle)$ of the bilinear form is the Hessian (matrix
%of second partial derivatives) of the quadratic form $q$, denoted $H(q)$. We will use the term \emph{determinant of $q$} to
%denote the determinant of the Hessian matrix, which does not depend on
%the choice of basis, up to rational squares. If $\{v_1 ,\dots,v_5\}$
%is a basis for an integral lattice $L$ in $\DimFive$ (i.e. all
%$\langle v_i,v_j\rangle \in \ZZ$) then the determinant is an integer
%depending only on $L$, positive if $q$ is positive-definite.
%
%Given a $\Z_p$-lattice $\mathcal{L}$ in
%$\DimFive_p$, its dual $\mathcal{L}^{\vee}$ is defined by
%\[
%  \mathcal{L}^{\vee}:=\{v\in \DimFive_p:\,\,\langle v,w\rangle\in\Z_p\,\,\forall w\in\mat%hcal{L}\}.
%  \]
Recall that by the local-to-global principle, to give a lattice in a
$\Q$-vector space is equivalent to giving it locally for each finite
place $p$.

  \begin{defi}
    Let $\Lattice$ be the $\Z$-lattice in $\DimFive$ whose local completion
    $\Lattice_p:=\Lattice \otimes \Z_p$ at a prime $p$ is as follows.
\begin{enumerate}
\item For $p\nmid D$, $\Lattice_p:=(\UU_{p^0}^+/\Z_p I)^{\vee}$.
\item For $p\mid D^+$, $\Lattice_p:=(\UU^+_{p^n}/\Z_p I)^{\vee}$, where $n= v_p(D^+)$.
\item For $p\mid D^-$, $\Lattice_p:=(\UU^-_p/\Z_p I)^{\vee}$.
\end{enumerate}
\label{definition:locallattice}    
\end{defi}

\begin{prop}\label{prop:integral}
  The lattice $\Lattice$ is integral with respect to the quadratic form $Q$. Furthermore, there exists a basis such that the Hessian matrix of the quinary form is as follows:
  \begin{itemize}
  \item If $p \nmid D^-$, let $n = v_p(D)$, then
    \[
      H(Q) =  2D \perp \frac{D}{p^n}\begin{pmatrix}0 & 1\\ 1 & 0\end{pmatrix} \perp \frac{D}{p^n}\begin{pmatrix}0 & 1\\ 1 & 0\end{pmatrix}.
    \]
    
  \item If $p \mid D^-$ is odd, let $\varepsilon$ be a non-square modulo $p$. Then
    \[
H(Q) =  2D\varepsilon \perp \frac{2D}{p} \perp \frac{-2D\varepsilon}{p} \perp \frac{D}{p} \begin{pmatrix}0 & 1\\ 1 & 0\end{pmatrix}.
\]
\item If $2 \mid D^-$, then
  \[
H(Q) = \frac{-2D}{3} \perp \frac{-D}{2}\begin{pmatrix} 2 & 1\\ 1 & 2\end{pmatrix} \perp \frac{D}{2} \begin{pmatrix} 0 & 1\\ 1 & 0\end{pmatrix}.
    \]
  \end{itemize}
In particular, the determinant of $H(Q)$ equals $2D$.
\end{prop}
\begin{proof}
  We can check this locally prime-by-prime. For the determinant
  statement, since $Q$ is positive-definite and $D>0$, it is enough to
  check that the valuation is correct at each prime $p$.

\begin{enumerate}
\item If $p\nmid D$ the quaternion algebra $B$ is unramified at $p$.
  In the canonical basis
  $\BB:=\left\{ \smat{1&0\\0 & 0},\smat{0 &0\\0 & 1},\smat{0& 1\\0&0},
      \smat{0 & 0\\-1 & 0}\right\}$ the quadratic norm form has Hessian
  matrix $\smat{0&1\\ 1&0} \perp \smat{0 & 1\\ 1& 0}$. Let $n = v_p(D)$. Consider the basis for $\Level_{p^n}$
\[
  \left\{ \left(\begin{matrix}
        I_2 & 0_2 \\
        0_2 & 0_2\end{matrix} \right),
    \left(\begin{matrix} 0_2 & 0_2 \\
        0_2 & I_2\end{matrix} \right),
    \left(\begin{matrix} 0_2 & v\bar{\pi}^n\\
        \pi^n\overline{v} & 0_2\end{matrix} \right) : v \in \BB\right\}.
\]
In particular, a basis for the quotient $\Level_{p^n}/\Z_p I$ is given
by the last five elements.  It is easy to check that in such a basis,
the quadratic form $Q$ has Hessian matrix
%\[
%  \frac{1}{4D}\diag\left(
%    2\left(\begin{smallmatrix} 1 & 1\\1 & 1 \end{smallmatrix} \right)
%    ,4\left(\begin{smallmatrix} 0 & 1\\1 & 0 \end{smallmatrix} \right),
%  4\left(\begin{smallmatrix} 0 & 1\\1 & 0 \end{smallmatrix} \right)\right),
%\]
%while in the quotient, the Hessian matrix on this basis equals the
%diagonal matrix
%\[
%  \frac{1}{D}\,\diag\left(\begin{pmatrix} 0 & 1\\1 & 0\end{pmatrix}, \begin{pmatrix} 0 & 1\\1 & 0\end{pmatrix}, 1/2\right),
%\]
%whose dual equals
%\[
%  D\,\diag\left(\begin{pmatrix} 0 & 1\\1 & 0\end{pmatrix}, \begin{pmatrix} 0 & 1\\1 & 0\end{pmatrix}, 2\right).
%\]
%This proves integrality at $p$ (given $p\nmid D$) and the assertion regarding the determinant.
%\item For $p^n\parallel D^+$ with $n\geq 1$, the proof is very similar
%  to the previous case. Note that we replaced the anti-diagonal
%  elements by multiples of $\bar{\pi}^n$, hence the norm gets
%  multyplied by $p^n$. In particular, the Hessian matrix of
%  $\Level_{p^n}/\Z_p$ equals
\[
  \frac{1}{4D}\left(2 \perp
    4p^n\begin{pmatrix} 0 & 1\\1 & 0 \end{pmatrix} \perp
  4p^n\begin{pmatrix} 0 & 1\\1 & 0 \end{pmatrix}\right).
\]
For $\Lattice_p=(\UU^+_{p^n}/\Z_p)^{\vee}$ a Hessian matrix is then
\begin{equation}
  \label{eq:quadformpmidD+}
2D \perp \frac{D}{p^n}
    \begin{pmatrix} 0 & 1\\1 & 0 \end{pmatrix} \perp
  \frac{D}{p^n}\begin{pmatrix} 0 & 1\\1 & 0 \end{pmatrix}.
\end{equation}
  Integrality at $p$ follows from the fact that
  $D/p^n\in\Z_p^{\times}$.
\item If $p\mid D^-$, the quaternion algebra is ramified at
  $p$. Recall from Lemma~\ref{lemma:quadraticalternatedefinition} that
  if $v \in \DimSix$, $(v - \adj(v))^2$ is a diagonal matrix, hence its
  trace (which gives the quadratic form) is invariant under
  conjugation by $\xi$. In particular, it is enough to understand the lattice
  \[
    \left \{ \left(\begin{matrix} r & ps\\ t & \bar{r}\end{matrix}
      \right) \; : \; s,t \in \Z_p, r \in \OOO_p\right\}.
  \]
  with the quadratic form $\frac{1}{4D}((r-\bar{r})^2+4pst)$. If $p$ is
  odd, then $\OOO_p=\langle 1,\mu,\pi,\mu\pi\rangle_{\Z_p}$, where
  the last three elements have trace zero and satisfy
  $\mu^2=\varepsilon \in\Z_p^{\times}$ (a non-square), $\pi^2=-p$ and
  $\pi\mu=-\mu\pi$. Then the Hessian matrix of the quadratic form
  $\frac{1}{4D}(r -\bar{r})^2$ has diagonal entries (see
  \cite{MR693798} and also Section 5 of \cite{joh})
  \[
    \frac{1}{4D}(0 \perp 8\varepsilon \perp 8p \perp -8p\varepsilon).
  \]
  Then in the basis
  $\left\{ \left(\begin{smallmatrix}\mu & 0 \\ 0 & -\mu\end{smallmatrix}\right),
    \left(\begin{smallmatrix}\pi & 0\\ 0 & -\pi\end{smallmatrix}\right),
    \left(\begin{smallmatrix}\mu \pi & 0\\ 0 & -\mu \pi\end{smallmatrix}\right),\left(\begin{smallmatrix} 0 & 1\\ 0 & 0\end{smallmatrix}
    \right), \left(\begin{smallmatrix} 0 & 0\\ p & 0\end{smallmatrix}
    \right)\right \}$ (a basis for $\UU^-_p/\Z_p I$) the
  quadratic form has matrix
  \[
    \frac{4}{4D}\left(2\varepsilon \perp 2p \perp -2p\varepsilon
      \perp \begin{pmatrix} 0 & p\\p &
          0\end{pmatrix}\right).
    \]
    In particular, its dual lattice has Hessian matrix
    \begin{equation}
      \label{eq:quadformpmidD-}
    2D\varepsilon \perp  \frac{2D}{p}\perp \frac{-2D\varepsilon}{p} \perp \frac{D}{p}\begin{pmatrix} 0 & 1\\1 & 0\end{pmatrix}.
    \end{equation}
    This implies both the integrality and the determinant statement (recall that $v_p(D) = 1$ hence $D/p \in \Z_p^\times$).

    If $p=2$, $B_2$ is the Hamilton $2$-adic quaternion algebra
    ($i^2=j^2=-1$), a basis for $\OOO_2$ is
    $\langle 1, i, j, \frac{1+i+j+k}{2}\rangle$. A better basis for
    the quadratic form $(r - \bar{r})^2/2$ (over $\Z_2$) is
    $\langle 1, \frac{1+i+j+k}{2}, \frac{-1+2i-j-k}{3}, \frac{-1 -i+
      2j -k}{3}\rangle$, where the Gram matrix becomes
    \[
      \frac{4}{4D} \; \left(0 \perp -3/2 \perp\frac{2}{3}
        \left(\begin{matrix} -2 & 1\\1 & -2\end{matrix}\right)\right)
    \]
    Then, the quinary form in the quotient equals
        \[
      \frac{-3}{2D}\perp,-\frac{2}{3D}
      \begin{pmatrix} 2 & -1\\-1 & 2\end{pmatrix} \perp
      \frac{1}{D}
        \begin{pmatrix} 0 & 2\\2 & 0\end{pmatrix},
    \]
    and its dual lattice has Gram matrix
    \[
\frac{-2D}{3} \perp \frac{-D}{2}\begin{pmatrix} 2 & 1\\1 & 2\end{pmatrix} \perp
   \frac{D}{2}\begin{pmatrix} 0 & 1\\1 & 0\end{pmatrix},
    \]
    which is an integral quadratic form, whose Hessian matrix has determinant
    valuation $2$ (since $D/2$ is a unit in $\Z_2$).
 \end{enumerate}
\end{proof}

\begin{prop}\label{HW}
  The lattice $\Lattice$ is special. Furthermore, its Eichler and Hasse-Witt
  invariants are the following:
\begin{enumerate}
\item $\mathrm{HW}_{\infty}(Q)=-1$.
\item $\mathrm{HW}_p(Q)=e(\Lattice_p) =1$ if $p \nmid D^-$.
\item $\mathrm{HW}_p(Q)=e(\Lattice_p)=-1$ if $p\mid D^-$.
\end{enumerate}
\end{prop}
\begin{proof}
%  The quantities of the assertion satisfy the necessary product
%  formula \break $\prod_{v}\mathrm{HW}_v(Q)=1$; the right hand side, from the
%  fact that $D^-$ is the product of an odd number of primes, while on
%  the left hand side it follows from the same property on Hilber
%  symbols (after diagonalizing the form $Q$ over $\Q$). Then it is
%  enough to prove the proposition for places different from $2$.
  By Corollary~\ref{cor:hasse-eichler} it is enough to compute the Eichler
  invariant at each finite place and the Hasse-Witt invariant at the
  infinite place. Note that Proposition~\ref{prop:integral} expresses
  each completion of the lattice at a finite place as an orthogonal
  sum of a rank one lattice and a unimodular one, hence $L$ is special
  by Proposition~\ref{prop:speciallocalglobal}. 
  \begin{enumerate}
\item At the infinity place, (\ref{eq:HWdefi}) gives $\mathrm{HW}_{\infty}(Q)=(-1,-1)_{\infty}(1,1)_{\infty}^{10}=-1$.
\item If $p \nmid D^-$, Proposition~\ref{prop:integral} implies that $e(\Lattice_p) = (p, 1)_p = 1$.
  
\item If $p \mid D^-$ is odd, Proposition~\ref{prop:integral} implies
  that $e(\Lattice_p) = (p, \varepsilon)_p = -1$ while the case $p=2$
  gives $e(\Lattice_2) = (2,-3)_2= -1$.
\end{enumerate}
\end{proof}

\begin{remar}
  \label{remark:identification} Let $(\tilde{V},\tilde{Q})$ be a
  quinary quadratic space, whose quadratic form is positive
  definite. Then by Remark \ref{M2BCliff} the even Clifford algebra $\Cliff_0(\tilde{V})$ is
  isomorphic to $M_2(B)$, where $B$ is a quaternion algebra ramified
  precisely at infinity and the finite primes where $\tilde{Q}$ has
  Hasse-Witt invariant $-1$. The quadratic space $(V,Q)$ constructed
  from $M_2(B)$ then has the same Hasse-Witt invariants as
  $(\tilde{V},\tilde{Q})$ by the last proposition, in particular they
  are isometric, providing an isomorphism
  $\GUtwo{B}/\Q^{\times}\simeq\SO(\tilde{V})$.
\end{remar}

\begin{remar}\label{Ltilde}
  Let $(\tilde{\Lattice},\tilde{Q})$ be a quinary lattice, whose
  quadratic form is positive definite, and is special of determinant
  $2D$. (Note that if $D$ is square-free, $(\tilde{\Lattice},\tilde{Q})$ is automatically special.) Let $S = \{p \; : \; \HW(\tilde{Q})_p = -1\}$ and suppose that
  $v_p(D) = 1$ for all $p \in S$. Let $(\Lattice,Q)$ be the quinary
  lattice of Definition~\ref{definition:locallattice}. Then
  $(\Lattice,Q)$ and $(\tilde{\Lattice},\tilde{Q})$ are in the same
  genus, in particular $\SO(Q) \simeq \SO(\tilde{Q})$.
\end{remar}

\subsection{Radicals}
Let $(q, \Lambda_p)$ be an integral
quadratic form, where $\Lambda_p$ is a $\Z_p$-lattice. 

\begin{defi}
  The \emph{radical} of the form $(q,\Lambda_p)$ equals
  \[
\Rad(q,\Lambda_p):=\left\{v\in \Lambda_p\otimes \Z/2p:\, \langle v,w\rangle \equiv 0 \pmod{2p}\,\,\,\forall w\in \Lambda_p\right\}.
    \]
\end{defi}
\noindent In particular, if $p \neq 2$, $\Rad(q,\Lambda_p)$ is an $\FF_p$-vector
space, while for $p=2$ it is a $\Z/4$-module.  

\begin{lem}
  Let $p$ be a prime number and $(Q,\Lattice_p)$ be as in
  Definition~\ref{definition:locallattice}. If $p \mid D$ then
  $\Rad(Q,\Lattice_p)$ is a $\Z/2p$ lattice of rank $1$.
  \label{lemma:radicalrank}
\end{lem}

\begin{proof}
  Recall from Proposition~\ref{prop:integral} that the quadratic form $Q$
  is equivalent to
  \[
    H(Q)=
    \begin{cases}
      2D \oplus \frac{D}{p^n}\begin{pmatrix}0 & 1\\ 1 &
    0\end{pmatrix} \oplus \frac{D}{p^n}\begin{pmatrix}0 & 1\\ 1 &
    0\end{pmatrix} & \text{ if }p \mid D^+,\\
  2D\varepsilon \oplus \frac{2D}{p} \oplus \frac{-2D\varepsilon}{p} \oplus \frac{D}{p} \begin{pmatrix}0 & 1\\ 1 & 0\end{pmatrix} & \text{ if }p \mid D^-,\; p \neq 2,\\
  \frac{-2D}{3} \oplus \frac{-D}{2}\begin{pmatrix} 2 & 1\\ 1 & 2\end{pmatrix} \oplus \frac{D}{2} \begin{pmatrix} 0 & 1\\ 1 & 0\end{pmatrix} & \text{ if }2 \mid D^-.
\end{cases}
\]
In all cases, the first element of the basis clearly spans the radical.
\end{proof}

\section{Stabilisers of the local lattices \texorpdfstring{$\Lattice_p$}.}
Let us compute for each prime $p$ the stabiliser (under conjugation)
of the lattice $\Lattice_p$ of
Definition~\ref{definition:locallattice}. To easy notation, let us
denote by $\UU_p$ either $\UU_{p^0}^+$, $\UU_{p^n}^+$ or $\UU_p^-$
according to the case.

\begin{lem}
  \label{lemma:stabilizerquotient}
  The stabiliser in $\GUtwo{B_p}$ of the rank five lattice $\Lattice_p$ equals that of the rank six lattice $\UU_p$.
\end{lem}
\begin{proof}
  Since the quadratic form is invariant under the conjugation action of $\GUtwo{B_p}$ 
   (by a local version of 
  Proposition~\ref{proposition:quadforminvariance}) the stabiliser of $\Lattice_p$ is the same as that of its dual lattice $\UU_p/\ZZ_p I$, which is what we shall actually look at.

  The action of $\GUtwo{B_p}$ is trivial at the identity matrix, hence
  if an element stabilises the rank $6$ lattice $\UU_p$ it also stabilises the
  quotient $\UU_p/\ZZ_p I$. To prove the converse, let $g$ be an element stabilising the
  quotient lattice $\UU_p/\ZZ_p I$. Let $v \in \UU_p$ be any vector, so
  $g\bar{v}g^{-1} = \bar{w}$ for some $w$ in $\Lattice_p$. In particular, there exists $\lambda \in \Q_p$ such that
  \begin{equation}
    \label{eq:stabilizerquotient}
g v g^{-1} = w + \lambda \left(\begin{smallmatrix} 1 & 0\\ 0 & 1\end{smallmatrix}\right),    
  \end{equation}
  for some element $w \in \UU_p$ in the preimage of $\bar{w}$. Since
  $v ,w \in \UU_p$, their traces are integral and since
  $\tr(gvg^{-1}) = \tr(v)$ (by Lemma~\ref{lemma:traceinvariance}),
  $2\lambda \in \Z_p$. This gives the statement when $p\neq
  2$. Suppose that $p=2$ and $\lambda \not \in \Z_2$. We can look at
  the ``determinants'' of equation~(\ref{eq:stabilizerquotient}). For
  that purpose, take a quadratic extension of $\Q_2$ that splits the
  quaternion algebra, and take the determinant (as $4 \times 4$
  matrices with coefficients in such an extension). Since all elements
  of $\UU_2$ have integral entries, their determinants are
  integral. Since $\det(AB) = \det(BA)$, $\det(gvg^{-1})$ is integral,
  which is not the case for
  $w + \left(\begin{smallmatrix}\lambda & 0 \\ 0 &
      \lambda\end{smallmatrix}\right)$ (as it corresponds to a
  $4 \times 4$ matrix with integral entries outside the diagonal, but
  with negative valuation at all diagonal elements).
\end{proof}

\begin{prop} Let $p \not \in S$ be an unramified prime.
\begin{enumerate} 
\item The subgroup $K_{0,p}$ of $\GUtwo{B_p}$ preserves the $\Z_p$-lattice $\UU_{p^0}^+\subseteq \DimSix_p$, which was defined by 
  \[
    \UU_{p^0}^+=\left\{\begin{pmatrix} s & \begin{pmatrix}
          a&b\\c&d\end{pmatrix}\\\begin{pmatrix}d&-b\\-c&a\end{pmatrix}
        & t\end{pmatrix}:\,s,t,a,b,c,d\in\Z_p\right\}
  \]
\item In fact, the image of $K_{p^0}^+$ is the full stabiliser of $\UU_{p^0}^+$ in $\GUtwo{B_p}/\Q_p^{\times}$.
\end{enumerate}
\label{K0pstab}
\end{prop}
\begin{proof}
\begin{enumerate}
\item Immediate.
\item Suppose $g\in\GUtwo{B_p}\simeq\GSp_2(\Q_p)$ is such that
  $g\,\UU_{p^0}^+g^{-1}=\UU_{p^0}^+$. By \cite[Lemma 4.1]{I},
  $gM_4(\Z_p)g^{-1}=M_4(\Z_p)$, hence we are led to compute the
  normaliser of $M_4(\Z_p)$. Although the same proof given by Eichler
  to prove the ``Lemma'' (\cite{MR0485698}, page 93 for $M_2(\Z_p)$)
  applies \emph{mutatis mutandis}, we recall the one given by
  Ibukiyama.
  For some sufficiently large $n$,
  $p^ng\in M_4(\Z_p)$, and then $p^ngM_4(\Z_p)=p^nM_4(\Z_p)g$ is a
  two-sided ideal of $M_4(\Z_p)$. As in \cite[Lemma 3.1]{I2},
  necessarily $p^ngM_4(\Z_p)=p^eM_4(\Z_p)$ for some $e\geq
  0$. Equating sets of determinants, $p^{4n}\det(g)\Z_p=p^{4e}\Z_p$,
  so $\det(g)\in p^{4(e-n)}\Z_p^{\times}$, and
  $p^{n-e}g\in\GSp_2(\Q_p)\cap\GL_4(\Z_p)=K_{p^0}^+$, as required.
\end{enumerate}
\end{proof}

Let us state an elementary result.
\begin{lem}
  If $A \in \frac{1}{p^n}\Z_p + M_2(\Z_p)$ 
  has integral determinant, then $A \in M_2(\Z_p)$. Similarly, if 
  $A \in
  \frac{1}{p^n}\Z_p+\begin{pmatrix}\Z_p&p^{-n}\Z_p\\p^n\Z_p&\Z_p\end{pmatrix}$ has integral determinant, then $A \in \begin{pmatrix}\Z_p&p^{-n}\Z_p\\p^n\Z_p&\Z_p\end{pmatrix}$.
  \label{lemma:integrality}
\end{lem}
\begin{proof}
  Suppose on the contrary that
  $A = \smat{\frac{a}{p^r}&b\\ c & \frac{d}{p^r}}$ with
  $a,d \in \Z_p^\times$, $0 < r \le n$. 
%The hypothesis $\tr(A)= \frac{a}{p^r} + \frac{d}{p^s} \in \Z_p$ implies that
%  $r = s$, while 
  The hypothesis $\det(A) = \frac{ad}{p^{2r}}-bc \in \Z_p$ implies
  $2r \leq 0$ getting a contradiction. The other case is similar.
\end{proof}
\begin{prop} Let $p \not \in S$ be an unramified prime.
  \label{K+pstab} 
\begin{enumerate} 
\item The subgroup $K^+_{p^n}$ of $\GUtwo{B_p}$ preserves the
  $\Z_p$-lattice $\UU^+_{p^n}\subseteq \DimSix_p$, which was defined by 
  \[
    \UU^+_{p^n}:=\left\{\begin{pmatrix} s & \begin{pmatrix}
          a&b\\c&d\end{pmatrix}\\\begin{pmatrix}d&-b\\-c&a\end{pmatrix}
        & t\end{pmatrix}:\,s,t,b,d\in\Z_p, a,c\in p^n\Z_p\right\}.
  \]
  So does the Atkin-Lehner element $W_{p^n}^+$.
\item In fact, the subgroup of $\GUtwo{B_p}/\Q_p^{\times}$ generated by the images
  of $K^+_{p^n}$ and $W_{p^n}^+$ is the full stabiliser of
  $\UU^+_{p^n}$.
\end{enumerate}
\end{prop}
\begin{proof}
(1) Recall that 
  \[
   K^+_{p^n}:=\{k\in \GUtwo{B_p}:\,h^{-n}kh^n\in \GL_4(\Z_p)\},
    \]
  where $h:=\diag(1,1,1,p)$.  Given
  $A=\begin{pmatrix} s & \begin{pmatrix}
      a&b\\c&d\end{pmatrix}\\\begin{pmatrix}d&-b\\-c&a\end{pmatrix} &
    t\end{pmatrix}\in \DimSix_p$,
    \begin{multline}
    h^{-n}Ah^n\in M_4(\Z_p)\iff\begin{pmatrix} s &
      \begin{pmatrix}
        a&p^nb\\
        c&p^nd
      \end{pmatrix}\\
      \begin{pmatrix}d&-b\\-p^{-n}c&p^{-n}a\end{pmatrix} & t
    \end{pmatrix}\in M_4(\Z_p) \\
      \iff b,d\in \Z_p,\, a,c\in p^n\Z_p\iff A\in
      \UU^+_{p^n}.
\label{eq:condition}
    \end{multline}
If $k \in K^+_{p^n}$ and $A \in \UU^+_{p^n}$, by
Proposition~\ref{prop:Winvariance}, $kAk^{-1} \in \DimSix_p$. Let
$k=h^nmh^{-n}\in K^+_{p^n}$, with $m\in \GL_4(\Z_p)$, and
$A=h^nm'h^{-n}\in \UU^+_{p^n}$, with $m'\in M_4(\Z_p)$, then
\[
kAk^{-1}=h^nm(h^{-n}(h^nm'h^{-n})h^n)m^{-1}h^{-n}=h^n(mm'm^{-1})h^{-n},
\]
and the latter is an element in $\DimSix_p$ which is in $\UU^+_{p^n}$ by (\ref{eq:condition}).

Recall that
$W_{p^n}^+:=\smat{0&0&p^n&0\\
  0&0&0&1\\
  1&0&0&0\\
  0&p^n&0&0}$. Then if $A \in \UU^+_{p^n}$ is as before,
\[
  W_{p^n}^+A{W_{p^n}^+}^{-1}=
  p^{-n}W_{p^n}^+AW_{p^n}^+=
  \begin{pmatrix}t&
    \begin{pmatrix}p^nd&-b\\-p^nc&a\end{pmatrix}
    \\\begin{pmatrix} a&b\\p^nc&p^nd\end{pmatrix}&s\end{pmatrix}\in \UU^+_{p^n}.
  \]

  \noindent (2) Note that $K^+_{p^n}=\GUtwo{B_p}\cap \tilde{R}^{\times}$,
  where $\tilde{R} = \Psi^{-1}(R)$ (cf. (\ref{eq:eltsparamodular})) given by 
  \[
    \tilde{R}=\left(\begin{smallmatrix}
        \smat{\Z_p&\Z_p\\ \Z_p & \Z_p} & \smat{p^n\Z_p&\Z_p\ \\p^nZ_p&\Z_p}\\
        \smat{\Z_p&\Z_p\\ p^n\Z_p&p^n\Z_p} & \smat{\Z_p&p^{-n}\Z_p\\
          p^n\Z_p & \Z_p}
    \end{smallmatrix}\right).
  \]
  By Lemma \ref{normaliser} it suffices to show that if
  $g\in\GUtwo{B_p}$ satisfies $g^{-1}\,\UU^+_{p^n}g=\UU^+_{p^n}$ then
  $g^{-1}\,\tilde{R}g=\tilde{R}$. Switching $g$ and $g^{-1}$ to get the reverse
  inclusion, it suffices to show that $g^{-1}\,\tilde{R}g\subseteq \tilde{R}$.

By Lemma \ref{order}, the minimal order containing $\UU^+_{p^n}$ equals
\[
  R'=\smat{
    \Z_pI_2&0_2\\0_2&\Z_pI_2}
  \oplus\left(\begin{smallmatrix}
      \smat{p^n\Z_p&p^n\Z_p\\p^n\Z_p&p^n\Z_p} &
      \smat{p^n\Z_p&\Z_p \\ p^n\Z_p&\Z_p}\\
      \smat{Z_p&\Z_p\\ p^n\Z_p&p^n\Z_p} &
      \smat{p^n\Z_p&\Z_p \\ p^{2n}\Z_p&p^n\Z_p}.
    \end{smallmatrix}\right)
  \]
  In particular, if $g^{-1}\,\UU^+_{p^n}g=\UU^+_{p^n}$ then
  $g^{-1}\,R'g=R'$.  $\tilde{R}$ is generated (as a $\Z_p$-module) by $R'$ and
  by elements of the form $\smat{M&0_2\\0_2&0_2}$ or
  $\smat{0_2&0_2\\0_2&M'}$, with $M\in M_2(\Z_p)$ and
  $M'\in \smat{\Z_p&p^{-n}\Z_p\\p^n\Z_p&\Z_p}$. The Atkin-Lehner
  operator also fixes $R'$, and since conjugating by $W_{p^n}^+$ a
  general element of the first form gives one of the second, it
  suffices to show that if $M \in M_2(\Z_p)$, $g^{-1}\,\smat{M&0_2\\0_2&0_2}g\in \tilde{R}$.

  Write $g=\smat{A&B\\C&D}$, with $A,B,C,D\in M_2(\Q_p)$.  From
  $g^*\,g=\nu(g)\,I$, we
  get
  \[
    g^{-1}=\frac{1}{\nu(g)}\,g^*=\frac{1}{\nu(g)}\,
    \begin{pmatrix}\overline{A}&\overline{C}\\
      \overline{B}&\overline{D}\end{pmatrix}.
    \]
    Recall that if $A=\smat{a&b\\c&d}$ then
    $\overline{A}=\smat{d&-b\\-c&a}$, and
    $A\overline{A}=\overline{A}A=\det A$.  Now, 
    \begin{equation}
      \label{eq:relation1}
    g^{-1}\,
    \begin{pmatrix}
      M&0_2\\
      0_2&0_2
    \end{pmatrix}g=
    \frac{1}{\nu(g)}
    \begin{pmatrix}
      \overline{A}MA&\overline{A}MB\\
      \overline{B}MA&\overline{B}MB 
    \end{pmatrix}. 
  \end{equation}
  In the particular case $M = I_2$, since $\smat{I_2& 0 \\ 0 & 0} \in R'$, looking
  at the top left and bottom right blocks we find that
  $\frac{\det A}{\nu(g)}, \frac{\det B}{\nu(g)}\in \Z_p$. 

  Since $\smat{p^nM&0_2\\0_2&0_2}\in R'$, its conjugate is also in
  $R'$, hence the relation (\ref{eq:relation1}) implies that
  $\frac{1}{\nu(g)}\overline{A}p^nMA\in \Z_p+p^nM_2(\Z_p),$ so
  $\frac{1}{\nu(g)}\overline{A}MA\in \frac{1}{p^n}\Z_p+M_2(\Z_p).$ On
  the other hand,
%  $\tr\left(\frac{1}{\nu(g)}\overline{A}MA\right)=\tr(A^{-1}MA)=\tr(M)\in\Z_p$,
%  and
  $\det\left(\frac{1}{\nu(g)}\overline{A}MA\right)=\frac{\det A}{\nu(g)}\det(M)\in\Z_p$,
  hence Lemma~\ref{lemma:integrality} implies that
  $\frac{1}{\nu(g)}\overline{A}MA\in M_2(\Z_p)$.

  Similarly,
  $\frac{1}{\nu(g)}\overline{B}MB\in
  \frac{1}{p^n}\Z_p+\begin{pmatrix}\Z_p&p^{-n}\Z_p\\p^n\Z_p&\Z_p\end{pmatrix}$. The same determinant computation
  combined with Lemma~\ref{lemma:integrality} implies that
  $\frac{1}{\nu(g)}\overline{B}MB\in \begin{pmatrix}\Z_p&p^{-n}\Z_p\\p^n\Z_p&\Z_p\end{pmatrix}$. 

  Finally,
  $$\frac{1}{\nu(g)}\overline{A}MB=\frac{\overline{A}MA}{\nu(g)}\,\frac{\nu(g)}{\det
    A}\,\frac{\overline{A}B}{\nu(g)}.$$ We have already shown that the
  first factor is in $M_2(\Z_p)$, the second factor is in $\Z_p$ since
  $\frac{\det A}{\nu(g)}\in\Z_p^{\times}$, and the third factor is in
  $\begin{pmatrix}p^n\Z_p&\Z_p\\p^n\Z_p&\Z_p\end{pmatrix}$, using the
  special case $M=I_2$. Hence
  $\frac{1}{\nu(g)}\overline{A}MB\in\begin{pmatrix}p^n\Z_p&\Z_p\\p^n\Z_p&\Z_p\end{pmatrix}$,
  and similarly
  $\frac{1}{\nu(g)}\overline{B}MA\in\begin{pmatrix}\Z_p&\Z_p\\p^n\Z_p&p^n\Z_p\end{pmatrix}$.
\end{proof}

\begin{prop}\label{K-pstab} 
  Let $p \in S$ be a ramified prime.
  \begin{enumerate}
  \item The subgroup $K^-_p$ of $\GUtwo{B_p}$ preserves the
    $\Z_p$-lattice $\UU^-_p\subseteq \DimSix_p$, which was defined
    by $\UU^-_p=\xi\,\tilde{\UU}^-_p\xi^{-1}$, where
    \[
      \tilde{\UU}^-_p:=
      \left\{\begin{pmatrix} r&ps\\
          t&\overline{r}\end{pmatrix}:\,r\in\OOO_p,\, s,t\in\Z_p\right\}.
      \]
So does the Atkin-Lehner element $\omega_p$.
\item The subgroup of $\GUtwo{B_p}/\Q_p^{\times}$ generated by the images of $K^-_p$ and $\omega_p$ is the full stabiliser of $\UU^-_p$.
\end{enumerate}
\end{prop}
\begin{proof} 
  (1) Let $\tilde{h}:=\smat{\pi&0\\0 & 1}$, where $\pi^2=-p$. Recall
  that $K^-(p)$ is the stabiliser (via left multiplication) of the
  lattice $\id{p} \oplus \OOO_p = \tilde{h} \OOO_p^2$. Then
  $k \in K^-(p) \iff \tilde{h}^{-1} k \tilde{h} \in M_2(\OOO_p)$.
Given $\tilde{A}=\smat{r&s\\t&\overline{r}}\in \xi U_p\xi^{-1}$,
\begin{multline}
  \tilde{h}^{-1}\tilde{A}\tilde{h}\in M_2(\OOO_p)\iff
  \begin{pmatrix} \pi^{-1} r \pi&\pi^{-1} s\\t \pi&\overline{r}\end{pmatrix}\in M_2(\OOO_p)\\
  \iff r \in \OOO_p, s\in p\Z_p,\,t\in \Z_p\iff \tilde{A}\in \tilde{\UU}^-_p.
  \end{multline}
Now given $k\in K^-(p)$, $k=\tilde{h}m\tilde{h}^{-1}$ with $m\in \GL_2(\OOO_p)$, hence
\[
  k\tilde{A} k^{-1}=\tilde{h}m\tilde{h}^{-1}\tilde{A}\tilde{h}m^{-1}\tilde{h}^{-1}.
\]
If $\tilde{A}\in \tilde{\UU}^-_p$ then
$\tilde{h}^{-1}\tilde{A}\tilde{h}=m'\in M_2(\OOO_p)$, and
$\tilde{h}^{-1}(k\tilde{A}k^{-1})\tilde{h}=mm'm^{-1}\in M_2(\OOO_p)$,
so $k\tilde{A}k^{-1}\in \tilde{\UU}^-_p$, as required. Also,
\[
  \omega_p\tilde{A}\omega_p^{-1}=p^{-1}\omega_p\tilde{A}\omega_p=\begin{pmatrix} \overline{r} & pt\\s/p & r\end{pmatrix}\in\tilde{\UU}^-_p.
  \]

  \noindent (2) Suppose that $g\in\GUone{B_p}$ is such that
  $g\tilde{\UU}^-_pg^{-1}=\tilde{\UU}^-_p$. By Lemma~\ref{lemma:orderramifiedprime}, the element $g$ also normalises the order
  \[
    \left\{ \begin{pmatrix} a & b \\ c & d\end{pmatrix}
      \in \begin{pmatrix} \mxlram & p\mxlram \\ \mxlram &
        \mxlram\end{pmatrix} \; : \; a \equiv \bar{d}
      \pmod{\pi}\right\}.
    \]
    By \cite[Lemma 4.3]{I}, then $gRg^{-1}=R$, where
    $R=\begin{pmatrix}\OOO_p & \id{p}\\\id{p}^{-1} &
      \OOO_p\end{pmatrix}$ is the left order of $\id{p}\oplus\OOO_p$
    (this is precisely the statement of \cite[Corollary 4.4]{I}, which
    is only stated for odd primes).

  Looking at the form of $R$, clearly for some
  sufficiently large $n$, $p^ng\in R$, and then $p^ngR=p^nRg$ is a
  two-sided ideal of $R$. As in \cite[Proposition 3.2]{I2},
  $p^ngR=\omega_p^eR$ for some $e\geq 0$. Recalling that
  $\omega_p^2=pI$, $\omega_p^{2n-e}gR=R$, so
  $\omega_p^{2n-e}g\in R^{\times}$. Letting
  $m=\lfloor{\frac{2n-e+1}{2}}\rfloor$, either
  $p^mg\in R^{\times}\cap\GUone{B_p}=K^-(p)$, or
  $p^m\omega_p^{-1}g\in R^{\times}\cap\GUone{B_p}=K^-(p)$,
  as required.
    
%
%
%  Looking at the form of $R$, clearly
%  for some sufficiently large $n$, $p^ng\in R$, and then $p^ngR=p^nRg$
%  is a two-sided ideal of $R$. As in \cite[Proposition 3.2]{I2},
%  $p^ngR=\omega_p^eR$ for some $e\geq 0$. Recalling that
%  $\omega^2=pI$, $\omega_p^{2n-e}gR=R$, so
%  $\omega_p^{2n-e}g\in R^{\times}$. Letting
%  $m=\lfloor{\frac{2n-e+1}{2}}\rfloor$, either
%  $p^mg\in R^{\times}\cap\GU(1,1)(B\otimes\Q_p)=K_2(p)$, or
%  $p^m\omega_p^{-1}g\in R^{\times}\cap\GU(1,1)(B\otimes\Q_p)=K_2(p)$,
%  as required.
\end{proof}
\subsection{Paramodular subgroups as kernels of sign characters}

Let
$\mathrm{Stab}(\Lattice_p):=\{g\in \GUtwo{B_p}:\,gvg^{-1}=v\,\,\,\forall v\in
\Lattice_p\}$ be the stabiliser of $\Lattice_p$. If $g\in\mathrm{Stab}(\Lattice_p)$ then
the action of $g$ preserves $\mathrm{Rad}(Q,\Lattice_p)$. Suppose that
$p \mid D$, and let $v_0$ denote a generator of $\mathrm{Rad}(Q,\Lattice_p)$
as $\Z/2p$-module (it has rank $1$ by Lemma~\ref{lemma:radicalrank}). Since
$Q(gv_0g^{-1})=Q(v_0)$, we must have
$gv_0g^{-1}\equiv\pm v_0 \pmod{2p}$.

\begin{defi}
  Let $p$ be a prime dividing $D$. Define a homomorphism
$\theta_{p}:\mathrm{Stab}(\Lattice_p)\rightarrow \{\pm 1\}$ by
\[
  gv_0g^{-1}\equiv\theta_{p}(g) v_0 \pmod{2p}.
\]
\label{defi:character}
\end{defi}
\noindent By Propositions~\ref{K-pstab} and \ref{K+pstab}, if $p\mid D$ then 
\[
  \Stab(\Lattice_p) =
  \begin{cases}
    \langle p^\ZZ, \omega_p,K^-_p\rangle & \text{ if } p\mid D^-,\\
    \langle p^\ZZ,W^+_{p^n},K^+_{p^n}\rangle & \text{ if }p^n\parallel D^+.
    \end{cases}
\]

  \begin{prop}
    \label{prop:thetapn}
Let $p$ be an odd prime such that $p \mid D$. Then:
    \begin{enumerate}
\item $\theta_{p}(p^{\Z})=\{1\}$.
\item If $p \mid D^-$, $\theta_p(K^-_p) = \{1\}$ and $\theta_{p}(\omega_p)=-1$.
\item If $p^n \parallel D^+$, $\theta_{p}(K^+_{p^n})=\{1\}$ and $\theta_p(W^+_{p^n})=-1$.
  
%\item If $4 \mid D^+$, $\theta_{2}(K^+_{2^n})=\{1\}$ and $\theta_2(W_{2^n})=-1$.
\end{enumerate}
\end{prop}
\begin{proof}
  (1) Immediate.

\vspace{.1cm}

\noindent(2) From the proof of Proposition~\ref{prop:integral}, it follows
that if $p \neq 2$, then the generator of $\Rad(Q,\Lattice_p)$ is the element
$\smat{p\mu & 0\\ 0 & -p\mu}$ (we are taking the dual of the
basis described in the proof of Proposition~\ref{prop:integral}), where $\tr(\mu)=0$ and $\mu^2=\varepsilon$. Recall
that $\OOO_p=\langle 1,\mu, \pi, \pi\mu\rangle$, where
$\pi \mu = -\mu \pi$, $\pi^2=-p$. If
$g = \smat{a & \pi^{-1}b \\ \pi c & d} \in K^-_p$, then
$g^{-1} = \frac{1}{\nu(g)}\smat{\bar{d} & \bar{b}\bar{\pi}^{-1} \\
  \bar{c}\bar{\pi} &\bar{a}}$ (where $\nu(g)$ is a unit). In
particular, $\bar{d}a -\bar{b}c = \nu(g)$ (as
$\bar{\pi}^{-1} = -\pi^{-1}$). Then
\begin{equation}
  \label{eq:character}
  \frac{p}{\nu(g)}
  \begin{pmatrix} \bar{d} & \bar{b} \overline{\pi}^{-1} \\
    \bar{c} \overline{\pi} & \overline{a}
  \end{pmatrix}
  \begin{pmatrix} \mu & 0\\
    0& \bar{\mu}
  \end{pmatrix}
  \begin{pmatrix}
    a & \pi^{-1}b \\
    \pi  c & d
  \end{pmatrix}
=
\frac{p}{\nu(g)}\begin{pmatrix}
  \bar{d}\mu a + \bar{b} \overline{\pi}^{-1}\bar{\mu}\pi c & *
  \\ * & *
\end{pmatrix}.
\end{equation}

The $(1,1)$ entry equals then $\frac{p}{\nu(g)}(\bar{d}\mu a - \bar{b}\mu c)$ (recall
that $\bar{\mu}=-\mu$ and that $\mu$ and $\pi$ anti-commute), and
since $\OOO_p/(\pi)$ is commutative (generated by $\{1,\mu\}$), it is
congruent to $p\mu \; \frac{\bar{d}a-\bar{b}c}{\nu(g)} = p\mu$ modulo
$\pi p$. Looking at this one entry modulo $\pi p$ is enough to distinguish the cases $gv_0g^{-1}\equiv \pm v_0\pmod{2pL_p}$.

For the Atkin-Lehner statement,
\[
  \begin{pmatrix}
    0 & 1\\
    p & 0
  \end{pmatrix}
  \begin{pmatrix}
    p\mu & 0\\
    0 & p\bar{\mu}
  \end{pmatrix}
  \begin{pmatrix}
    0 & 1/p\\
    1 & 0
  \end{pmatrix} =
  \begin{pmatrix}
    p\bar{\mu} & 0\\
    0 & p\mu
  \end{pmatrix}=
-  \begin{pmatrix}
    p\mu & 0\\
    0 & -p{\mu}
  \end{pmatrix}.
\]

For $p=2$ (the $(-1,-1)$ algebra), let
$\mu = 2D \left(\frac{1+i+j+k}{2}\right)$ and let $\pi = i+k$, so that
$R_2 = \langle 1, \mu, \pi, \pi \mu\rangle$ and $\pi$ generates the
maximal ideal. From the proof of Proposition~\ref{prop:integral} it
follows that once again the generator of $\Rad(Q,L_2)$ is the element
$\left(\begin{smallmatrix} 4\mu & 0 \\ 0 &
    -4\mu\end{smallmatrix}\right)$.
Looking at the $(1,1)$ entry of (\ref{eq:character}),
we are left to prove that
\begin{equation}
  \label{eq:p2char}
\frac{\bar{d} \mu a + \bar{b} \bar{\pi}^{-1}\bar{\mu}\pi c}{\nu(g)} \equiv \mu \pmod{4}.  
\end{equation}
Recall that the maximal order of $B_2$ equals
$\mxlramtwo = \langle 1, \frac{1+i+j+k}{2}, \frac{-1+2i-j-k}{3},
\frac{-1-i+2j-k}{3}\rangle$ (so $\mu$ equals $2D$ times the second
element). Note that the left hand side of (\ref{eq:p2char}) lies in
$4\mxlramtwo$ so in particular, it can be written as
$4\alpha + \beta \mu + 4\gamma e_3 + 4\delta e_4$, for some
$\alpha,\beta,\gamma,\delta \in \Z_2$, where $e_3, e_4$ denote the
third and fourth elements of the generators of $\mxlramtwo$. In
particular, it is congruent to $\beta \mu$ modulo $4\Lattice_2$ (as
expected).

It is easy to check that $\bar{\pi}^{-1} \bar{\mu}\pi = -2D + \frac{\mu}{3} - 2De_4$, hence the left hand side of (\ref{eq:p2char}) is congruent to $\frac{\bar{d}\mu a - \bar{b} \mu c}{\nu(g)}$ modulo $4\Lattice_2$.

On the other hand, since $\mxlramtwo /2$ is commutative (as can easily
be verified), $\frac{\bar{d}\mu a - \bar{b} \mu c}{\nu(g)}$ is
congruent to $\mu$ modulo $8\mxlramtwo$ (recall that $g^{-1}g=1$ together with $\bar{\pi}=-\pi$
implies that $\bar{d}a+\bar{b} c = \nu(g)$). In particular, 
\[
\frac{\bar{d}\mu a - \bar{b} \mu c}{\nu(g)} = 8\tilde{\alpha} + \tilde{\beta}\mu +8\tilde{\gamma}e_3 + 8\tilde{\delta} e_4.
  \]
Since $\tr(1) \equiv \tr(e_3) \equiv \tr(e_4) \equiv 0 \pmod{2}$, the right hand side of the previous line has trace congruent to $2D \tilde{\beta} \pmod{16}$.
Using the fact that in a quaternion algebra,
$\tr(\alpha \beta) = \tr(\beta \alpha)$, the left hand side has trace equal to
\[
\tr\left(\frac{(\bar{d}a - \bar{b}c)}{\nu(g)}\mu\right) = \tr(\mu) = 2D.
\]
Since $v_2(D)=1$, we get that $\tilde{\beta} \equiv 1 \pmod 4$ as stated. The Atkin-Lehner statement follows from the same proof of the odd prime case, noting that (once again) $\bar{\mu} \equiv -\mu$ modulo $4\Lattice_2$.
%and
%$\tr(\alpha) = \tr(\overline{\alpha})$, so the entry $(1,1)$ of
%(\ref{eq:character}) gives
%\[
%\frac{1}{\nu(g)}\tr(\bar{d}\mu a + \bar{b} \overline{\pi}^{-1}\bar{\mu}\pi c) = \frac{1}{\nu(g)}\tr((d \bar{a}+\pi c\bar{b}\bar{\pi}^{-1})\bar{u}) = \tr(\bar{u})=1,
%\]
%since $g^{-1}g = 1$ implies that $(d \bar{a}+\pi c\bar{b}\bar{\pi}^{-1})=\nu(g)$.
%

\noindent (3) Arguing as in the previous case, $v_0:=2D\cdot \smat{I_2 & 0_2\\ 0_2 & 0_2}$ is a generator for the radical. Given
$g=\smat{a&b\\c&d}\in K_{p^n}^+ \subset \GUtwo{B_p}\simeq\GSp_2(\Q_p)$,
\[
  gv_0g^{-1}=
  \frac{2D}{\nu(g)}
  \begin{pmatrix}a&b\\
    c&d
  \end{pmatrix}
  \begin{pmatrix}I_2&0_2\\
    0_2&0_2\end{pmatrix}
  \begin{pmatrix}
    \overline{a}&\overline{c}\\
    \overline{b}&\overline{d}
  \end{pmatrix}=
  \frac{2D}{\nu(g)}\begin{pmatrix}
    a \overline{a}&a \overline{c}\\
    c\overline{a}&c \overline{c}\end{pmatrix}.
  \]
The condition
\[
  \frac{1}{\nu(g)}\begin{pmatrix}a&b\\
    c&d\end{pmatrix}\begin{pmatrix}
    \overline{a}&\overline{c}\\
    \overline{b}&\overline{d}\end{pmatrix}=I_4,
\]
implies that $a \bar{a} = \nu(g) -b \bar{b}$ and $c \bar{c} = \nu(g)-d \bar{d}$. Then
\begin{multline*}
  gv_0g^{-1} - v_0= 
  \frac{2D}{\nu(g)}
  \begin{pmatrix}
    \nu(g) - b \overline{b} - \nu(g)  & a\overline{c}\\
    c\overline{a}&\nu(g) - d \bar{d} \end{pmatrix} = \\ 
=  \frac{2D}{\nu(g)}(\nu(g)- d \bar{d})\cdot \begin{pmatrix} I_2 & 0\\ 0 & I_2\end{pmatrix}
  -\frac{2D}{\nu(g)}
      \begin{pmatrix}
        (\nu(g)+b\bar{b}-d\bar{d})I_2&a\overline{c}\\
        c\overline{a}&0_2
      \end{pmatrix}.
\end{multline*}
    The first term is zero in $\DimSix_p/\Q_p I$, hence it is
    enough to check that
    $\nu(g)+b\bar{b}-d\bar{d} \equiv 0 \pmod{2p}$. The equivalence
    $g^{-1} g = I_4$ implies that
    $\nu(g) -b \bar{b} - d \bar{d} = 0$, hence it is enough to prove that $2b \overline{b} \equiv 0 \pmod{2p}$. The fact
$g\in K^+_{p^n}\implies
b\in\begin{pmatrix}p^n\Z_p&\Z_p\\p^n\Z_p&\Z_p\end{pmatrix}$, so
$b\overline{b}=\det b\in p^n\Z_p$ hence the statement. 
Regarding the Atkin-Lehner statement,
\begin{eqnarray*}
W^+_{p^n}\frac{v_0}{2D}{W^+_{p^n}}^{-1}& = &
\begin{pmatrix}0&0&p^n&0\\
  0&0&0&1\\1&0&0&0\\
  0&p^n&0&0
\end{pmatrix}
\begin{pmatrix}
I_2&0_2\\
0_2&0_2
\end{pmatrix}
\begin{pmatrix}
  0&0&1&0\\
  0&0&0&p^{-n}\\
  p^{-n}&0&0&0\\
  0&1&0&0
\end{pmatrix}\\
 &=&\begin{pmatrix}
   0&0&0&0\\
   0&0&0&0\\
   1&0&0&0\\
   0&p^n&0&0
 \end{pmatrix}
\begin{pmatrix}
  0&0&1&0\\
  0&0&0&p^{-n}\\
  p^{-n}&0&0&0\\
  0&1&0&0\end{pmatrix}
         =\begin{pmatrix}0_2&0_2\\0_2&I_2\end{pmatrix} \simeq -\frac{v_0}{2D},
\end{eqnarray*}                                       
where the last statement comes from the fact that our lattice is a
quotient by scalar matrices. Then $\theta_{p^n}(W^+_{p^n})=-1$.
\end{proof}
\begin{remar} Our lattice $\Lattice_p$ with quadratic form $Q$ is equivalent
  (up to unit scaling of the form) to Gross's $\Lambda(N)$ with his
  $\langle,\rangle/N$ \cite[\S 5]{G1}, to Tsai's $\mathbb{L}_m$ with
  $\langle,\rangle_m$ \cite[Definition 7.1.1]{Ts}, and to
  Lachaus\'ee's $L$ with quadratic form $q'$ \cite[Definition 3.4.1,
  \S 3.5]{L}. Our $v_0$ is Gross's $Nc$, Tsai's $\pp^mv_0$ and
  Lachaus\'ee's $v_0'$. Remarque 2 after
  \cite[Proposition-D\'efinition 3.5.1]{L} shows that our
  $\theta_{p^n}$ (when descended to
  $\GUtwo{B_p}/\Q_p^{\times}\simeq \SO_5(\Q_p)$) is the same
  as his $\alpha$, with kernel his $J^+$ (case $n=1$), or the
  character considered for any $n\geq 1$ by Tsai \cite[Definition
  7.1.2]{Ts}. All these authors work more generally with
  $\SO_{2m+1}(\Q_p)$ for any $m\geq 1$, where the kernel of
  $\theta_{p^n}$ on $\SO_{2m+1}(\Q_p)$ is Brumer's generalisation of
  $\Gamma_0(p^n)$ for $m=1$ and paramodular subgroups for $m=2$
  \cite{BK}.
\end{remar}
\begin{remar}\label{properproof}
  A clear statement of Proposition \ref{K+pstab}(2) is given in
  \cite[\S 6.2, p.81]{Ts}, and some combination of Propositions
  \ref{K+pstab}(2) and \ref{prop:thetapn}(3) is stated in \cite[\S 5
  ``When $n=2$'']{G1}. As far as we know, the detailed proofs we
  present here are the first in the literature.
\end{remar}

\begin{remar}\label{snorm}
  As explained in \S~\ref{section:spingroup}, we may reconstruct
  $M_2(B_p)$ as the even part of the Clifford algebra of the quadratic
  space $\DimFive_p$.  The multiplier character
  $\nu:\GUtwo{B_p}\rightarrow\Q_p^{\times}$ (well-defined modulo
  squares on $\GUtwo{B_p}/\Q_p^{\times}$) becomes the spinor norm
  $\nu:\SO_5(\Q_p)\rightarrow\Q_p^{\times}/(\Q_p^{\times})^2$ or
  $\nu:\SO_5^*(\Q_p)\rightarrow\Q_p^{\times}/(\Q_p^{\times})^2$.  Let
  $\chi:\Q_p^{\times}\rightarrow\{\pm 1\}$ be the unramified
  character.  Then $\chi\circ\nu$ acts as $+1$ on $K^+_{p^n}$
  ($p^n\parallel D^+$) or $K^-_p$ ($p\mid D^-$), on which $\nu$ takes
  unit values. Its acts as $-1$ on the Atkin-Lehner elements $W_{p^n}$
  (when $n$ is odd) and $\omega_p$, on which we have seen $\nu$ takes
  values $p^n$, $p$ respectively. So when $p^n\parallel D^+$ with $n$
  {\em odd}, $\theta_p$ agrees with $\chi\circ\nu$.  But beware
  that when $p^n\parallel D^+$ with $n$ {\em even}, $\theta_p$ and
  $\chi\circ\nu$ are not the same, as the latter is the trivial character.
\end{remar}

\section{An isomorphism of algebraic modular forms for $\GUtwo{B}$ and $\SO(V)$.}
\subsection{Finite-dimensional complex representations of
\texorpdfstring{$\GSp_2(\CC)$}{GSp₂(C)} and
\texorpdfstring{$\SO_5(\C)$}{SO₅(C)}.}\label{repns}
Whereas $B\otimes\R\not\simeq M_2(\R)$, $B\otimes\C\simeq M_2(\C)$, so
the situation is like for $B\otimes\Q_p$ with $p\nmid D^-$. Thus
$\GUtwo{B\otimes\C}\simeq \GSp_2(\C)$
and
\[
  \DimFive\otimes\C\simeq\left\{\begin{pmatrix} t & \begin{pmatrix}
      d&-b\\-c&a\end{pmatrix}\\\begin{pmatrix}a&b\\c&d\end{pmatrix} &
    -t\end{pmatrix}:\,t,a,b,c,d\in\C\right\},
\]
with
$Q(A)=\frac{1}{D}(t^2+ad-bc)$.  Let $W$ be the natural $4$-dimensional
complex representation of $\GUtwo{B\otimes\C}\simeq \GSp_2(\C)$
. Given positive integers $a \ge b \ge 0$, let $n=a+b$. Consider the
representation $W^{\otimes(n)}$ (the tensor product
representation). The permutation group $S_{n}$ acts on
$W^{\otimes(n)}$ and its action commutes with that of
$\GSp_2(\C)$. The Young diagram
\begin{equation}
  \begin{ytableau}
    1 & 2 & \cdots & b & \cdots & a\\
    1 & 2 & \cdots & b
  \end{ytableau}
  \label{tableau:GSP4}
\end{equation}
gives rise to a idempotent on the complex group algebra of $S_{n}$
(as explained for example in \cite[Section 4.1]{FH}). More
concretely, enumerate the squares of \eqref{tableau:GSP4} (from $1$ to
$a$ in the first row and from $a+1$ to $n$ in the second one). The
diagram \eqref{tableau:GSP4} corresponds to the partition $\lambda: n = a+b$. Let
$P$ be the subgroup of $S_{n}$ preserving the rows of
\eqref{tableau:GSP4} (via our enumeration) and $Q$ that
preserving the columns. Define two elements:
\[
a_\lambda = \sum_{\sigma \in P} e_{\sigma} \qquad \text{and} \qquad b_\lambda = \sum_{\tau \in Q} \mathrm{sgn}(\tau) e_{\tau}.
  \]
  Then define $c_\lambda : = a_\lambda \cdot b_\lambda$, and let
  $\SSS_\lambda$ denote the subspace $c_\lambda W^{\otimes(n)}$. Such a
  representation is not in general irreducible. For each pair of
  indices $I = \{p,q\}$ with $1 \le p < q \le n$, let
  $\Phi_I : W^{\otimes(n)} \to W^{\otimes(n-2)}$ be the linear map
  \begin{equation}
    \label{eq:reduction}
\Phi_I(v_1 \otimes \cdots \otimes v_n) = E(v_p,v_q) v_1 \otimes \cdots \otimes \widehat{v_p} \otimes \cdots \otimes \widehat{v_q} \otimes \cdots \otimes v_n,
  \end{equation}
where $\widehat{v_i}$ means remove such term from the tensor product, and $E(v,w)$ is the symplectic form preserved by $\Sp_2$. Let $W^{\langle n\rangle}$ denote the intersection of the kernels of all the $\Phi_I$ and $\SSS_{\langle \lambda \rangle} = W^{\langle n\rangle} \cap \SSS_\lambda$.
\begin{thm}\label{Slambda}
  The representation $\SSS_{\langle \lambda \rangle}$ of $\GSp_2(\C)$ is irreducible.
  The standard maximal torus acts on a highest weight vector via the character
  $$\diag(t_1,t_0t_1^{-1},t_2,t_0t_2^{-1})\mapsto t_1^at_2^b.$$ Furthermore, its dimension equals
  \[
\dim(\SSS_{\langle \lambda \rangle}) = \left(\frac{(a+2)^2-(b+1)^2}{3}\right) \cdot \frac{(a+2)}{2} \cdot(b+1).
    \]
  \end{thm}
  \begin{proof}
    See Theorem 17.11 and Exercise 24.17 of \cite{FH}, which shows that in fact the restriction to $\Sp_2(\C)$ is irreducible.
  \end{proof}
  Note that the central character of $\SSS_{\langle \lambda \rangle}$ is $z\mapsto z^{a+b}=z^n$, as it must be, inside $W^{\otimes(n)}$. The similitude character restricts to the torus as $$\diag(t_1,t_0t_1^{-1},t_2,t_0t_2^{-1})\mapsto t_0.$$
  For $z=zI_4$ (with $z\in\C^{\times}$), $t_1=t_2=z$ and $t_0=z^2$, so $\nu(zI_4)=z^2$.
  
In the case $a\equiv b\pmod{2}$, the representation of $\GSp_2(\C)$ on $$W_{a,b}:=\SSS_{\langle \lambda \rangle}\otimes\nu^{-\frac{a+b}{2}}$$ has trivial central character, so descends to an irreducible representation of $$\GSp_2(\C)/\C^{\times}\simeq\SO_5(\C),$$ which may be extended to $\OO_5(\C)$ by letting $-I_5$ act trivially. (We shall not consider the other extension, where $-I_5$ acts by $-1$.)

The trivial representation is $W_{0,0}$, the original $W$ is $W_{1,0}$, while $W_{1,1}$ is the irreducible $5$-dimensional representation of $\OO_5(\C)$. We constructed this as $V$ (rather $V\otimes\C$) inside the matrix space $\Hom(W,W)\simeq W^*\otimes W$, on which the natural action of $\GSp_2(\C)$ is via $g^{-1}$ on the domain, $g$ on the codomain, i.e. our conjugation. Via the symplectic form, $W^*\simeq W\otimes\nu^{-1}$, so we see $V\otimes\C$ inside $W^{\otimes(2)}\otimes\nu^{-1}$. In fact the $6$-dimensional anti-symmetric part $(\bigwedge^2 W)\otimes\nu^{-1}$ is a direct sum of $\C\,I_4$ and $V\otimes\C$.
  
  All representations of $\SO_5(\CC)$ come from Young diagrams and tensor powers of the $5$-dimensional representation $W_{1,1}$, with $W_{a,b}$ associated to the partition $\frac{a+b}{2}+\frac{a-b}{2}$ of $a$. To make sense of this, note that conjugation by $\diag(t_1,t_0t_1^{-1},t_2,t_0t_2^{-1})$ acts on $\begin{pmatrix} t & \begin{pmatrix} d&-b\\-c&a\end{pmatrix}\\\begin{pmatrix}a&b\\c&d\end{pmatrix} & -t\end{pmatrix}$ by 
  $$(t,a,b,c,d)\mapsto (t,s_2^{-1}a, s_1b, s_1^{-1}c, s_2d),$$
  where $(s_1,s_2)=(t_1t_2t_0^{-1}, t_1t_2^{-1})$. The highest weight character on $W_{a,b}$ is now $$\diag(t_1,t_0t_1^{-1},t_2,t_0t_2^{-1})\mapsto t_1^at_2^bt_0^{-(a+b)/2}=s_1^{(a+b)/2}s_2^{(a-b)/2}.$$

\subsection{Spaces of algebraic modular forms.}
\label{section:spaces}
Following the conventions of \cite[Introduction]{GV}, let $G/\Q$ be a reductive group with $G(\R)$ compact. Let $W'$ be a finite-dimensional complex representation of $G(\Q)$, and let $\widehat{K}$ be an open compact subgroup of $G(\A_f)$, where $\A_f\simeq \hat{\mathbb{Z}}\otimes\Q$ is the ring of finite ad\`eles. We define
$$M(W',\widehat{K}):=\{f:G(\A_f)\rightarrow W':\,f(\gamma gk)=\gamma\cdot f(g)\,\,\,\forall \gamma\in G(\Q), g\in G(\A_f), k\in\widehat{K}\}.$$
If $\{g_i:\,\,1\leq i\leq h\}$ is a set of representatives of $G(\Q)\backslash G(\A_f)/\hat{K}$ then
\begin{equation}\label{AMF} M(W',\widehat{K})\simeq\bigoplus_{i=1}^h \, (W')^{\Gamma_i},\end{equation}
where $\Gamma_i:=g_i\widehat{K}g_i^{-1}\cap G(\Q)$.

Two special cases will be of particular interest to us.
\begin{enumerate}
\item $G$ such that $G(\Q)=\GUtwo{B}$, with $B/\Q$ a definite quaternion algebra as above, $W'=W_{a,b}$ with $a\geq b\geq 0$, as in the previous section, and $\widehat{K}=\widehat{K}(D)$ defined by its local components
$$\widehat{K}(D)_p:=\begin{cases}K_{0,p} & \text{if } p\nmid D;\\
K^+_{p^n} & \text{if } p^n\parallel D^+;\\
K^-_p & \text{if } p\mid D^-.\end{cases}$$
Then $$M_{a,b}(\hat{K}(D)):=M(W_{a,b}, \widehat{K}(D)).$$
Although strictly speaking $G(\R)$ is only compact modulo its centre, since $W_{a,b}$ has trivial central character we could just as well be working with $\GUtwo{B}/\Q^{\times}$, which does satisfy the condition.
\item $G=\SO(\tilde{V})$, where $\tilde{V}$ is as in
  Remark \ref{remark:identification}. Since the lattice
  $\tilde{L}\subset \tilde{V}$ from Remark \ref{Ltilde} is in the same genus as
  $\Lattice \subset \DimFive$, there is an isomorphism $\Phi$ of
  $\Q$-quadratic spaces from $(\tilde{V}, \tilde{Q})$ to $(\DimFive, Q)$, such that for
  each prime $p$ there is $h_p\in\SO(\DimFive_p)$ with
  \[
    \Phi(\tilde{L}_p)= h_p\Lattice_p.
  \]
  (Though at first $h_p\in\OO(\DimFive_p)$, we may ensure it is in
  $\SO(\DimFive_p)$ by multiplying by $-I_5$ if necessary.)  Via
  $\Phi$, $\SO(\tilde{V})\simeq\SO(\DimFive)\simeq \GUtwo{B}/\Q^{\times}$ (with
  slight abuse of notation), so the representations $W_{a,b}$ may be
  viewed also as representations of $\SO(\tilde{V})$. 
  
  It follows that there is essentially no difference between $(\tilde{V}, \tilde{L}, \tilde{Q})$ and $(\DimFive, \Lattice, Q)$, as far as algebraic modular forms are concerned. (In (\ref{AMF}), replacing $\widehat{K}$ by $h\widehat{K}h^{-1}$ and $g_i$ by $g_ih^{-1}$, where $h:=(h_p)$, leaves $\Gamma_i$ the same.) So from this point onwards we neglect the distinction.

We define an open compact subgroup $\widehat{K(\Lattice)}$ of $G(\A_f)$ by
\[
  \widehat{K(\Lattice)}_p:=\mathrm{Stab}_{G(\Q_p)}(\Lattice_p)\,\,\forall\text{ primes }p,
\]
and let 
\[
  \widetilde{M}_{a,b}(\widehat{K(\Lattice)}):=M(W_{a,b},\widehat{K(\Lattice)}).
\]
We also define a slightly smaller subgroup $\widehat{K(\Lattice)}^+$ by
\[
  \widehat{K(\Lattice)}^+_p:=\begin{cases}\mathrm{Stab}_{G(\Q_p)}(\Lattice_p) & \text{ for }p\nmid D;\\
    \mathrm{Stab}_{G(\Q_p)}(\Lattice_p)^+ & \text{ for }p\mid D,\end{cases}
  \]
  where $\mathrm{Stab}_{G(\Q_p)}(\Lattice_p)^+$ is the subgroup of
  index $2$ in $\mathrm{Stab}_{G(\Q_p)}(\Lattice_p)$ that is
  the kernel of $\theta_p$, with $\theta_p$ as
  in Definition~\ref{defi:character}.  Then we define
  \[
    \widetilde{M}_{a,b}(\widehat{K(\Lattice)}^+):=M(W_{a,b},\widehat{K(\Lattice)}^+).
    \]
\end{enumerate}
\begin{thm}\label{GU2SO5iso}
  There is an isomorphism
  \[
    \widetilde{M}_{a,b}(\widehat{K(\Lattice)}^+)\simeq M_{a,b}(\widehat{K}(D)),
  \]
  equivariant for the right-translation action of
  $SO(\DimFive\otimes\A_f)$ on the left-hand-side and the isomorphic
  $\GUtwo{B\otimes\A^f}/\A_f^{\times}$ on the right-hand-side. For
  $p\mid D$, the action of the involution
  $\mathrm{Stab}_{G(\Q_p)}(\Lattice_p)/\mathrm{Stab}_{G(\Q_p)}(\Lattice_p)^+$
  on the left matches that on the right of $W^+_{p^n}$ (for
  $p^n\parallel D^+$) or $\omega_p$ (for $p\mid D^-$).
\end{thm}
\begin{proof} This is a direct consequence of
  Propositions \ref{K+pstab}(2), \ref{K-pstab}(2) and
  \ref{prop:thetapn}. 
\end{proof}  
\begin{remar}\label{Ladd} Ladd \cite[Theorem 1]{Lad} addressed the
  case $D=p$ by a different approach, proving that the left-hand-side injects into the right-hand-side.
 \end{remar}
\begin{remar}
  We extended the representations $W_{a,b}$ of $\SO(\DimFive)$ to
  representations of $\OO(\DimFive)$ by making the element
  $-I_5\in\OO(\DimFive)-\SO(\DimFive)$ act trivially. Extending
  $\tilde{f}\in\widetilde{M}_{a,b}(\widehat{K(\Lattice)})$ to
  $\OO(\DimFive\otimes\A_f)$ by $\tilde{f}(-I_5g):=\tilde{f}(g)$, and noting
  that $-I_5\in\mathrm{Stab}_{\OO(\DimFive_p)}(\Lattice_p)$,
  we see that $\widetilde{M}_{a,b}(\widehat{K(\Lattice)})$ remains the same
  if everywhere in the definition we substitute $\OO(\DimFive)$ (and
  $\OO(\DimFive\otimes\A_f)$ etc.) for $\SO(\DimFive)$ (and $\SO(\DimFive\otimes\A_f)$
  etc.).
\end{remar}
\begin{prop}\label{Hecke}
  For $p\nmid D$, the Hecke operators $T(p)$ and $T_1(p^2)$ on
  $M_{a,b}(\widehat{K}(D))$ correspond to the $p$-neighbour and
  $p^2$-neighbour Hecke operators on
  $\widetilde{M}_{a,b}(\widehat{K(\Lattice)}^+)$, in the sense of
  \cite[Theorem 5.11]{GV}.
\end{prop}
\begin{proof}
\begin{enumerate}
\item The Hecke operator usually denoted $T(p)$ or
  $T_1(p)$ is associated to the double coset
  $K_{0,p}\,\diag(p,1,p,1)\,K_{0,p}$. Note that $\Psi: \GUtwo{B_p}\simeq \GSp_2(\Q_p)$ maps $\diag(p,1,p,1)$ to the usual $\diag(p,p,1,1)$. Acting by
  conjugation on $\DimSix_p$,
  \[
    \diag(p,1,p,1):\,\begin{pmatrix} t & \begin{pmatrix} d&-b\\-c&a\end{pmatrix}\\
      \begin{pmatrix}a&b\\c&d\end{pmatrix} & s\end{pmatrix}\mapsto \begin{pmatrix} t & \begin{pmatrix} d&-pb\\-c/p&a\end{pmatrix}\\\begin{pmatrix}a&pb\\c/p&d\end{pmatrix} & s\end{pmatrix}.
    \]
  In the language of \S \ref{repns}, $(t_0,t_1,t_2)=(p,p,p)$, $(s_1,s_2)=(t_1t_2t_0^{-1},t_1t_2^{-1})=(p,1)$, and
$$\diag(p,1,p,1): (t,s,a,b,c,d)\mapsto (t,s,a,pb,p^{-1}c,d).$$
Thus (passing to the $5$-dimensional quotient) $T(p)$ on $M_{a,b}(\widehat{K}(D))$ corresponds
to
$$\mathrm{Stab}_{\OO(\DimFive_p)}(\Lattice_p)\, \diag(1,1,p,p^{-1},1)\, \mathrm{Stab}_{\OO(\DimFive_p)}(\Lattice_p)$$
on $\widetilde{M}_{a,b}(\widehat{K(\Lattice)}^+)$, which is the
$p$-neighbour operator as in \cite[Theorem 5.11]{GV}.
\item Similarly if $(t_0,t_1,t_2)=(p^2,1,p)$ then
  $(s_1,s_2)=(p^{-1},p^{-1})$, and we see that the operator $T_1(p^2)$
  on $M_{a,b}(\widehat{K}(D))$ corresponds to the $p^2$-neighbour
  Hecke operator on $\widetilde{M}_{a,b}(\widehat{K(\Lattice)}^+)$, in
  the sense of \cite[Theorem 5.11]{GV} (which might also reasonably be
  called ``$(p,p)$-neighbour operator'').
\end{enumerate}
\end{proof}

Given a positive divisor $d\mid D$ with $\mathrm{gcd}(d,D/d)=1$, we
may define a character
$\theta_d: \widehat{K(\Lattice)}\rightarrow\{\pm 1\}$ by
\[
  \theta_d|_{\widehat{K(\Lattice)_p}}=\begin{cases}
    \mathrm{id.} & \text{ if }p\nmid d;\\
\theta_p & \text{ if }p\mid d.\end{cases}.
\]
We define a subspace $\widetilde{M}_{a,b}(\widehat{K(\Lattice)}^+)^{\theta_d}$ of $\widetilde{M}_{a,b}(\widehat{K(\Lattice)}^+)$ by 

\begin{multline*}
  \widetilde{M}_{a,b}(\widehat{K(\Lattice)}^+)^{\theta_d}:=\{\tilde{f}\in\widetilde{M}_{a,b}(\widehat{K(\Lattice)}^+):\,\,\,f(gk)=\theta_d(k)f(g)\\ \forall g\in \SO(\DimFive\otimes\A_f), k\in\widehat{K(\Lattice)}\}.
\end{multline*}

Under the isomorphism of Theorem \ref{GU2SO5iso}, clearly this corresponds to the subspace of $M_{a,b}(\widehat{K}(D))$ on which $W^+_{p^n}$ or $\omega_p$ acts as $-1$ precisely for $p\mid d$.

\begin{cor}
  \label{coro:ALdecomposition}
  Keeping the previous notation, there is a natural isomorphism
  \[
M_{a,b}(\widehat{K}(D)) \simeq \bigoplus_{\stackrel{d \mid D}{\gcd(d,D/d)=1}} \widetilde{M}_{a,b}(\widehat{K(\Lattice)}^+)^{\theta_d}.
    \]
\end{cor}

If there is no $p$ such that $\ord_p(d)$ is even, in particular if $d$
is square-free, then $\theta_d$ may be extended to
$\SO(\DimFive\otimes\A_f)$, since locally at $p\mid d$ it is $\chi\circ\nu$,
as in Remark \ref{snorm}. Thus we may define a representation
$W_{a,b}\otimes\theta_d$ of $\SO(\DimFive\otimes\A)$, with
$\SO(\DimFive\otimes\R)\simeq\SO_5(\R)$ acting on the $W_{a,b}$ factor. This
restricts to give a representation of the diagonally embedded
$\SO(\DimFive)$.
\begin{prop}\label{thetad} Suppose that $d$ is a positive divisor of $D$ with $\mathrm{gcd}(d,D/d)=1$, and that there is no $p$ such that $\ord_p(d)$ is even. Then
  \[
    \widetilde{M}_{a,b}(\widehat{K(\Lattice)}^+)^{\theta_d}\simeq M(W_{a,b}\otimes\theta_d, \widehat{K(\Lattice)}).
    \]
\end{prop}
\begin{proof} An element of
  $M(W_{a,b}\otimes\theta_d, \widehat{K(\Lattice)})$ is a function
  $f:\SO(\DimFive\otimes\A_f)\rightarrow W_{a,b}$ such that for
  $\gamma\in\SO(\DimFive)$, $g\in\SO(\DimFive\otimes\A_f)$ and
  $k\in \widehat{K(\Lattice)}$,
  $$f(\gamma gk)=\theta_d(\gamma)\gamma\cdot f(g).$$ An element of
  $\widetilde{M}_{a,b}(\widehat{K(\Lattice)}^+)^{\theta_d}$ is a
  function $\tilde{f}:\SO(\DimFive\otimes\A_f)\rightarrow W_{a,b}$ such that
  for $\gamma\in\SO(\DimFive)$, $g\in\SO(\DimFive\otimes\A_f)$ and
  $k\in\widehat{K(\Lattice)}$,
  $$\tilde{f}(\gamma gk)=\theta_d(k)\gamma\cdot f(g).$$
  $\tilde{f}(g):=\theta_d(g)f(g)$ gives the isomorphism we seek.
\end{proof}

\section{From algebraic modular forms for $\GUtwo{B}$ to Siegel modular forms of paramodular level}
For any positive integer $N$, let 
$$P(N):=\left[\begin{array}{cccc} \Z & N\Z & \Z &\Z \\ \Z & \Z & \Z & \frac{1}{N}\Z \\ \Z & N\Z & \Z &\Z \\ N\Z & N\Z & N\Z & \Z \end{array}\right] \cap \mathrm{Sp}_2(\Q)$$ be the paramodular group of level $N$. For integers $k\geq 1$, $j\geq 0$, let $\rho_{k,j}:\GL_2(\C)\rightarrow\Aut(V_{k,j})$ be the $\det^k\otimes\Sym^j$ representation. Let $S_{k,j}(P(N))$ be the space of holomorphic functions $F:\HH_2\rightarrow V_{k,j}$, satisfying a cuspidality condition, and 
$$F((AZ+B)(CZ+D)^{-1})=\rho_{k,j}(CZ+D)(F(Z))\,\,\,\,\forall\left(\begin{smallmatrix}A&B\\C&D\end{smallmatrix}\right)\in P(N).$$
Although it is not true that $P(p^n)$ is contained in $P(p^m)$ for
$n > m$, there is still a theory of oldforms and newforms for
automorphic forms on $\GSp_2$, as studied in \cite{RS} (in
particular, the formula for the number of oldforms given in Theorem
7.5.6). 

In a series of articles, Ibukiyama and some coauthors (see
\cite{MR3638279} and also \cite{I,I2}) stated a series of conjectures
relating automorphic forms on $\GUtwo{B}$ and $\GSp_2$. The
conjectures were proven by R\"osner and Weissauer in a recent article \cite{RW}, using the trace formula. A somewhat less general result was obtained independently by van Hoften \cite{vH} using very different, algebro-geometric tools. Although the
original conjecture (and the proof) was made for square-free levels
$D=D^-$, a more general version holds. In \cite{MR3638279} the authors
consider a quaternion algebra ramified at all primes dividing $D$ and at infinity, (taking as open compact subgroup the one corresponding to $D^+=1$), and
relate the spaces $M_{a,b}(\widehat{K}(D))$ and
$S_{b+3,a-b}(P(D))$.

\subsection{Generalisation of Ibukiyama-Kitayama conjecture} 
For $M \mid N$ with $\gcd(M,N/M)=1$, let $S_k^{\text{$M$-new}}(N)$
denote the space of cusp forms for $\Gamma_0(N)$ (genus-$1$) that are
new at $M$. We keep the notation of previous sections, i.e.
$D= D^+ D^-$ where $D^-$ is square-free. The contribution of the
Yoshida and Saito-Kurokawa lifts to the algebraic modular forms for
$\GUtwo{B}$ is given in Propositions \ref{prop:yoshida} and
\ref{prop:saito-kurokawa} below.

\begin{prop}[Yoshida lifts]
Fix a pair of positive integers $d_+, c_+$ with $d_+c_+\mid D^+$.
There is an injective embedding, sending eigenforms to eigenforms (in both cases, for Hecke operators at primes not dividing the level, Atkin-Lehner involutions at primes dividing the level, where the level depends on the form),
\[
  \iota: \bigoplus_{\stackrel{d_-\mid D^-}{\omega(d_-)\ \text{odd}}}S^{\text{new}}_{2+a-b}(d_-d_+)\times S^{\text{new}}_{4+a+b}\left(\frac{D^-}{d_-}\,c_+\right)\hookrightarrow M_{a,b}(\widehat{K}(D^-d_+c_+))
  \]
  The oldforms in $M_{a,b}(\widehat{K}(D))$ generated from the
  $\iota(g,h)$ by applying Roberts and Schmidt's level-raising
  operators $\theta, \theta'$ and $\eta$, at primes dividing
  $\frac{D^+}{d_+c_+}$, (or the newforms $\iota(g,h)$ if
  $\frac{D^+}{d_+c_+}=1$), span (as we vary $d_+$ and $c_+$) the subspace
  $M_{a,b}(\widehat{K}(D))_Y$ of all forms of Yoshida type, i.e. the
  span of $\widehat{K}(D)$-fixed vectors in endoscopic automorphic
  representations.

  Furthermore, for eigenforms $g$ and $h$ the spinor $L$-function
  $L^{D^-d_+c_+}(s,\iota(g,h),\mathrm{spin})$ of $\iota(g,h)$ is given
  by $L^{D^-d_+c_+}(s-b-1,g)L^{D^-d_+c_+}(s,h)$, where the superscript
  indicates that Euler factors at primes dividing $D^-d_+c_+$ are
  omitted.
\label{prop:yoshida}
\end{prop}
\begin{proof}
  The proof of \cite[Proposition 12.1]{RW} (based on results
  of Chan and Gan (\cite{MR3267112}) applies to our more general
  context, but there $D^+=1$ so there is no $d_+$, $c_+$ or oldforms. Let $\pi^{\iota(g,h)}$ be the associated automorphic representation of $\GUtwo{B\otimes\A}$.
  
  In the notation of \cite[2.2, 2.4]{MR3267112}, $\pi^{\iota(g,h)}_p$ is $\pi_{\phi}^{-+}$ for $p\mid d_-$, $\pi_{\phi}^{+-}$ for $p\mid\frac{D^-}{d_-}$, and $\pi^{\iota(g,h)}_{\infty}$ is $\pi^{-+}$. Since $\omega(d_-)$ is odd and $\omega\left(\frac{D^-}{d_-}\right)$ is even, the product of the $\pm\pm$ pairs over all places is trivial, hence \cite[Theorem 3.1]{MR3267112} (Arthur's multiplicity formula) gives the existence of $\pi^{\iota(g,h)}$. Note that for primes $p\nmid D^-$,  $\pi^{\iota(g,h)}_p$ is $\pi_{\phi}^{+}$ in the notation of \cite[2.1]{MR3267112}, or even $\pi_{\phi}^{++}$ in the notation of \cite[2.3]{MR3267112}, and does not contribute to the product. (When there are two elements in the $L$-packet, we must choose the generic one, to have paramodular fixed vectors, cf. \cite[Remark 3.5]{MR3092267}.)
  
We need to know that the paramodular level of $\pi^{\iota(g,h)}_p$ is $p^{\ord_p(d_+c_+)}$, for primes $p\nmid D^-$, so that $\iota(g,h)\in M_{a,b}(\widehat{K}(D^-d_+c_+))$. This follows from the fact that the $L$-parameter of $\pi^{\iota(g,h)}_p$ is a kind of direct sum of those of $\pi^g_p$ and $\pi^h_p$ (representations of $\GL_2(\Q_p)$), from the behaviour of $\epsilon$-factors under the local Langlands correspondence, and the relation between paramodular level and $\epsilon$-factors, for generic representations, which is \cite[Corollary 7.5.5]{RS}.
\end{proof}

\begin{exo} Let $D^- = 5$, $D^+ = 77$ and $(a,b)=(0,0)$. Then the
  only non-zero contributions come from taking $d_- = 5$, $d_+ \in \{1, 7, 11, 77\}$ and $c_+=D^+/d_+$. The algorithm described in \cite{RT}
  computes the space $M_{0,0}(385)$ corresponding to the genus of quinary forms with
  Hasse-Witt invariant $-1$ only at the prime $5$. There are four rational
  Yoshida lifts, corresponding to the pairs of modular forms:
  $(\lmfdbfform{35}{2}{a}{b},\lmfdbfform{11}{4}{a}{a})$,
  $(\lmfdbfform{35}{2}{a}{a},\lmfdbfform{11}{4}{a}{a})$,
  $(\lmfdbfform{55}{2}{a}{a},\lmfdbfform{7}{4}{a}{a})$ and
  $(\lmfdbfform{55}{2}{a}{b},\lmfdbfform{7}{4}{a}{a})$.

  There are precisely two newforms in $S_2(\Gamma_0(35))$, labelled
  $\lmfdbfform{35}{2}{a}{a}$ and $\lmfdbfform{35}{2}{a}{b}$ and two
  newforms in $S_2(\Gamma_0(55))$, labelled $\lmfdbfform{55}{2}{a}{a}$
  and $\lmfdbfform{55}{2}{a}{b}$. There is a unique form in
  $S_4(\Gamma_0(7))$ labelled $\lmfdbfform{7}{4}{a}{a}$ and one form in
  $S_4(\Gamma_0(11))$ labelled $\lmfdbfform{11}{4}{a}{a}$. All other combinations of spaces involve one that is trivial.
  
\end{exo}

\begin{lem}\label{trivchi} Consider an eigenform
  $f\in M_{b,b}(\widehat{K}(D))$ whose associated automorphic
  representation $\pi^f$ of $\GUtwo{B\otimes\A}$ is CAP associated to
  the Siegel parabolic. In the notation of \cite{MR2402681}, $\pi^f$
  belongs to a global $A$-packet $A_{\tau,\chi}$, where $\tau$ is a
  cuspidal automorphic representation of $\GL_2(\A)$ with trivial
  central character, and $\chi$ is a quadratic character of
  $\A^{\times}/\Q^{\times}$. Then $\chi$ is trivial.
\end{lem}
\begin{proof} This is modelled on the proof (in the case $\GSp_2(\A)$) of \cite[Proposition 5.2(i)]{MR4103275}, where $\tau$ is $\mu$ and $\chi$ is $\sigma$.
(In \cite[\S 5.5]{RS}, $\tau$ is $\pi$ and $\chi$ is $\sigma$.)

At any split prime $p$, $\pi^f_p$ is the unique irreducible quotient of a representation denoted $\nu^{1/2}\tau_p\rtimes\nu^{-1/2}\chi_p$, induced from a Siegel parabolic. At a non-split prime $p\mid D^-$, $\pi^f_p$ is the unique irreducible quotient of a representation denoted $\nu^{1/2}\mathrm{JL}(\tau_p)\rtimes\nu^{-1/2}\chi_p$, induced from a parabolic with Levi subgroup $B_p^{\times}\times\Q_p^{\times}$. Considering the action of an element $(\mathrm{id}.,u)$ of the Levi subgroup, where $u\in\Z_p^{\times}$, it is easy to show that $\pi^f_p$ cannot have a non-zero $\widehat{K}(D)_p$-fixed vector unless $\chi_p$ is unramified, and the only $\chi$ such that all $\chi_p$ are unramified is the trivial character.
\end{proof}

Let $\mathfrak{S}^{\text{$\frac{D^-}{d_-}$-new},
  -(-1)^{\omega(d_-)}}_{2b+4}(\Gamma_0(D/d_-)$ be the subspace of
$S^{\text{$\frac{D^-}{d_-}$-new},-(-1)^{\omega(d_-)}}_{2b+4}(\Gamma_0(D/d_-)$
spanned by eigenforms with the same Atkin-Lehner
eigenvalues, at all primes dividing $D^+$, as the newforms they come
from. For more, see around \cite[(23)]{MR2208781}.

\begin{prop}[Saito-Kurokawa lifts]
  For each positive divisor $d_-$ of $D^-$ there is an injective embedding 
  \[
    \iota_{d_-}: \mathfrak{S}^{\text{$\frac{D^-}{d_-}$-new}, -(-1)^{\omega(d_-)}}_{2b+4}(\Gamma_0(D/d_-)) \hookrightarrow M_{b,b}(\widehat{K}(D))
    \]
    
    Furthermore, the span of the images of the $\iota_{d_-}$ is precisely the subspace $M_{b,b}(\widehat{K}(D))_{\mathrm{SK}}$ of forms of Saito-Kurokawa type, i.e. spanned by $\widehat{K}(D)$-fixed vectors in automorphic representations that are CAP for the Siegel parabolic subgroup. 
    
If $h\in \mathfrak{S}_{2b+4}^{\text{$\frac{D^-}{d_-}$-new}, -(-1)^{\omega(d_-)}}(\Gamma_0\left(D/d_-\right))$ is
    an eigenform (of Hecke operators for $p\nmid\frac{D}{d_-}$, Atkin-Lehner involutions for $p\mid\frac{D}{d_-}$), then $\iota_{d_-}(h)$ is an eigenform (for Hecke operators at $p\nmid D$, Atkin-Lehner involutions at $p\mid D$). The spinor L-function
    $L^D(s,\iota_{d-}(h),\mathrm{spin})$ of $\iota_{d_-}(h)$ is given by
    $\zeta^D(s-b-2)\zeta^D(s-b-1)L^D(s,h)$, where the superscript $D$ indicates that Eulers factors at primes $p\mid D$ are omitted.
\label{prop:saito-kurokawa}
\end{prop}
\begin{proof}
  This may be proved in very much the same way as \cite[Proposition 12.2]{RW}. (The part about the span of the images follows from a theorem of E. Sayag \cite[Theorem 6.9]{MR2402681}.)
  The difference is that on both left and right, forms need not be $D^+$-new, whereas in \cite[Proposition 12.2]{RW}, $D^+=1$. 
  
  Any eigenform on the left comes from some newform
  $g\in
  S^{\text{new},-(-1)^{\omega(d_-)}}_{2b+4}\left(\Gamma_0\left(\frac{D^-}{d_-}\frac{D^+}{d_+}\right)\right)$,
  for some positive divisor $d_+$ of $D^+$. The sign
  $ -(-1)^{\omega(d_-)}$ is by definition required to be the global
  sign in the functional equation of the $L$-function associated to
  $g$. Using Gan's Saito-Kurokawa lifting \cite{MR2402681} as in
  \cite[Proposition 12.2]{RW}, we obtain an eigenform
  $G\in M_{b,b}\left(\widehat{K}\left(D/d_+\right)\right)$. In fact
  $D^+/d_+$ is precisely the paramodular level of the cuspidal
  automorphic representation $\pi^G$ of $\GUtwo{B\otimes\A}$ attached
  to $G$. This means that for all split primes, $p\nmid D^-$, the
  least $n$ such that $\pi^G_p$ has non-zero $K^+_{p^n}$-invariants
  (necessarily $1$-dimensional) is $\ord_p(D^+/d^+)$. This is
  precisely \cite[Proposition 5.5.5(i)]{RS}, bearing in mind that at
  split places Gan's Saito-Kurokawa lifting is locally the local
  Saito-Kurokawa lifting of \cite[\S 5.5]{RS}.
    
On the left we can obtain, from $g$, oldforms in $S^{{\text{$\frac{D^-}{d_-}$-new}},-(-1)^{\omega(d_-)}}_{2b+4}(\Gamma_0(D/d_-))$. Recall that we are taking only those with the same Atkin-Lehner eigenvalues as $g$ for primes dividing $D^+$. On the right we can obtain from $G$ oldforms in $M_{b,b}(\widehat{K}(D))$, by applying Roberts and Schmidt's operators $\theta$ and $\eta$ at primes dividing $d_+$. Including their third level-raising operator $\theta'$ would not produce anything linearly independent of what one obtains using $\theta$ and $\eta$, since by \cite[Theorem 5.5.9]{RS}, $\theta$ and $\theta'$ act the same (up to $\pm$) on paramodular-fixed vectors in local representations of Saito-Kurokawa type. 

In fact, by Lemma \ref{trivchi} and \cite[Theorem 5.5.9(iv)]{RS}, the sign is $+$, i.e. $\theta$ and $\theta'$ act the same, and the oldforms coming from $G$ as above have the same Atkin-Lehner eigenvalues as $G$. As explained in the discussion preceding \cite[Theorem 6.3]{MR2208781}, the oldforms coming from $g$ can be exactly paired with those coming from $G$. Although the argument there was applied to Siegel modular forms, it applies just the same, being local at split primes.
 \end{proof}

\begin{exo}
  Let us continue the previous example. Let $D^-= 5$ and $D^+ =
  77$. Taking $d_- = 1$ implies computing the space
  $S_4^{\text{5-new},-}(385)$. The newforms of the different
  spaces are $S_4^{\text{new},-}(5) = \{\lmfdbfform{5}{4}{a}{a}\}$,
  $S_4^{\text{new},-}(35) = \{\lmfdbfform{35}{4}{a}{a}\}$,
  $S_4^{\text{new},-}(55) = \emptyset$ and
  $S_4^{\text{new},-}(385) = \{\lmfdbfform{385}{4}{a}{b},
  \lmfdbfform{385}{4}{a}{e}, \lmfdbfform{385}{4}{a}{h},
  \lmfdbfform{385}{4}{a}{i}, \lmfdbfform{385}{4}{a}{j},
  \lmfdbfform{385}{4}{a}{l}\}$.

  Taking $d_- = 5$ implies computing spaces of forms of level not
  divisible by $5$ and sign $+1$ in their functional equation. The
  newforms of the different spaces are $S_4^{\text{new},+}(1)=\emptyset$,
  $S_4^{\text{new},+}(7) = \{\lmfdbfform{7}{4}{a}{a}\}$,
  $S_4^{\text{new},+}(11) = \{\lmfdbfform{11}{4}{a}{a}\}$ and
  $S_4^{\text{new},+}(77) = \{\lmfdbfform{77}{4}{a}{a},
  \lmfdbfform{77}{4}{a}{d}, \lmfdbfform{77}{4}{a}{e}\}$.
  All such forms contribute to the space $M_{0,0}(\widehat{K}(385))$.
\end{exo}

Let $M_{(a,b)}(\widehat{K}(D))_G$ denote the subspace of $M_{(a,b)}(\widehat{K}(D)$ orthogonal to all
Yoshida lifts, and to all Saito-Kurokawa
lifts. Here we are using a natural inner product, with respect to which Hecke operators are self-adjoint. Let $S_{k,j}^{\text{$D^-$-new}}(P(D))_G$ denote the subspace of forms (among those $D^-$-new in the sense of Roberts and Schmidt) that are orthogonal to all Saito-Kurokawa lifts. When $j>0$ there are no Saito-Kurokawa lifts, and this is the whole space. Then we have the following natural generalisation of the conjecture of Ibukiyama-Kitayama proved in \cite[Proposition 12.3]{RW}. The subscript ``$G$'' stands for ``general type'', and will be justified in the course of the proof.

\begin{thm}
There is a linear isomorphism, sending eigenforms to eigenforms (for Hecke operators at $p\nmid D$, Atkin-Lehner involutions at $p\mid D$), and preserving $L$-functions (with Euler factors at primes dividing $D$ omitted):
  \[
S_{b+3,a-b}^{\text{$D^-$-new}}(P(D))_G \simeq M_{a,b}(\widehat{K}(D))_G.
    \]
  \label{thm:Rainer1}
  \end{thm}
  \begin{proof}
  This may be proved following \cite[Proposition 12.3]{RW}, though since for us $\omega(D^-)$ is odd, the large part of their proof dealing with their Conjecture 7.5 is not needed here.  
  
There are no forms of Yoshida type on the left, as explained at the beginning of the proof of \cite[Theorem 12.3]{RW}, or on the right, by definition. By \cite[Proposition 5.1]{MR4103275}, there cannot be forms that are CAP with respect to the Borel or Klingen parabolic subgroups. This leaves only forms of general type, neither CAP nor endoscopic.
  
For forms of general type we apply \cite[Theorem 11.4]{RW} as in
\cite[Proposition 12.3]{RW}. Again, the difference is that on both
sides there may be forms that are not $D^+$-new. Consider an eigenform
$f$ in $M_{a,b}(\widehat{K}(D/d_+))_G$ (where $d_+$ is some positive
divisor of $D^+$), with paramodular level exactly $D^+/d_+$, and
associated automorphic representation $\pi^f$ of
$\GUtwo{B\otimes\A}$. Applying \cite[Theorem 11.4]{RW} gives us an
eigenform $F\in S_{b+3,a-b}^{\text{$D^-$-new}}(P(D/d_+))_G$, with
exact paramodular level $D/d_+$, and associated cuspidal automorphic
representation $\pi^F$ of $\GSp_2(\A)$. The arguments of
\cite[Proposition 12.3]{RW} prove this paramodular level for primes
dividing $D^-$, but it also holds at other primes $p$, simply because
$\pi^f_p\simeq \pi^F_p$, even for $p\mid D^+$. Then the oldforms in
$M_{a,b}(\widehat{K}(D))_G$ and
$S_{b+3,a-b}^{\text{$D^-$-new}}(P(D))_G$ generated from $f$ and $F$
respectively, by applying Roberts and Schmidt's level-raising
operators $\theta, \theta'$ and $\eta$ at primes $p\mid d_+$, exactly
correspond, again because $\pi^f_p\simeq \pi^F_p$.
\end{proof}
\begin{remar} Under a different convention, we could leave in the $D^-$-new Saito-Kurokawa lifts on both sides (thus having simply $S_{b+3,a-b}^{\text{$D^-$-new}}(P(D))$ on the left), and they would correspond to each other.
  \end{remar}

\begin{thm}
  Keeping the previous notation, let $D = D^+D^-$. Then
  \[
    S_{b+3,a-b}^{\text{$D^-$-new}}(P(D))_G\simeq
    \bigoplus_{\stackrel{d \mid
        D}{\gcd(d,D/d)=1}}\widetilde{M}_{a,b}(\widehat{K(\Lattice)}^+)_G^{\theta_d}.
    \]
  \label{thm:basisproblem}
\end{thm}
\begin{proof}
  Follows from Corollary~\ref{coro:ALdecomposition} and Theorem~\ref{thm:Rainer1}.
\end{proof}
In particular, this proves Conjecture 15 in \cite{RT}.
\begin{remar}
  \label{rem:basisproblem}
  Theorem~\ref{thm:basisproblem} gives a solution of the basis problem
  when the level of the paramodular form is divisible by a prime to
  the first power (as studied by Eichler \cite{MR0485699} and Hijikata
  \cite{MR337783} for classical modular forms). If $N$ is a positive
  integer such that there exists $p$ with $v_p(N)=1$, then the space
  of paramodular newforms of level $N$ can be computed using the
  decomposition $N = p (N/p)$, i.e. looking at a special positive
  definite quinary form of discriminant $2N$ with Hasse-Witt invariant
  $-1$ at $p$, and Eichler invariant $+1$ at all primes dividing $N/p$
  as described in Section~\ref{section:globallattice}.
  
\end{remar}

\section{Atkin-Lehner eigenvalues}
Recall from Corollary \ref{coro:ALdecomposition} the decomposition
\[
  M_{a,b}(\widehat{K}(D)) \simeq \bigoplus_{\stackrel{d \mid D}{\gcd(d,D/d)=1}} \widetilde{M}_{a,b}(\widehat{K(\Lattice)}^+)^{\theta_d}.
  \]
  We know that the left hand side can contain eigenforms of
  Saito-Kurokawa type (if $a=b$), Yoshida type and general type. We
  now wish to identify, for each of these types, what we have in a
  given $\widetilde{M}_{a,b}(\widehat{K(\Lattice)}^+)^{\theta_d}$, in
  terms of the Atkin-Lehner eigenvalues of associated modular
  forms. In this section, as already in the previous section, we
  freely apply the isomorphism $\Psi: \GUtwo{B_p}\simeq\GSp_2(\Q_p)$
  at split $p$, and do not distinguish between $W^+_{p^n}$ and
  $W_{p^n}$.
 \begin{thm}\label{ALsignchange}  
 %Consider a positive-definite $\Q$-quadratic space $(V,Q)$ with $\dim V=5$, $\Lattice\subset V$ an even integral $\Z$-lattice for which the determinant of a Gram matrix $\disc(\Lattice)=2D$, with $D$ odd. Suppose that $D=D^+D^-$, where $D^-$ is the square-free product of an odd number of primes, and that 
%$$\mathrm{HW}_p(Q)=\begin{cases} 1 & \text{ if } p\mid D^+;\\
%-1 & \text{ if } p\mid D^-.\end{cases}$$
%Suppose that for all $p^n\parallel D^+$ we have genus invariants $n_1=4$, $n_{p^n}=1$ (a condition that is automatically satisfied if $D^+$ is square-free).

%For a positive divisor $d\mid D$ with $\gcd(d,D/d)=1$, consider $\tilde{f}\in \widetilde{M}_{a,b}(\widehat{K(\Lattice)}^+)^{\theta_d}$, a Hecke eigenform for all $p$-neighbour and $p^2$-neighbour operators for $p\nmid D$.%

Let $M_{a,b}(\widehat{K}(D))$ be a space of $\GUtwo{B}$ algebraic modular forms as in \S \ref{section:spaces}. In particular $D=D^-D^+$, with $\gcd(D^-,D^+)=1$ and $D^-$ square-free, $\omega(D^-)$ odd. Consider $f\in M_{a,b}(\widehat{K}(D))$, a Hecke eigenform for $T(p)$ and $T_1(p^2)$ (all $p\nmid D$), $W_{p^n}$ ($p^n\parallel D^+$) and $\omega_p$ ($p\mid D^-$). For $p\mid D$ let $e_p$ denote the eigenvalue of $W_{p^n}$ or $\omega_p$ on $f$. % is $-1$ precisely for $p\mid d$. 
\begin{enumerate}
\item If $f\in M_{a,b}(\widehat{K}(D))_G$, with corresponding $F\in S_{b+3,a-b}^{\text{$D^-$-new}}(P(D))_G$, with eigenvalues $w_p$ of $W_{p^n}$ on $F$ (where $n$ depends on $p$, with $p^n\parallel D$), then
$$e_p=\begin{cases} w_p & \text{ if $p\mid D^+$};\\-w_p & \text{ if $p\mid D^-$}.\end{cases}$$
\item If $f\in M_{a,b}(\widehat{K}(D))_{\mathrm{SK}}$, say $f=\iota_{d_-}(h)$ for some $h\in \mathfrak{S}^{\text{$\frac{D^-}{d_-}$-new}, -(-1)^{\omega(d_-)}}_{2b+4}(\Gamma_0(D/d_-))$. Let $\epsilon_p$ be the local sign at $p$ attached to $h$, in particular $\epsilon_p$ is the eigenvalue of a $\GL_2$ Atkin-Lehner operator when $p\mid \frac{D}{d_-}$, and is $1$ for other $p$. Then 
$$e_p=\begin{cases} \epsilon_p & \text{ if $p\mid d_-D^+$};\\-\epsilon_p & \text{ if $p\mid \frac{D^-}{d_-}$}.\end{cases}$$

\item If $f'=\iota(g,h)\in M_{a,b}(\widehat{K}(D^-d_+c_+))_Y$, as in Proposition \ref{prop:yoshida}, then the eigenvalue of $W_{p^n}$ on $f'$, for $p\mid D^+$ is a product of local signs $\epsilon_p(g)\epsilon_p(h)$, while that of $\omega_p$ for $p\mid D^-$ is $-\epsilon_p(g)\epsilon_p(h)$.

Oldforms in $M_{a,b}(\widehat{K}(D))_Y$ may be produced from $f'$ by (possibly repeated) application of Roberts and Schmidt's level-raising operators $\theta, \theta'$ (at $p\mid \frac{D^+}{d_+c_+}$) and $\eta$ (at $p$ when $p^2\mid\frac{D^+}{d_+c_+}$). Application of $\eta, \theta+\theta'$ or $\theta-\theta'$ takes (Atkin-Lehner) eigenforms to eigenforms, with a change of sign in the case of $\theta-\theta'$.
\end{enumerate}
 \end{thm}
 \begin{proof} 
 \begin{enumerate}
 \item First suppose that $f\in M_{a,b}(\widehat{K}(D))_G$. Let $\pi^F$ be 
the cuspidal automorphic representation of $\GSp_2(\A)$ associated with $F$, and $\pi^f$ the automorphic representation of $\GUtwo{B\otimes\A}$ attached to $f$. For $p\mid D^+$, $\pi^f_p\simeq\pi^F_p$, so $e_p=w_p$. For $p\mid D^-$, we need  to show that $\omega_p$ acts on $f$ by $-w_p$. 
 
As in the proof of \cite[Proposition 12.3]{RW}, $\pi^F_p$ is of type IIa, in the notation (from \cite{RS}) used there. The local component $\pi^F_p$ (for $p\mid D^-$) is of type IIa, i.e. $\chi\,\mathrm{St}_{\GL_2}\rtimes\sigma$, for some unramified characters $\sigma, \chi$ of $\Q_p^{\times}$, with $\chi^2\neq \nu^{\pm 1}, \chi\neq\nu^{\pm 3/2}$, where $\nu$ is an unramified character of $\Q_p^{\times}$ with $\nu(p)=p^{-1}$. (Here we switch temporarily to the notation of \cite{RS}, \cite{Sc}, so $\nu$ here is what would have been $\chi^{-1}$ in the notation of Remark \ref{snorm}, not the spinor norm.) The fact that $\sigma$ and $\chi$ must be unramified follows from the fact that $p\mid\mid D$, and the `$N$' column of \cite[Table A.14]{RS}. By \cite[Table A.12]{RS}, the eigenvalue of $W_p$ acting on $F$ is $w_p=-(\chi\sigma)(p)$. Note that the ambiguous comment above \cite[Table A.15]{RS} suggests this might not be correct, but it is referring to the $\pm$-signs under the dimensions of the fixed spaces, not the entries in the Atkin-Lehner eigenvalue column. This is clearer in \cite[\S 1.3]{MR2114732}.

The local component of $\pi^f_p$, viewed as a representation of $\GUone{B_p}$, is of the type called $\text{IIa}^G$ in \cite[table in A.4]{Sc}. It is $\chi\,1_{B^{\times}}\rtimes\sigma$, for the same $\chi$ and $\sigma$ as above. For $\pi\in B_p$ with $\pi^2=-p$, the Atkin-Lehner element in $\GUone{B_p}$ is
$$\omega'_p=\begin{pmatrix}0&p\\1&0\end{pmatrix}=\begin{pmatrix}\pi &0\\0&\pi\end{pmatrix}\begin{pmatrix}0& -\pi\\\pi^{-1}& 0\end{pmatrix},$$
with $\begin{pmatrix}0& -\pi\\\pi^{-1}& 0\end{pmatrix}\in K^-(p)$, so $\begin{pmatrix}\pi &0\\0&\pi\end{pmatrix}$ is an acceptable alternative Atkin-Lehner element. It now follows from \cite[after Lemma A.2]{Sc} that $\omega_p$ acts on the $K_p^-$-fixed line in $\pi^f_p$ as multiplication by $(\chi\sigma)(p)=-w_p$.

\item Suppose that $f\in M_{a,b}(\widehat{K}(D))_{\mathrm{SK}}$, with
  $f=\iota_{d_-}(h)$ for some eigenform
  $h\in \mathfrak{S}^{\text{$\frac{D^-}{d_-}$-new},
    -(-1)^{\omega(d_-)}}_{2b+4}(\Gamma_0(D/d_-))$, $\epsilon_p$ the
  local sign at $p$ attached to $h$, or rather to the newform $g$ from
  which it comes.

Suppose that $p\mid D^-$. As in the proof of \cite[Proposition 12.2]{RW}, the representation $\pi^f_p$ of $\GUone{B_p}$ could be $\text{IIa}^G$ (if $p\mid d_-$), or $\text{Vb}^G$ or $\text{VIc}^G$ (if $p\mid \frac{D^-}{d_-}$). 

We have already seen that in the case that $\pi^f_p\simeq \chi\,1_{B^{\times}}\rtimes\sigma$, of type $\text{IIa}^G$, $e_p=(\chi\sigma)(p)$. As in \cite[Proposition 5.5.1(i)]{RS}, there is an associated representation of $\GSp_2(\Q_p)$, of type IIb, which is $\chi\,1_{\GL_2}\rtimes\sigma$. The $\sigma$ here is the same as the $\sigma$ in the tables at the end of \cite{RS}, but it corresponds to the $\chi^{-1}\sigma$ in \cite[Proposition 5.5.1(i)]{RS}, where $\chi$ is the same. If we call the $\sigma$ in \cite[Proposition 5.5.1(i)]{RS} ``$\sigma'$'' instead, then $e_p=\sigma'(p)$. In fact $\sigma'$ is trivial, by Lemma \ref{trivchi}. On the other hand, $\epsilon_p=1$, since $p\nmid\frac{D}{d_-}$. Hence $e_p=\epsilon_p$.

Next consider the case that $\pi^f_p$ is $\text{Vb}^G$, so it is $L(\nu^{1/2}\xi\,1_{B^{\times}},\nu^{-1/2}\sigma)$ in the notation of \cite[Proposition A.1]{Sc}, and is a quotient of the induced representation $\nu^{1/2}\xi\,1_{B^{\times}}\rtimes\nu^{-1/2}\sigma$. This is for characters $\xi,\,\sigma$ of $\Q_p^{\times}$, with $\xi$ non-trivial quadratic. Arguing as in the proof of Lemma \ref{trivchi} (considering also elements of $B_p^{\times}$ with unit norm), we see that $\xi$ and $\sigma$ are unramified. (Alternatively, in the next paragraph we can apply $p\mid\mid D$ and the `$N$' column of \cite[Table A.14]{RS}.) Consider $\mathfrak{h}:\GUone{B_p}\rightarrow \mathcal{V}$ a vector in the space of the induced representation $\nu^{1/2}\xi\,1_{B^{\times}}\rtimes\nu^{-1/2}\sigma$, mapping to a $K_p^-$-fixed vector in $L(\nu^{1/2}\xi\,1_{B^{\times}},\nu^{-1/2}\sigma)$, and normalised so that $\mathfrak{h}(\mathrm{id.})=1$. The eigenvalue of $\omega'_p$ is $\mathfrak{h}\left(\left(\begin{smallmatrix}\pi&0\\0&\pi\end{smallmatrix}\right)\right)$, where $\pi\in B_p$ with $\pi^2=-p$. Letting $a=\pi,\,\lambda=p$, this is
$$\mathfrak{h}\left(\left(\begin{smallmatrix}a&0\\0&\lambda\,\overline{a}^{-1}\end{smallmatrix}\right)\right)=\nu^{1/2}\xi(a\overline{a})\,\nu^{-1/2}\sigma(\lambda)=\xi(p)\sigma(p)=-\sigma(p),$$
where $\xi(p)=-1$ because $\xi$ is non-trivial, quadratic and unramified. This computation is of exactly the same type that may be used to prove that $\omega_p$ acts as $(\chi\sigma)(p)$ in the case $\text{IIa}^G$. In the case $\text{Vb}^G$ we have simply substituted $\nu^{1/2}\xi$ for $\chi$ and $\nu^{-1/2}\sigma$ for $\sigma$.

There is a corresponding representation of $\GSp_2(\Q_p)$ of type Vb, which is $L(\nu^{1/2}\xi\mathrm{St}_{\GL_2},\nu^{-1/2}\sigma)$, and by \cite[Table A.12]{RS}, the eigenvalue of $W_p$ on a $K^+_p$-fixed vector is $\sigma(p)$. Since this local representation is of Saito-Kurokawa type (for the same newform $g$), this is $\epsilon_p$, by \cite[Proposition 5.5.8(i)]{RS}. Hence $e_p=-\sigma(p)=-\epsilon_p$, as required.

Finally, suppose that $\pi^f_p$ is $\text{VIc}^G$, which is $\nu^{1/2}\,1_{B^{\times}}\rtimes\nu^{-1/2}\sigma$, with $\sigma$ unramified. Then similarly $e_p=\sigma(p)$. There is a corresponding representation of $\GSp_2(\Q_p)$ of type VIc, which is $L(\nu^{1/2}\mathrm{St}_{\GL_2},\nu^{-1/2}\sigma)$, and by \cite[Table A.12]{RS}, the eigenvalue of $W_p$ on a $K^+_p$-fixed vector is $-\sigma(p)$. Since this local representation is of Saito-Kurokawa type (for the same newform $g$), this is $\epsilon_p$, by \cite[Proposition 5.5.8(i)]{RS}. Hence $e_p=\sigma(p)=-\epsilon_p$, as required.
\item The first part, in the case that $p\mid D^+$, follows from the
  fact that the $L$-parameter of $\pi^{\iota(g,h)}_p$ is a kind of
  direct sum of those of $\pi^g_p$ and $\pi^h_p$ (representations of
  $\GL_2(\Q_p)$), from the behaviour of $\epsilon$-factors under the
  local Langlands correspondence, and the relation between
  Atkin-Lehner eigenvalues and $\epsilon$-factors, for generic
  representations of $\GSp_2(\Q_p)$, which is \cite[Corollary
  7.5.5]{RS}. In the case $p\mid D^-$, the representation
  $\pi^{\iota(g,h)}_p$ of $\GUone{\Q_p}$ is of type $\text{IIa}^G$ (as
    in the proof of \cite[Proposition 12.3]{RW}). The corresponding
    representation of $\GSp_2(\Q_p)$ of type $\text{IIa}$ has the same
    $L$-parameter (cf. \cite[Table 1]{MR2887605}), and we may argue as
    in the proof of (1) to obtain the sign-change in the Atkin-Lehner
    eigenvalue.

The second part follows from the fact that (being a bit sloppy over the distinction between different Atkin-Lehner involutions at different levels), Atkin-Lehner involutions commute with $\eta$ and intertwine $\theta$ and $\theta'$ with each other. (This is in \cite[\S 4]{MR2208781} or \cite[\S 3.2]{RS}.) That none of $\eta, \theta+\theta'$ or $\theta-\theta'$ will produce $0$ follows from the linear independence in \cite[Theorem 7.5.6]{RS}. (See also following \cite[(5.49)]{RS}.)
 \end{enumerate}
 \end{proof}
 \begin{remar}\label{whichd} Returning to the decomposition $$M_{a,b}(\widehat{K}(D)) \simeq \bigoplus_{\stackrel{d \mid D}{\gcd(d,D/d)=1}} \widetilde{M}_{a,b}(\widehat{K(\Lattice)}^+)^{\theta_d},$$ given an eigenform $f\in M_{a,b}(\widehat{K}(D))$ as in the theorem, we want to know for which $d$ the corresponding $\tilde{f}$ lies in $\widetilde{M}_{a,b}(\widehat{K(\Lattice)}^+)^{\theta_d}$. We already know that this is characterised by $p\mid d\iff e_p=-1$, and the theorem, which tells us the $e_p$ for all primes $p\mid D$, thus determines $d$.
 \end{remar}
 \begin{exo}\label{RT61} (1) is illustrated by \cite[Example 9]{RT}, where $D=D^-=61$. Here there is $F\in S_{3,0}(P(61))_G$ with $w_{61}=-1$ (as in \cite[\S 8, Example 1]{MR3315514}), while for the associated $f\in M_{0,0}(\widehat{K}(61))$, $e_{61}=+1$, hence for $\tilde{f}\in\widetilde{M}_{0,0}(\widehat{K(61)}^+)$, $d=1$.
 \end{exo}
 \begin{remar}\label{oldF} The second part of (3) applies equally to the production of oldforms from newforms in (1). For the linear independence, apply \cite[Theorem 7.5.8]{RS}, that $\pi^f_p$ is generic because it is tempered. For $D^-=61$, $D^+=5$, oldforms with both signs for $e_5$, both coming from the newform in the previous remark, are $D_2$ and $E_1$ in \cite[Table 3]{RT}. As expected, $d=1$ when $e_5=+1$ (for $D_2$), while $d=5$ when $e_5=-1$ (for $E_1$).
 \end{remar}
\begin{exo} For the squarefree $D=5\cdot 61$, (2) is illustrated by $A_3$, $C_4$, $D_8$ and $G_2$ in \cite[Table 3]{RT}, while (3) is illustrated by $D_5$ and $F_1$ in \cite[Table 3]{RT}. In reading that table, beware that the Atkin-Lehner eigenvalues are not in the column ``A-L'', but they can be deduced from the values of $d$.
 \end{exo}
 
 \section{Applications to congruences}
 \label{section:congruences}
 \begin{thm} Let $F\in S_{k,j}(P(N))_G$ be a new Hecke eigenform, with $k\geq 3$. For any prime number $\ell$, there exists a $4$-dimensional $\Qbar_{\ell}$-vector space $V$, a continuous representation
 $$\rho_F:\Gal(\Qbar/\Q)\rightarrow\Aut_{\Qbar_{\ell}}(V),$$
 and a Galois-equivariant symplectic pairing $$\langle\cdot,\cdot\rangle: V\times V\rightarrow \Qbar_{\ell}(3-2k-j),$$
 such that for each prime $p\nmid \ell N$ it is unramified, with 
 $\det(I-\rho_F(\Frob_p^{-1})p^{-s})$ the reciprocal of the Euler factor at $p$ in the spin $L$-function of $F$. In particular, for $p\nmid\ell N$, $\tr(\rho_F(\Frob_p^{-1}))$ is the Hecke eigenvalue $\lambda_F(p)$ of the Hecke operator $T(p)$ on $F$.
 
The Hodge-Tate weights of $\rho_F|_{\Gal(\Qbar_{\ell}/\Q_{\ell})}$ are $0,k-2, j+k-1$ and $j+2k-3$. If $\ell\nmid N$ then $\rho_F|_{\Gal(\Qbar_{\ell}/\Q_{\ell})}$ is crystalline, and the Artin conductor of $\rho_F$ is $N$.
 \end{thm}
 This comes from the work of many mathematicians, and is summarised in \cite[Theorem 3.1]{M}. (The part about the conductor follows from the local-global compatibility.) It is part of a more general theorem, about cohomological automorphic representations of $\GSp_2(\A_F)$, with $F$ a totally real field. This uses a lifting to $\GL_4(\A_F)$, and the $\ell$-adic cohomology of Shimura varieties for unitary groups. But in the case $F=\Q$, most of it had been proved by Weissauer \cite[Theorem I]{MR2234860}, \cite{MR2530981}, using the $\ell$-adic cohomology of Siegel modular three-folds.
 
 It is expected that $\rho_F$ is irreducible. If we choose a $\Gal(\Qbar/\Q)$-invariant $\Zbar_{\ell}$-lattice then reduce modulo the maximal ideal, we get a representation $\rhobar_F$ of $\Gal(\Qbar/\Q)$ on a $4$-dimensional $\Fbar_{\ell}$-vector space.
For some $\ell$ it may be reducible, in which case it depends on the choice of invariant lattice, but its irreducible composition factors are independent of the choice.

If $\theta$ on $W$ is a composition factor of $\rhobar_F$ then so is $\Hom_{\Fbar_{\ell}[\Gal(\Qbar/\Q)]}(W,\Fbar_{\ell}(3-2k-j))$, i.e. $\theta^*(3-2k-j)$. If $\ell>j+2k-2$  and $\ell\nmid N$ then the Hodge-Tate weights can be detected on $\rhobar_F$, via its associated Fontaine-Lafaille module \cite{MR707328}. They come in twisted dual pairs $\{0, j+2k-3\}$ and $\{k-2, j+k-1\}$. It follows that if  $N$ is square-free, then any $1$-dimensional composition factor must be unramified away from $\ell$, to avoid it and its twisted dual partner (which is different) contributing a square factor to the conductor.

Therefore, for $\ell>j+2k-2$, $\ell\nmid N$ and $N$ square-free, the only possible pairs of $1$-dimensional composition factors of $\rhobar_F$ are $\{\Fbar_{\ell}, \Fbar_{\ell}(3-2k-j)\}$ and $\{\Fbar_{\ell}(2-k), \Fbar_{\ell}(1-k-j)\}$. 
Hence if $\rhobar_F$ is reducible, the possibilities for its composition factors are as follows.
\begin{enumerate}
\item $\{\rhobar_g, \Fbar_{\ell}(2-k), \Fbar_{\ell}(1-k-j)\}$, where $\rhobar_g$ is a $2$-dimensional representation attached to a cuspidal Hecke eigenform for $\Gamma_0(N)$, of weight $j+2k-2$. 
\item $\{\rhobar_g(2-k), \Fbar_{\ell}, \Fbar_{\ell}(3-2k-j)\}$, where $\rhobar_g$ is a $2$-dimensional representation attached to a cuspidal Hecke eigenform for $\Gamma_0(N)$, of weight $j+2$. 
\item $\{\rhobar_g, \rhobar_h(2-k)\}$, $2$-dimensional representations attached to cuspidal Hecke eigenforms $g$ of weight $j+2k-2$, $h$ of weight $j+2$, of levels $\Gamma_0(M)$ and $\Gamma_0(N/M)$ for some $M\mid N$.
\item $\{\Fbar_{\ell}, \Fbar_{\ell}(2-k), \Fbar_{\ell}(1-k-j), \Fbar_{\ell}(3-2k-j)\}$.
\end{enumerate}
We have used here the theorem of Khare and Wintenberger \cite{MR2551763, MR2551764} and Kisin \cite{MR2551765}, that an odd, irreducible representation of $\Gal(\Qbar/\Q)$ with coefficients in $\Fbar_{\ell}$ is modular, proved also by Dieulefait in the case that $\ell N$ is odd \cite{MR2959671}. These reducibilities translate into congruences of Hecke eigenvalues as follows, for all primes $p\nmid \ell N$. The instances we prove are actually for all $p\nmid N$.
\begin{enumerate}
\item $\lambda_F(p)\equiv a_p(g)+p^{k-2}+p^{j+k-1}\pmod{\lambda}$, with $g$ a cuspidal Hecke eigenform for $\Gamma_0(N)$, of weight $j+2k-2$.
\item $\lambda_F(p)\equiv p^{k-2}a_p(g)+1+p^{j+2k-3}\pmod{\lambda}$, with $g$ a cuspidal Hecke eigenform for $\Gamma_0(N)$, of weight $j+2$.
\item $\lambda_F(p)\equiv a_p(g)+p^{k-2}a_p(h)$, with cuspidal Hecke eigenforms $g$ of weight $j+2k-2$, $h$ of weight $j+2$, of levels $\Gamma_0(M)$ and $\Gamma_0(N/M)$ for some $M\mid N$.
\item $\lambda_F(p)\equiv 1+p^{k-2}+p^{j+k-1}+p^{j+2k-3}\pmod{\lambda}$.
\end{enumerate}
Here, $\lambda$ is a divisor of $\ell$ in a sufficiently large extension of $\Q$. Unless the conductor of $\rhobar_F$ happens to be a proper divisor of the conductor of $\rho_F$, $g$ in (1) and (2), and $g, h$ in (3), will be newforms. This will certainly be the case in the instances of such congruences that we prove, which will be of types (1) and (2).

The strategy for proving instances of such congruences is as follows. Assuming $N$ is exactly divisible by at least one prime, we choose a square-free divisor $D^-$ of $N$ with $\omega(D^-)$ odd, and apply Theorem \ref{thm:Rainer1} for suitable $D=D^-D^+$ with $\gcd(D^-,D^+)=1$ and $N\mid D$. So we take $F$ (or some oldform derived from $F$ if $N\neq D$) in $S_{k,j}^{\text{$D^-$-new}}(P(D))_G$, and corresponding $f\in M_{j+k-3,k-3}(\widehat{K}(D))_G$. For $p\nmid D$, the Hecke eigenvalue $\lambda_F(p)$, the left-hand-side of the congruence, is the same as the eigenvalue of $T(p)$ on $f$. 

The aim is  to arrange for there to be another Hecke eigenform $f_1\in M_{j+k-3,k-3}(\widehat{K}(D))$ such that the right-hand-side is the eigenvalue of $T(p)$ on $f_1$, or is at least congruent to it modulo $\lambda$. If we can then observe a congruence $\pmod{\lambda}$ between the vectors $f$ and $f_1$ inside 
$M_{j+k-3,k-3}(\widehat{K}(D))$ (with respect to the natural integral structure on the coefficient module $W_{j+k-3, k-3}$) then the desired congruence of Hecke eigenvalues follows, at least for $p\nmid D$. (In fact this gives us a congruence of Hecke eigenvalues also for $T_1(p^2)$, not just for $T(p)$.)

We further arrange for $f$ and $f_1$ to have matching Atkin-Lehner eigenvalues for $p\mid D$. Then in practice we actually prove the congruence between corresponding $\tilde{f}$ and $\tilde{f_1}$ in $\widetilde{M}_{j+k-3,k-3}(\widehat{K(\Lattice)}^+)^{\theta_d}$, for some suitable $d\mid D$ with $\gcd(d,D/d)=1$, determined by Theorem \ref{ALsignchange}. Actually, except in the first following section, we shall see some interesting departures from the simple strategy outlined above, with the involvement of additional eigenforms in $\widetilde{M}_{j+k-3,k-3}(\widehat{K(\Lattice)}^+)^{\theta_d}$, or even of two spaces for different genera of lattices.

 \subsection{Congruences between forms of Saito-Kurokawa and general types}
 First we look at case (1), in the sub-case $j=0$, with a newform $g\in S_{2k-2}(\Gamma_0(N))$. If the sign in the functional equation of $L(g,s)$ is $(-1)^{\omega(N)}$, and $q$ is an auxiliary prime when $\omega(N)$ is even, then for $p\nmid N$ (or $p\nmid Nq$), $a_p(g)+p^{k-2}+p^{k-1}$ is the Hecke eigenvalue of $T(p)$ on the Saito-Kurokawa lift $$\hat{g}=\begin{cases}\iota_1(g) & \text{if $\omega(N)$ is odd};\\\iota_q(g) & \text{if $\omega(N)$ is even}.\end{cases}$$ In Proposition \ref{prop:saito-kurokawa}, $D=D^-=N$ or $Nq$, and $d_-=1$ or $q$. We seek then a congruence between $\tilde{f}$ and $\tilde{\hat{g}}$ in $M_{k-3,k-3}(\widehat{K(\Lattice)}^+)^{\theta_d}$, for suitable $d$, where $\tilde{f}$ comes from a newform $F\in S_k(P(N))_G:=S_{k,0}(P(N))_G$.
 \subsubsection{Example with $N=61$, $k=3$, $\ell=43$}\label{61and43}
 The space $S_4(\Gamma_0(61))$ (\lmfdbform{61}{4}{a}) is
 $15$-dimensional, spanned by a newform $g$ (\lmfdbfform{61}{4}{a}{a})
 with Hecke eigenvalues in a number field $E$ of degree $6$, and a
 newform whose coefficient field has degree $9$. The sign in the
 functional equation of $L(g,s)$ is $\epsilon_{61}=-1$. Consequently
 there exists a $6$-dimensional subspace of Saito-Kurokawa lifts, of
 $g$ and its conjugates, inside $S_3(P(61))$. In \cite[\S 8, Example
 1]{MR3315514}, Poor and Yuen refer to this as the subspace of
 Gritsenko lifts of associated Jacobi forms. They show that
 $S_3(P(61))$ is $7$-dimensional, with $S_3(P(61))_G$ spanned by a
 Hecke eigenform $F$ with rational Fourier coefficients and
 Atkin-Lehner eigenvalue $w_{61}=-1$. They also prove (with
 appropriate scaling) a mod $43$ congruence of Fourier coefficients
 between $F$ and some Gritsenko lift (not a Hecke eigenform). It is
 easy to deduce from this a congruence of Hecke eigenvalues (or even
 of Fourier coefficients) between $F$ and the Saito-Kurokawa lift
 $\mathrm{SK}(g)$, modulo some prime divisor $\lambda$ of $43$ in
 $E$, as in \cite[Example 5.7]{DGlas}. It is (for all primes $p\neq 61$)
 $$\lambda_F(p)\equiv a_p(g)+p+p^2\pmod{\lambda}.$$
 
We can recover this as a congruence of Hecke eigenvalues between associated $f$ and $\hat{g}$ in $M_{0,0}(\widehat{K}(61))$, with $D=D^-=61, d_-=1$. (Note that, unlike the method of Poor and Yuen, ours does not lead to a congruence of Fourier coefficients.) By Theorem \ref{ALsignchange}, $\omega_{61}$ has eigenvalue $e_{61}=1$ on both $f$ and $\hat{g}$. So both $\tilde{f}$ and $\tilde{\hat{g}}$ find themselves inside $\widetilde{M}_{0,0}(\widehat{K(\Lattice)}^+)^{\theta_1}$ (cf. Remark \ref{whichd}). In fact, as noted in Example \ref{RT61}, the space $\widetilde{M}_{0,0}(\widehat{K(\Lattice)}^+)^{\theta_1}$ was computed in \cite[Example 9]{RT}. Computing with integral coefficients, the congruence between the suitably scaled eigenvectors $\tilde{f}$ and $\tilde{\hat{g}}$ may be observed directly.

We computed in Sage \cite{sagemath}, using the package \texttt{quinary\_module\_l.sage}, which may be found in \cite{rama_quinary}. The quadratic form $q$ is associated to a special lattice of determinant $2D$, with $D=61$. All the quinary forms used in this article, like those tabulated in \url{http://www.cmat.edu.uy/cnt/omf5/}, were obtained via a box search.

\begin{verbatim}
sage: q = QuadraticForm(ZZ, 5, [1, 0, 0, 1, 1, 1, 0, 1, 0, 1, 0, 0, 
1, 0, 8])
sage: qmod = quinary_module(q)
sage: T2 = qmod.Tp_d(2, 1)
sage: T2
[7 4 4 0 0 0 0 0]
[1 4 3 3 3 1 0 0]
[1 3 3 0 0 0 2 6]
[0 2 0 5 0 2 2 4]
[0 6 0 0 1 0 4 4]
[0 1 0 3 0 9 0 2]
[0 0 4 6 4 0 1 0]
[0 0 3 3 1 1 0 7]
sage: T2.fcp()
(x - 15) * (x + 7) * (x^6 - 29*x^5 + 322*x^4 - 1714*x^3 + 4471*x^2 - 
5205*x + 2026)
sage: v61 = (T2 + 7).right_kernel().0
sage: v61
(0, 6, -6, -4, -12, 0, 12, 3)
sage: K.<a> = NumberField(T2.fcp()[2][0])
sage: vSK = (T2 - a).right_kernel().0
sage: vSK*= denominator(vSK)
sage: I = K.ideal([43, a + 7])
sage: R = K.residual_field(I)
sage: (v61 - 4*vSK).change_ring(R) == 0
True
\end{verbatim}
Note that we are able to use the single Hecke operator $T(2)$ to decompose the space into simple Hecke submodules.

Generalising a construction of Ribet, as in \cite[\S 11]{MR4301048}, the invariant $\Zbar_{43}$-lattice in $V$ may be chosen in such a way that inside $\rhobar_F$ we get a non-trivial extension of $\Fbar_{43}(-1)$ by $\rhobar_g$, hence of $\Fbar_{43}$ by $\rhobar_g(1)$. This leads to a non-zero element in a certain Selmer group, which by the Bloch-Kato conjecture should lead to divisibility by $\lambda$ of a suitably normalised algebraic part $L^N_{\mathrm{alg}}(3, g)$ (Euler factors at primes dividing $N$ omitted). Note that $3$ is paired with $1$ by the functional equation with respect to $s$ and $4-s$. More generally, $j+k$ is 
paired with $k-2$. Using the command LRatio in the Magma computer package \cite{
MR1484478}, one readily checks that in fact $43\mid \Nm_{E/\Q}(L_{\mathrm{alg}}(
3, g))$.

A general theorem of Brown and Li, Corollary 6.14 in \cite{MR4301048}, proves a mod $\lambda$ congruence of Hecke eigenvalues between $\mathrm{SK}(g)$ ($g\in S_{2k-2}(\Gamma_0(N)^-$ a newform) and a Hecke eigenform $F\in S_k(P(N))_G$, from divisibility by $\lambda$ of $L_{\mathrm{alg}}(k, g)$, under various conditions. These include that $k\geq 6$, so it does not apply here.

\subsubsection{Example with $N=89$, $k=3$, $\ell=29$}\label{89and29}

As in the previous example, we can prove a congruence between the Hecke eigenvalues of the classical modular form $g$ with label \lmfdbform{89}{4}{b}
and a paramodular form in $S_3(P(89))_G$. More precisely, the space $S_3(P(89))$ decomposes as the sum of $4$ eigenspaces, $3$ of degree $1$ and $1$ of
degree $6$. This can be proved as before using the Hecke operator $T(2)$. Because the degree of $E$, the coefficient field of $g$, is $6$ we conclude
that the eigenspace of degree $6$ in $S_3(P(89))$ must correspond to a Saito-Kurokawa lift of $g$. Of the other $3$ eigenspaces, two correspond to Saito-Kurokawa lifts, which we can identify by looking at their eigenvalues. The third one, spanned by $F$, must be in $S_3(P(89))_G$. Using the same argument as in the previous example, we have proved the following.
\begin{thm}
The following congruence holds for all primes $p\neq 89$:
\[\lambda_F(p)\equiv a_p(g) + p + p^2\pmod{\lambda},\]
where $\lambda$ is a prime divisor of $29$ in $E$.
\end{thm}
 
 \subsection{Harder's conjecture for paramodular level: examples of Fretwell}
Still in case (1), we suppose now that $j>0$. Then the right hand side $a_p(g)+p^{k-2}+p^{j+k-1}$ is no longer the Hecke eigenvalue of $T(p)$ on some element of $S_{k,j}(P(N))$. Still, one may conjecture that if $\ell>j+2k-2$ and $\lambda$ divides $L_{\mathrm{alg}}(j+k, g)$ then there exists a Hecke eigenform $F\in S_{k,j}(P(N))_G$ satisfying the congruence
$$\lambda_F(p)\equiv a_p(g)+p^{k-2}+p^{j+k-1}\pmod{\lambda}.$$

This conjecture was made by Harder in the case $N=1$ \cite{MR2409680}, and an instance $(k,j,\ell)=(10,4,41)$ was proved by Chenevier and Lannes \cite{MR3929692}, using algebraic modular forms, with constant coefficients, for orthogonal groups of even unimodular lattices of rank $24$.

For $N=2,3,5,7$, Fretwell \cite{MR3803970} found experimental evidence for several instances of such congruences, in each case checking it for a few small values of $p$.
He computed Hecke eigenvalues by computing traces of Hecke operators on $1$-dimensional spaces of algebraic modular forms for $\GUtwo{B}$, with $B$ a definite quaternion algebra over $\Q$, ramified at $N$. In Theorem \ref{mod7} below, we prove an instance of a congruence of the same type, but for larger $N$. Three further examples we have proved, noted in \S \ref{further}, include one of Fretwell's. 

\subsubsection{Example with $N=19$, $k=3$, $j=2$, $\ell=7$} Notice
that $\ell>j+2k-2$, just. Consider the newform
$g\in S_6(\Gamma_0(19))$, with rational Hecke eigenvalues and
$\epsilon_{19}=-1$
(\lmfdbfform{19}{6}{a}{a})). Its Fourier expansion starts as
$$g=q-6q^2+4q^3+4q^4+54q^5-24q^6+\ldots.$$ Using the command LRatio in
the Magma computer package \cite{MR1484478}, one checks that
$7\mid L_{\mathrm{alg}}(5, g))$, so we expect a Hecke eigenform
$F\in S_{3,2}(P(19))_G$ with
\begin{equation}
  \label{eq:ex1}
\lambda_F(p)\equiv a_p(g)+p+p^4\pmod{7},  
\end{equation}
for all primes $p\nmid 7\cdot 19$.

\begin{thm}\label{mod7}
  The congruence (\ref{eq:ex1}) holds for all $p\neq 19$.
\end{thm}
\begin{proof}
We may view the right-hand-side as $a_p(g)+p(1+p^3)$, which looks like
the Hecke eigenvalue of a Yoshida lift of $g$ and an Eisenstein series
of weight $4$ and level $1$. Though there is no such thing, we can
replace the Eisenstein series by a cuspidal Hecke eigenform $h$ of
weight $4$ and prime level $q$, chosen to have the same Hecke
eigenvalues mod $7$. Such an $h$ exists as long as
$q^4\equiv 1\pmod{7}$. This is an instance of a general theorem on
congruences ``of local origin'', conjectured by Harder \cite{HSO} and
proved independently in \cite[Theorem 1.1]{MR3227346} and by Billerey
and Menares \cite{MR3512875}. In the case at hand, the smallest $q$
that works is $q=13$, and we can find the form $h$
(\lmfdbfform{13}{4}{a}{a}) in the LMFDB \cite{lmfdb}, with
$\epsilon_{13}=-1$, $$h=q-5q^2-7q^3+17q^4-7q^5+35q^6+\ldots.$$

By Proposition \ref{prop:yoshida}, with $a=2, b=0$, $d_-=D^-=13$, $c_+=D^+=19$, we have $f_1=\iota(h, g)\in M_{2,0}(\widehat{K}(13\cdot 19))$. By Theorem \ref{ALsignchange}, $\omega_{13}$ and $W_{19}$ have eigenvalues $e_{13}=1, e_{19}=-1$ on $f_1$. Hence the corresponding $\tilde{f_1}$ lives in $\widetilde{M}_{2,0}(\widehat{K(\Lattice_1)}^+)^{\theta_{19}}$, where $\Lattice_1$ is the lattice associated to $D^- = 13$ and $D^+ = 19$.

 With $D=D^-=19$, one finds that the space $\widetilde{M}_{2,0}(\widehat{K(\Lattice)}^+)^{\theta_{1}}$ (with a different $\Lattice$ for $D=19$) is $1$-dimensional, spanned by an eigenform corresponding by Theorem \ref{thm:Rainer1} to our target $F\in S_{3,2}(P(19))_G$, with $w_{19}=-1$ by Theorem \ref{ALsignchange}. Following Remark \ref{oldF} (and switching to $D^- = 19$, $D^+ = 13$), we can manufacture an associated oldform $f\in M_{2,0}(\widehat{K}(19\cdot 13))$ with $e_{19}=1, e_{13}=-1$, whose associated $\tilde{f}$ is in $\widetilde{M}_{2,0}(\widehat{K(\Lattice_2)}^+)^{\theta_{13}}$, where $\Lattice_2$ corresponds to $D^- = 19$, $D^+ = 13$.

The space
$\widetilde{M}_{2,0}(\widehat{K(\Lattice_1)}^+)^{\theta_{19}}$ (working with coefficients in a $\Q$-vector space)
decomposes as $A_1\oplus A_2\oplus A_3$ where $A_1$ corresponds to the
Yoshida lift $f_1$, $A_2$ to the Yoshida lift of the forms \lmfdbfform{13}{4}{a}{a}
and \lmfdbfform{19}{6}{a}{d}. Finally, $A_3$ corresponds to a paramodular newform of
level $13\cdot 19$ for $k=3$, $j=2$.

The space
$\widetilde{M}_{2,0}(\widehat{K(\Lattice_2)}^+)^{\theta_{13}}$
decomposes as $B_1\oplus B_2\oplus B_3$ where $B_1$ corresponds to the
oldform $\tilde{f}$, $B_2$ to the Yoshida lift of the forms
\lmfdbfform{19}{4}{a}{a} and \lmfdbfform{13}{6}{a}{b}. Lastly, $B_3$
corresponds to a paramodular newform of level $13\cdot 19$ for $k=3$, $j=2$, so it is isomorphic to $A_3$.

If we work with coefficients in $\Z$-modules rather than $\Q$-vector spaces, we do not get direct sum decompositions. The space $A_1$ is included in $A_3$ modulo $7$, and $B_1$ is included in $B_3$ modulo $7$. In fact, the isomorphism of $A_3$ to $B_3$ sends $A_1$ to $B_1$ modulo $7$. We conclude that the eigenvalues of $\tilde{f}$ and $f_1$ must be the same modulo $7$. This would be for all $p\nmid 19\cdot 13$, but we may check the congruence for $p=13$ by hand, or by choosing a different $q$, as remarked below.
\end{proof}

The computations for this theorem were done using a package similar to that used in the previous example, but implemented in Pari/GP
\cite{PARI2}, where we implemented the representations in general. This
also can be found in \cite{rama_quinary}.

The next three primes $q$ such that $q^4\equiv 1\pmod{7}$ are
$q=29, 41$ and $43$. We checked that the same congruence can be proved
in the same fashion using any of these $q$ in place of $q=13$. Note in
particular that in all such cases there exists a congruence modulo $7$
between the paramodular form $F$ and a paramodular newform of level
$19\cdot q$. This naturally raises the following question.

\begin{ques}[level-raising]\label{lraise}
  Consider a new Hecke eigenform $F\in S_{k,j}(P(N))_G$ such that (for
  some $\ell>j+2k-2$) $\rhobar_F$ is reducible of type (1). Is it true
  that for any prime $q$ such that $q^{j+2}\equiv 1\pmod{\ell}$, there
  exists a new Hecke eigenform $H\in S_{k,j}(P(qN))_G$ such that
  $\rhobar_H$ has the same composition factors as $\rhobar_F$?
\end{ques}
Note that $\rhobar_F$ is reducible, so although we call this ``level-raising'', it should perhaps be viewed as more closely analogous to \cite{MR3512875}, \cite{MR3227346} or \cite{MR3917209} than to well-known results of Diamond and Ribet on level-raising for residually irreducible representations \cite{MR804706}, \cite{MR804706}.
\subsubsection{Further examples}\label{further} With the above approach we only succeeded in proving one of the examples of Fretwell. The reason is that our method
is very efficient for small values of $j$, in which case the dimension of
$W_{j+k-3, k-3}$ (especially for small $k$) can be manageably small,
cf. Theorem \ref{Slambda}.

The example we could handle is when $N=5$, $k=7$, $j=2$, and $\ell =
61$. Let $g\in S_{14}(\Gamma_0(5))$ be the newform with LMFDB label
\lmfdbfform{5}{14}{a}{b} and coefficient field $E$ of degree $3$, and
$F\in S_{7, 2}(P(5))_G$.
\begin{thm}
  The following congruence holds, for all primes $p\neq 5$:
\[\lambda_F(p) \equiv a_p(g) + p^3 + p^8\pmod{\lambda}
\]
where $\lambda$ is a prime in $E$ that divides $61$.  
\end{thm}
\begin{proof}
  Take the form $h\in S_{4}(\Gamma_0(11))$ of label
  \lmfdbfform{11}{4}{a}{a} to construct the Yoshida lift, and proceed
  as in the proof of Theorem \ref{mod7}. Note that $11^4\equiv 1\pmod{61}$.
\end{proof}

We also proved the following additional two examples.
\begin{thm}
  In the previous notation, let $N = 42$, $k = 3$, $j=2$, and
  $\ell = 13$. Let $g$ be the newform \lmfdbfform{42}{6}{a}{d} and $F$ the
  paramodular form in $S_{3, 2}(P(42))$. Then, for all primes $p\nmid 42$,
\[\lambda_F(p) \equiv a_p(g) + p + p^4\pmod{13}.
\]
\end{thm}
\begin{proof}
  Proceed as in the previous examples, using the form
  \lmfdbfform{5}{4}{a}{a} to construct the Yoshida lift. The only
  difference with the other examples is that since the primes $2,3,5$ and $7$ all divide $qN$, we have to decompose the space using $T(11)$, whereas in earlier examples we always used $T(2)$.
\end{proof}

\begin{thm}
  In the previous notation, let $N = 13$, $k = 3$, $j=4$, and
  $\ell = 11$. Let $g$ be the newform \lmfdbfform{13}{8}{a}{a} and $F$ the
  paramodular form in $S_{3, 4}(P(13))_G$. Then, for all primes $p\neq 13$, 
\[\lambda_F(p) \equiv a_p(g) + p + p^6\pmod{11}.
\]
\end{thm}
\begin{proof}
Take the form \lmfdbfform{23}{6}{a}{a} to construct the Yoshida lift.
\end{proof}

 \subsection{Proof of a congruence of Buzzard and Golyshev}
 We turn now to case (2). We already saw, in \S \ref{61and43}, a
 congruence, modulo a divisor of $43$, involving $F\in S_3(P(61))_G$
 and $g\in S_4(\Gamma_0(61))$. Around the end of 2010, V. Golyshev
 conjectured the existence of a second congruence for $F$, beyond the
 one involving $43$, then K. Buzzard found it experimentally, having
 realised the possibility of it involving weight $2$ rather than
 weight $4$, and computations of A. Mellit provided further
 support. The congruence is
 \begin{equation}
   \label{eq:buzzard}
\lambda_F(p)\equiv 1+p^3+pa_p(g)\pmod{\lambda},   
 \end{equation}
 where $g\in S_2(\Gamma_0(61))$ (\lmfdbfform{61}{2}{a}{b}) is a
 newform with cubic coefficient field $E$, $\epsilon_{61}=-1$, and
 $\lambda$ is a prime divisor of $19$ in $E$.

 If this is true then inside $\rhobar_F$ we get a non-trivial extension of $\Fbar_{19}$ by $\rhobar_g(-1)$ ($\rhobar_g(2-k)$ in general), which is connected by the Bloch-Kato conjecture to $L(3,g)$ ($L(j+k)$ in general), but since this is a non-critical value (not in the range $1\leq s\leq (j+2)-1$) we cannot detect the factor computationally.

 \begin{thm}
   The congruence~(\ref{eq:buzzard}) holds, for all primes $p\neq 61$.
\label{thm:buzzard}
 \end{thm}
 \begin{proof}
 To prove the congruence, we interpret $1+p^3+pa_p(g)$ as congruent
 mod $\lambda$ to $a_p(h)+pa_p(g)$, where $h\in S_4(\Gamma_0(q))$ is a
 newform congruent mod $\lambda$ to the level $1$ Eisenstein series of
 weight $4$, with $q^4\equiv 1\pmod{19}$. The smallest $q$ we can use
 is $q=37$, and $h$ is \lmfdbfform{37}{4}{a}{a}, with
 $\epsilon_{37}=-1$. Since $h$ has coefficient field of degree $4$, we
 need to replace $E$ by its compositum with this field, and $\lambda$
 by a suitable divisor of the original.
 
By Proposition \ref{prop:yoshida}, with $a=0, b=0$, $d_-=D^-=61$, $c_-=D^+=37$, we have $f_1=\iota(g, h)\in M_{0,0}(\widehat{K}(61\cdot 37))$. By Theorem \ref{ALsignchange}, $\omega_{61}$ and $W_{37}$ have eigenvalues $e_{61}=+1, e_{37}=-1$ on $f_1$. Hence the corresponding $\tilde{f_1}$ lives in $\widetilde{M}_{0,0}(\widehat{K(\Lattice)}^+)^{\theta_{37}}$. Using Theorem \ref{ALsignchange} ($w_{61}=-1\implies e_{61}=+1$) and Remark \ref{oldF}, we may produce an oldform $f$ associated to $F$, with $\tilde{f}$ in the same $\widetilde{M}_{0,0}(\widehat{K(\Lattice)}^+)^{\theta_{37}}$.

If we compute the Hecke operator $T(2)$ restricted to
$\widetilde{M}_{0,0}(\widehat{K(\Lattice)}^+)^{\theta_{37}}$, its
characteristic polynomial factors as $(x+7)\cdot p_2(x)\cdot p_3(x)$,
where $p_2$ and $p_3$ are irreducible of degree $12$ and $211$
respectively. Working with $\Z$-coefficients, let $C_1$, $C_2$, $C_3$ be $\Z$-submodules of $\widetilde{M}_{0,0}(\widehat{K(\Lattice)}^+)^{\theta_{37}}$ killed by $T(2)+7$, $p_2(T(2))$ and $p_3(T(2))$, respectively. The space $C_1$ corresponds to $\tilde{f}$, $C_2$
to $\tilde{f}_1$ (and its Galois conjugacy class) and $C_3$ likewise to a paramodular newform of level $37\cdot 61$ for $k=3$, $j=0$.

The kernel of $T(2)+7$ on $\widetilde{M}_{0,0}(\widehat{K(\Lattice)}^+)^{\theta_{37}}\otimes\FF_{19}$  has dimension $2$. In $\widetilde{M}_{0,0}(\widehat{K(\Lattice)}^+)^{\theta_{37}}\otimes\Z_{19}$ we find 
four $T(2)$-eigenspaces with eigenvalues congruent to $-7$ modulo $19$, all rank-one and therefore common eigenspaces for all the $T(p)$ and $T_1(p^2)$ ($p\nmid 61\cdot 37$). The line $C_1\otimes\Z_{19}$ has eigenvalue $-7$, inside $C_2\otimes\Z_{19}$ we have a line with eigenvalue $-7+10\cdot19+8\cdot19^2+\cdots$, inside $C_3\otimes\Z_{19}$ eigenvalues $-7+15\cdot 19+2\cdot 19^2+\cdots$ and $-7+18\cdot19+10\cdot19^2+\cdots$. 

We find that when we take four eigenvectors spanning these eigenspaces, and reduce them mod $19$, say to $\{v_1, v_2, v_3, v_4\}$, which lie in the aforementioned $2$-dimensional kernel of $T(2)+7$ on $\widetilde{M}_{0,0}(\widehat{K(\Lattice)}^+)^{\theta_{37}}\otimes\FF_{19}$, no two of them is collinear. Now for any $p\nmid 61\cdot 37$, consider the eigenvalues $\mu_1, \mu_2, \mu_3, \mu_4\in\FF_{19}$ of $T(p)$ acting on $v_1, v_2, v_3, v_4$ respectively. If $\mu_1\neq \mu_2$ then $v_3$, being in neither the $\mu_1$-eigenspace nor the $\mu_2$-eigenspace, would not be an eigenvector. This kind of contradiction shows that all four vectors lie in a single $2$-dimensional simultaneous eigenspace for all the $T(p)$ and $T_1(p^2)$ ($p\nmid 61\cdot 37$). This implies the congruence we were to prove. (It may be checked by hand for the auxiliary prime $p=37$.)
 \end{proof}

\begin{remar} The fact that the simultaneous eigenspace is not $1$-dimensional may be viewed as a ``multiplicity-one failure'', analogous to that discovered by Ribet and Yoo for certain Eisenstein ideals at composite level \cite[Example 4.7, Remark 4.11]{MR3473658}.
\end{remar}

The above computations were also performed using the Pari/GP package. Note that  our paramodular form has reducible residual representation modulo $19$, and is congruent to a newform of paramodular level $61 \cdot q$ (with $q=37$). The next $q$ such that $q^4\equiv 1\pmod{19}$ is $q=113$. This is rather large to reprove the congruence. 

\subsubsection{Example for $N = 89$, $\ell = 5$}
As in the case of $N = 61$, we also have a ``second'' congruence for $N=89$, involving a modular form of weight $2$.
Let $g_1$, $g_2$ be the classical modular forms of level $89$, weight $2$ and sign $-1$ in their $L$-functions, with LMFDB labels \lmfdbfform{89}{2}{a}{b} and \lmfdbfform{89}{2}{a}{c} respectively.
Their coefficient fields have degree $1$ and $5$ respectively, and we denote by $E$ the second one.
It is easy to prove in Sage that
\[a_p(g_1)\equiv a_p(g_2)\pmod{\lambda}\]
where $\lambda$ is a prime dividing $5$ in $E$, using the class \texttt{CuspForms} and the method
\texttt{Hecke\_matrix} to compute the Hecke matrix at $2$ and prove the congruence for the corresponding
eigenspaces.

Using the same method as before we have proved the following.
\begin{thm}
The following congruences hold, for all primes $p\neq 89$:
\[\lambda_F(p)\equiv pa_p(g_1) + 1 + p^3\pmod{5};\]
\[\lambda_F(p)\equiv pa_p(g_2) + 1 + p^3\pmod{\lambda},\]
where $\lambda$ is a prime divisor of $5$ in $E$.
\end{thm}
\begin{proof}
Take $q = 7$, $d_- = 1$, $D^- = 89$ and $D^+ = 7$, and proceed as with the previous example using
the form $h\in S_4(\Gamma_0(7))$ of label \lmfdbfform{7}{4}{a}{a} to construct the Yoshida lift.
\end{proof}

In this case all primes $q\neq5$ satisfy $q^4\equiv1\pmod{5}$. We checked that there exists a congruence
modulo $5$ of the paramodular form $F$ of level $89$ and a paramodular newform of level $5\cdot q$ for $q = 2, 3, 7, 11$. So once again, the following seems a very natural question.
\begin{ques}\label{lraise2}
  Consider a new Hecke eigenform $F\in S_{k,j}(P(N))_G$ such that (for
  some $\ell>j+2k-2$) $\rhobar_F$ is reducible of type (2). Is it true
  that for any prime $q$ such that $q^{j+2k-2}\equiv 1\pmod{\ell}$,
  there exists a new Hecke eigenform $H\in S_{k,j}(P(qN))_G$ such that
  $\rhobar_H$ has the same composition factors as $\rhobar_F$?
\end{ques}

\bibliographystyle{alpha}
\bibliography{biblio}
\end{document}